\documentclass[12pt,reqno]{amsart}

\usepackage{lipsum}

\usepackage{graphicx}
\usepackage{array}
\usepackage{amssymb}
\usepackage{hyperref}
\usepackage{amsthm}
\usepackage{amsfonts}
\usepackage{amsmath}
\usepackage{bm}
\usepackage{mathrsfs}
\usepackage[all]{xy}
\usepackage{color}
\usepackage{extarrows}
\usepackage{subfigure}
\usepackage{tikz, tikz-cd}
\usepackage{enumerate}
\usepackage{anysize}
\usepackage{amscd}
\usepackage[letterpaper]{geometry}
\usepackage{geometry}
\usepackage{multirow}
\usepackage{array}
\usepackage{booktabs}

\newtheorem{theorem}{Theorem}[section]
\newtheorem{proposition}[theorem]{Proposition}
\newtheorem{assumption}{Assumption}[section]

\newtheorem{lemma}[theorem]{Lemma}

\newtheorem{corollary}[theorem]{Corollary}

\newtheorem{question}[theorem]{Question}

\theoremstyle{definition}
\newtheorem{definition}[theorem]{Definition}
\newtheorem{remark}[theorem]{Remark}
\newtheorem{example}[theorem]{Example}

\newtheorem{Thm}{Theorem}%[section]
\newtheorem{problem}{Problem}[section]

\numberwithin{equation}{section}
\numberwithin{figure}{section}
\numberwithin{table}{section}

\usetikzlibrary{decorations.pathreplacing,bending}

\tikzset{
    on each segment/.style={
        decorate,
        decoration={
            show path construction,
            moveto code={},
            lineto code={
                \path [#1]
                (\tikzinputsegmentfirst) -- (\tikzinputsegmentlast);
            },
            curveto code={
                \path [#1] (\tikzinputsegmentfirst)
                .. controls
                (\tikzinputsegmentsupporta) and (\tikzinputsegmentsupportb)
                ..
                (\tikzinputsegmentlast);
            },
            closepath code={
                \path [#1]
                (\tikzinputsegmentfirst) -- (\tikzinputsegmentlast);
            },
        },
    },
}

\tikzset{
    partial ellipse/.style args={#1:#2:#3}{
        insert path={+ (#1:#3) arc (#1:#2:#3)}
    }
}

\makeatletter
\renewcommand{\tocsection}[3]{%
 \indentlabel{\@ifnotempty{#2}{\bfseries\ignorespaces#1 #2\quad}}\bfseries#3}
\renewcommand{\tocsubsection}[3]{%
  \indentlabel{\@ifnotempty{#2}{\ignorespaces#1 #2\quad}}#3}
\newcommand\@dotsep{4.5}
\def\@tocline#1#2#3#4#5#6#7{\relax
  \ifnum #1>\c@tocdepth % then omit
  \else
    \par \addpenalty\@secpenalty\addvspace{#2}%
    \begingroup \hyphenpenalty\@M
    \@ifempty{#4}{%
      \@tempdima\csname r@tocindent\number#1\endcsname\relax
    }{%
      \@tempdima#4\relax
    }%
    \parindent\z@ \leftskip#3\relax \advance\leftskip\@tempdima\relax
    \rightskip\@pnumwidth plus1em \parfillskip-\@pnumwidth
    #5\leavevmode\hskip-\@tempdima{#6}\nobreak
    \leaders\hbox{$\m@th\mkern \@dotsep mu\hbox{.}\mkern \@dotsep mu$}\hfill
    \nobreak
    \hbox to\@pnumwidth{\@tocpagenum{\ifnum#1=1\bfseries\fi#7}}\par% <-- \bfseries for \section page
    \nobreak
    \endgroup
  \fi}
\AtBeginDocument{%
\expandafter\renewcommand\csname r@tocindent0\endcsname{0pt}
}
\def\l@subsection{\@tocline{2}{0pt}{2.5pc}{5pc}{}}
\makeatother

\newcommand{\lesl}{\leqslant}
\newcommand{\gesl}{\geqslant}

\newcommand{\fq}{\mathfrak{q}}
\newcommand{\bX}{\overline{X^{n+1}}}

\newcommand{\bx}{\bm{x}}
\newcommand{\tp}{\tilde{p}}
\newcommand{\tx}{\tilde{x}}
\newcommand{\ty}{\tilde{y}}

\newcommand{\dH}{\mathbb{H}}
\newcommand{\dN}{\mathbb{N}}
\newcommand{\dR}{\mathbb{R}}

\newcommand{\dT}{\mathbb{T}}
\newcommand{\dZ}{\mathbb{Z}}

\newcommand{\chg}{\check{g}}
\newcommand{\cN}{\mathcal{N}}

\newcommand{\bo}{\bm{0}}

\newcommand{\Mob}{\text{M\"ob}}

\newcommand{\Dw}{\bm{d}_{\mu}}

\newcommand{\sL}{\mathscr{L}^{(1)}}

\newcommand{\ND}{\mathcal{D}^{\perp}}

\newcommand{\NS}{\mathcal{S}^{\perp}}

\newcommand{\isom}{\mathfrak{i}\mathfrak{s}\mathfrak{o}\mathfrak{m}}

\newcommand{\fo}{\mathfrak{o}}

\newcommand{\fs}{\mathfrak{s}}
\newcommand{\ft}{\mathfrak{t}}
\newcommand{\wg}{\widehat{g}}
\newcommand{\bg}{\bar{g}}
\newcommand{\tg}{\tilde{g}}
\newcommand{\tir}{\tilde{r}}
\newcommand{\wh}{\widehat{h}}

\newcommand{\p}{\partial}

\newcommand{\cS}{\mathcal{S}}

\newcommand{\sF}{\mathscr{F}}
\newcommand{\sQ}{\mathscr{Q}}
\newcommand{\sT}{\mathscr{T}}
\newcommand{\sX}{\mathscr{X}}
\newcommand{\wsQ}{\widetilde{\mathscr{Q}}}
\newcommand{\wsT}{\widetilde{\mathscr{T}}}

\newcommand{\wx}{\widehat{\xi}}

\newcommand{\ul}{\underline{\lambda}}

\newcommand{\ur}{\hat{\rho}}

\newcommand{\ucX}{\widetilde{X^{n+1}}}

\newcommand{\nexp}{\Exp_{\ell_0}^{\perp}}

\newcommand{\nExp}{\Exp_{\ell_{\Gamma}}^{\perp}}
\newcommand{\nExpX}{\Exp_{\sigma_{\Gamma}}^{\perp}}

\newcommand{\ellG}{\ell_{\Gamma}}
\newcommand{\sG}{\sigma_{\Gamma}}
\newcommand{\br}{\bm{r}}

\newcommand{\LB}{\langle}
\newcommand{\RB}{\rangle}

\DeclareMathOperator{\diam}{diam}
\DeclareMathOperator{\dvol}{dvol}

\DeclareMathOperator{\Diff}{Diff}

\DeclareMathOperator{\Exp}{Exp}

\DeclareMathOperator{\Id}{Id}

\DeclareMathOperator{\Injrad}{Injrad}

\DeclareMathOperator{\Isom}{Isom}
\DeclareMathOperator{\Ker}{Ker}
\DeclareMathOperator{\length}{length}
\DeclareMathOperator{\Min}{Min}

\DeclareMathOperator{\pt}{pt}

\DeclareMathOperator{\rank}{rank}
\DeclareMathOperator{\Ric}{Ric}

\DeclareMathOperator{\Rm}{Rm}
\DeclareMathOperator{\SO}{SO}

\DeclareMathOperator{\Tr}{Tr}
\DeclareMathOperator{\Vol}{Vol}

\DeclareMathOperator{\MM}{M\text{\"o}b}

\allowdisplaybreaks[3]

\begin{document}

\title{On the Poincar\'e-Einstein manifolds with cylindrical conformal infinity}

\author{Sun-Yung Alice Chang} 
\address{Department of Mathematics, Princeton University, Princeton, NJ, 08544}
\email{syachang@math.princeton.edu}

\author{Paul Yang}
\address{Department of Mathematics, Princeton University, Princeton, NJ, 08544}
\email{yang@math.princeton.edu}

\author{Ruobing Zhang}
\address{Department of Mathematics, University of California, San Diego,  CA, 92093}
\email{ruz071@ucsd.edu}

   \thanks{The first named author is partially supported by Simons Foundation Travel Support for Mathematicians
AWD-1006781; The third named author is partially supported by NSF Grants DMS-2304818 and DMS-2550348.}
 
\begin{abstract}
	In this paper, we prove several rigidity and quantitative rigidity results for asymptotically hyperbolic Poincar\'e-Einstein manifolds whose conformal infinities are diffeomorphic to a cylinder $S^1 \times S^{n - 1}$. It is a basic fact that the Riemannian product $S^1 \times S^{n - 1}$ can bound, in addition to a complete hyperbolic metric on $S^1 \times D^n$, other Poincar\'e-Einstein metrics such as the AdS-Schwarzschild metrics on $D^2 \times S^{n - 1}$. 
	
	The main result shows that any Poincar\'e-Einstein filling of $S^1 \times S^{n - 1}$ must be hyperbolic if it is non-positively curved. As corollaries, the Poincar\'e-Einstein filling of $S^1 \times S^{n - 1}$ is unique when the length of circle factor is sufficiently large or the $L^2$-energy of the Weyl curvature is sufficiently small relative to the Yamabe constant of the conformal infinity. To prove the Weyl pinching rigidity, we established a new $\epsilon$-regularity for the Weyl curvature of a general class of Poincar\'e-Einstein manifolds with conformal infinity of positive Yamabe type, which includes non-compact and volume-collapsed families of Poincar\'e-Einstein spaces in all dimensions. 
	
\end{abstract}

\date{\today}

\maketitle

\tableofcontents

\setcounter{tocdepth}{1}

\section{Introduction}

\label{s:introduction}
 
 \subsection{Background and main results}
 
The geometry of asymptotically hyperbolic Poincar\'e-Einstein manifolds is a fundamental theme in the study of Einstein spaces through the conformal geometric approach. 
 It is also intimately related to the AdS/CFT correspondence in mathematical physics; we refer to the reader to \cite{Biquard-ADS/CFT} for a more detailed discussion on this topic. 
 
Roughly speaking, an asymptotically hyperbolic Poincar\'e-Einstein manifold is a complete Einstein manifold with negative scalar curvature which is asymptotically modeled on a locally hyperbolic space and  can be conformally deformed to a compact space with a smooth boundary called the conformal infinity (see Definition \ref{d:PE-manifold} for the precise notion).  
 The geometric side was initiated by Fefferman and Graham in \cite{FG-conformal-invariants}  and the aim was to study new conformal invariants by establishing a correspondence between the Riemannian geometry of the Einstein manifold and the conformal structure of the boundary. In the context of self-dual Poincar\'e-Einstein metrics, one can find existence and uniqueness results in \cite{LeBrun-SD, Pedersen, Hitchin}, et al.

Fundamental problems in this field include the existence, uniqueness, and classifications of Poincar\'e-Einstein filling  for a given conformal manifold.

 \begin{problem}
 	{\it Given a conformal manifold $(M^n, [\wg])$, does there exist a manifold $X^{n+1}$ and a complete Poincar\'e-Einstein metric $g$ on $X^{n+1}$ such that $\mathscr{C}(X^{n+1}, g) = (M^n,[\wg])$? Under what assumptions can one classify all the Poincar\'e-Einstein fillings $(X^{n+1}, g)$  for a given conformal infinity?}
 \end{problem}

 This problem is a fundamentally difficult one. Until now, both the existence and uniqueness of Poincar\'e-Einstein fillings are widely open in general.  Indeed, until very recently only the perturbation existence results via the implicit function theorem are available in the literature \cite {Graham-Lee}, \cite{Lee-fredholm}, et al.  In the recent work of Chang-Ge \cite{Chang-Ge-existence}, the existence results are extended to classes of metrics that have some {\it definite} deviation from the standard model metrics, where the conformal infinity is also allowed to have topology other than the sphere.

  On the other hand, there are some non-existence results. In the recent works of  \cite{Gursky-Han} and \cite{Gursky-Han-Stolz}, the authors applied the index theory of Dirac operators and discovered new geometric obstructions to the existence of Poincar\'e-Einstein metrics for a large class of conformal manifolds with positive Yamabe constant. This also suggests the existence problem requires a delicate formulation.

 At the moment, the uniqueness of Poincar\'e-Einstein fillings is only known when the conformal infinity is the sphere with the standard metric (see \cite{Min-Oo, Leung, Andersson-Dahl, Wang-mass, CH, Qing, Shi-Tian, BMQ, Biquard-continuation, DJ, HQS, LQS}) or one that is sufficiently close to
the conformal class of the standard metric of the sphere (see  \cite{CG, CGQ, CGJQ}). These later works are established as consequences of some compactness results of compactified Poincar\'e-Einstein metrics with a given compact family of metrics in the conformal infinity.

In the aforementioned developments, most of the work relies on a basic scenario in which the round sphere uniquely 
bounds a Poincar\'e-Einstein manifold that is isometric to the hyperbolic space $(\dH^{n + 1}, h)$, where $\sec_h \equiv - 1$. Note that, in this case,  
the conformal infinity $(S^n, [g_c])$ attains the maximum of the Yamabe invariant among all closed conformal manifolds. By the deep comparison result (see Theorem \ref{t:volume-comparison}), this maximality yields the desired hyperbolic rigidity.  
 However, when the underlying manifold of the conformal infinity is not a sphere, the rigidity problem must be 
 approached  in fundamentally different ways.

There are numerous examples of conformal manifolds $(M^n, [\wg])$ that admit multiple non-isometric Poincar\'e-Einstein fillings. This occurs either when $M^n$ is not diffeomorphic to $S^n$, or $[\wg]$ is a conformal class on $S^n$ that deviates  significantly from the standard conformal class of the round metric. For instance, see \cite{Hawking-Page, Pedersen, Hitchin, Anderson-cusp}, as well as the Thurston theory of Dehn surgery and the now well-known ``opening cusps" phenomenon; see \cite{Gromov-hyperbolic, Thurston}. The first example of non-unique Poincar\'e-Einstein filling was discovered by Hawking-Page \cite{Hawking-Page} in their study of the $4$-dimensional AdS-Schwarzschild space, where the conformal infinity is $S^1 \times S^2$, equipped with the standard product metric. See Section \ref{ss:AdS-S} for a more detailed discussion on this topic.

%In this paper, we will establish some rigidity and quantitative rigidity results for the Poincar\'e-Einstein filling of $S^1 \times S^{n - 1}$. As mentioned, even for the standard product metric on $S^1 \times S^{n - 1}$, one can find non-isometric Poincar\'e-Einstein fillings which even own distinct topologies. A well known example of non-unique Poincar\'e-Einstein filling is the AdS-Schwarzschild space; see Section \ref{ss:AdS-S} for details about this. 
%It is of interest to study how one can rule out AdS-Schwarzschild metrics as Poincar\'e-Einstein 
%fillings of $S^1 \times S^{n - 1}$. 

The main goal of this paper is to establish some rigidity and quantitative rigidity results for the Poincar\'e-Einstein fillings of $S^1 \times S^{n - 1}$ equipped with the product metric. 

To state the results more precisely, we first introduce some notations. On the product manifold $S^1 \times S^{n - 1}$, if one normalizes the standard product metrics such that $S^{n - 1}$ is the unit sphere with curvature identically equal to $1$, then one has a one-parameter family of such metrics 
\begin{align}
	\wg_{\lambda} = \lambda^2 \tau^2 + g_c, \quad \lambda > 0, \label{e:product-metric-on-cylinder}
\end{align}
where $(S^1,\tau^2)$ is the unit circle and $(S^{n - 1},g_c)$ is the unit $(n-1)$-sphere 
with sectional curvature equal to $1$. 
It is well known that for any given $\lambda > 0$,  the product manifold $S^1 \times D^n$
 admits a unique hyperbolic metric $g_{\lambda}$ with constant sectional curvature $-1$, whose conformal infinity is $(S^1 \times S^{n - 1}, [\wg_{\lambda}])$.

The first main result of the paper is below.

\begin{Thm}\label{t:uniqueness-nonpositively-curved}
Given $n \gesl 3$ and $\lambda > 0$, let $(X^{n + 1}, g)$ be a complete Poincar\'e-Einstein manifold with conformal infinity $(S^1 \times S^{n - 1}, [\wg_{\lambda}])$ such that $\Ric_g = - n g$.
If $g$ has non-positive sectional curvature, then $(X^{n + 1}, g)$ must be isometric to the hyperbolic manifold $(S^1 \times D^n, g_{\lambda})$.
\end{Thm}

\begin{remark}
The assumption of the non-positive sectional curvature in the theorem cannot be removed. For example, by the classical Cartan-Hadamard theorem, the $(n + 1)$-dimensional AdS-Schwarzschild metrics defined on the simply connected manifold $D^2 \times S^{n - 1}$ cannot be non-positively curved everywhere. But their conformal infinities are the product metrics on $S^1 \times S^{n - 1}$. See Section \ref{ss:AdS-S} for more details.
\end{remark}

The next result exhibits the rigidity of Poincar\'e-Einstein filling for the conformal manifold $(S^1 \times S^{n - 1}, [\wg_{\lambda}])$ whose $S^1$-factor is a {\it sufficiently large}.

\begin{Thm}
\label{t:large-circle}
 Given $n\gesl 3$, there exists a positive number $\lambda_0 = \lambda_0(n) \gesl 1$ such that the following holds.
Let $(X^{n + 1}, g)$ be a complete Poincar\'e-Einstein manifold whose conformal infinity is given by
$(S^1 \times S^{n-1}, [\wg_{\lambda}])$. If the normalized product metric $\wg_{\lambda}$ satisfies $\lambda \gesl \lambda_0$, then 
 $(X^{n + 1}, g)$ is isometric to the hyperbolic manifold $(S^1 \times D^n, g_{\lambda})$.
 
\end{Thm}

\begin{remark}
	When $\lambda > 0$ is small, the conformal manifold $(S^1 \times S^{n-1}, [\wg_{\lambda}])$ admits distinct Poincar\'e-Einstein fillings. For example, in addition to the hyperbolic manifold, it can also bound two non-isometric AdS-Schwarzschild metrics on $D^2 \times S^{n - 1}$. 
	See Section \ref{ss:AdS-S} for a more detailed discussion in the case of $n = 3$. 
	\end{remark}

  Theorems \ref{t:uniqueness-nonpositively-curved} and  \ref{t:large-circle} establish hyperbolic rigidity results, assuming that the boundary conformal structure on $S^1 \times S^{n - 1}$ is standard.
As a comparison, the next theorem formulates the uniqueness of Poincar\'e-Einstein filling,  assuming that the boundary conformal structure on $S^1 \times S^{n - 1}$ is {\it nearly standard}.
This result is essentially a combination of   
the hyperbolic rigidity stated in Theorem \ref{t:large-circle} and a compactness argument.

\begin{Thm}
	\label{t:uniqueness-large-circle-factor}
 	Given $n \gesl 3$, there exists a positive number $\lambda_0 = \lambda_0 (n) > 1$ such that for any $\lambda \gesl \lambda_0$ there exists some $\delta = \delta (\lambda, n) > 0$ that satisfies the following property.
  If a smooth metric $\wg$ on $S^1 \times S^{n - 1}$   satisfies 
  \begin{align}
  	\|\wg - \wg_{\lambda}\|_{C^{2,\alpha}(S^1 \times S^{n - 1})} <  \delta,
  \end{align}
  then the conformal manifold
   $(S^1 \times S^{n - 1}, [\wg])$ has a unique Poincar\'e-Einstein filling $(X^{n + 1}, g)$. In particular, $X^{n + 1}$ is diffeomorphic to $S^1\times D^n$.
\end{Thm}
\begin{remark}\label{re:uniqueness-large-circle}
The existence part of Theorem \ref{t:uniqueness-large-circle-factor} follows from an earlier result by J. Lee (\cite[theorem A]{Lee-fredholm}). Since Lee's result \cite[theorem A]{Lee-fredholm} was established via the implicit function theorem, the resulting Poincar\'e-Einstein  metric  is locally unique (in the moduli space of $S^1 \times D^n$). The main contribution of Theorem \ref{t:uniqueness-large-circle-factor} is to establish the global uniqueness of this metric.  
\end{remark}

The following theorem provides a hyperbolic rigidity under a {\it uniform pinching} condition on the  energy of the Weyl curvature. The main assertion of the theorem is that the pinching constant depends only on the dimension. A particularly new feature of this result is that this rigidity holds independent of any lower bound of the Yamabe constant, and hence is {\it independent of the non-collapse} of the conformal infinity.

\begin{Thm}\label{t:rigidity-for-small-weyl-energy} 
 For any $n \gesl 3$, there exists a dimensional constant $\delta = \delta (n) > 0$ such that the following property holds. Given any $\lambda > 0$, if a Poincar\'e-Einstein manifold $(X^{n+1}, g)$   
  has conformal infinity $(S^1 \times S^{n-1}, [\wg_{\lambda}])$ and satisfies \begin{align}
	\int_{X^{n+1}}|W_g|^{\frac{n + 1}{2}} \dvol_g   \lesl \delta \cdot \left(\mathcal{Y}(S^1 \times S^{n-1}, [\wg_{\lambda}])\right)^{\frac{n}{2}}, \label{e:small-Weyl-energy}
\end{align} 	
 then $(X^{n + 1}, g)$ must be isometric to the hyperbolic manifold $(S^1 \times D^n, g_{\lambda})$.

\end{Thm}

\begin{remark}The smallness of the Weyl energy in \eqref{e:small-Weyl-energy} is necessary for the hyperbolic rigidity result above. In fact, one can find a family of AdS-Schwarzschild metrics on $D^2\times S^2$ that are Poincar\'e-Einstein fillings of the standard products on $S^1 \times S^2$ and have uniformly bounded Weyl energy. However, the Weyl energy has a lower bound away from zero. See Section \ref{ss:AdS-S} for details.
\end{remark}

\begin{remark}
In the results above, we have assumed that the conformal infinity $S^1\times S^{n - 1}$ has dimension $n \gesl 3$. In Section \ref{ss:parabolic}, we will discuss the non-uniqueness of the Poincar\'e-Einstein fillings in the special case $n = 2$. 
\end{remark}

\subsection{Organization of the paper}

In Section \ref{s:preliminaries}, we introduce some basic definitions and results which will be frequently used later in the proofs of the main theorems. 
  Section \ref{s:non-positively-curved-PE} is devoted to the proof of Theorem \ref{t:uniqueness-nonpositively-curved}.
In Section \ref{s:more-results}, we will prove Theorems \ref{t:large-circle}, \ref{t:uniqueness-large-circle-factor},  \ref{t:rigidity-for-small-weyl-energy}, which provide various geometric rigidity and quantitative rigidity results. 
  Section \ref{s:examples} exhibits several motivating examples. In
Section \ref{s:discussions}, we will propose some open questions and conjectures.

\subsection{Acknowledgements} 
The authors would like to thank Yuxin Ge, Gang Li, and Jie Qing for many helpful discussions. The first named author also would like to thank Yuxin Ge for the help in providing and clarifying several useful points from their earlier collaborations.

\section{Preliminaries}
\label{s:preliminaries}

In this section, we summarize all preliminary materials for the proofs in later sections.

\subsection{Basics in conformally compact Einstein manifolds}
\label{ss:basics-in-PE}

This subsection introduces the basic notation and results in the geometry of Poincar\'e-Einstein manifolds.

Let $(X^{n + 1}, g)$ be a complete Riemannian manifold that is diffeomorphic to the interior of a smooth compact manifold $\bX$ with smooth boundary $ M^n \equiv \p\bX$. A
{\it conformal compactification} of $(X^{n + 1}, g)$ is given by a smooth defining function $\rho \in C^{\infty}(\bX)$ for the boundary $M^n$, satisfying
\begin{align}
	\begin{cases}
			\rho > 0  & \text{in}\ X,
		\\
		    \rho = 0, \ |d\rho|_{\wg} > 0 & \text{on}\ M^n,
	\end{cases}
	\end{align}
where $\bg \equiv \rho^2 g$ extends to a compact and smooth Riemannian metric on $\bX$. The non-degeneracy condition on $M^n$ implies that a small tubular neighborhood $\{\rho < \epsilon\} \subset \bX$ of $M^n$ is diffeomorphic to $[0,\epsilon) \times M^n$. For a given defining function $\rho$, letting $\wg \equiv \rho^2 g|_{\rho = 0}$, the conformal 
manifold $(M^n, [\wg])$
is called the {\it conformal infinity}.
Moreover, 
the sectional curvature of $g$ satisfies the asymptotic behavior 
\begin{align}\sec_g = - |d\rho|_{\bar{g}}^2 + O(\rho^2)\quad \text{as}\ \rho \to 0.
\end{align} 
A complete manifold $(X^{n + 1}, g)$
is said to be {\it  conformally compact} if $X^{n+1}$ admits a conformal compactification.  

 Next, we introduce the notions of asymptotic hyperbolicity and asymptotically hyperbolic Poincar\'e-Einstein manifolds.
\begin{definition}[Asymptotic hyperbolicity] \label{d:AH} A conformally compact manifold $(X^{n+1}, g)$ with conformal infinity $(M^n, [\wg])$ is called   asymptotically hyperbolic if it admits a compactified metric $\bg = \rho^2 g$ such that $|d \rho|_{\wg} \equiv 1$  on $M^n$, where $\wg \equiv \bg|_{\rho = 0}$.
\end{definition}

 \begin{definition}
 	[Poincar\'e-Einstein manifold]
 	\label{d:PE-manifold}
 	A complete Einstein manifold $(X^{n + 1}, g)$ is said to be
 	a Poincar\'e-Einstein manifold if $\Ric_g = -n g$ and   $(X^{n + 1}, g)$  is asymptotically hyperbolic in the sense of Definition \ref{d:AH}.
 \end{definition}

When studying Poincar\'e-Einstein manifolds, it is  convenient to 
introduce a coordinate system using a geodesic defining function associated with the conformal infinity. To describe this, we recall a well-known result; see \cite[lemma 5.1]{Lee} or \cite[lemma 2.1]{Graham}.  
\begin{lemma}\label{l:geodesic-defining-function}
	Let $(X^{n + 1}, g)$ be a complete Poincar\'e-Einstein manifold with conformal infinity $(M^n, [\wg])$. Let $\rho$ be any smooth defining function of $M^n$. Then for any representative $h \in [\wg]$, there exists a unique defining function $x \in C^{\infty}(\overline{X^{n + 1}})$ of $M^n$, called the geodesic defining function associated with $h$, such that 
$|dx|_{\bg} \equiv 1$
in some tubular neighborhood of $M^n$, 
and $\bg|_{x= 0} = h$, where $\bg = x^2 g$.

\end{lemma}

In the lemma below, we will use another special defining function for the conformal infinity, introduced in the  work of Fefferman-Graham (\cite{FG}). We will use this special choice of defining function later in the proof of Theorem \ref{t:uniqueness-large-circle-factor}.
\begin{lemma}
[\cite{FG}] \label{l:FG-metric} Let $(X^{n + 1}, g)$ be a complete Poincar\'e-Einstein manifold with conformal infinity $(M^n, [\wg])$. Then, 
for any representative $\wg\in[\wg]$ on $M^n$, there exists a conformal compactification 
  $\bg = \varrho^2 g \equiv  e^{2w} g$ that satisfies \begin{align}
	-\Delta_g w = n \quad  \text{on}\ X^{n + 1},  
\end{align}
and $\bg|_{M^n} = \wg$. Near the boundary, the function $w$ has the asymptotic expansion in $x$, 
\begin{align}w = \log x +O(x^\epsilon),\quad \text{for some} \ \epsilon \in (0, 1),\end{align}    where $x$ is the geodesic defining function associated with $\wg$, 
 as in Lemma \ref{l:geodesic-defining-function}.
\end{lemma}

 The special conformal compactification in the lemma is often referred to the Fefferman-Graham compactification (or Fefferman-Graham metric) in the literature. 

\medskip

In the following, we recall the definition of renormalized volume,  an important conformal invariant of Poincar\'e-Einstein manifolds. Using the geodesic defining function as in Lemma \ref{l:geodesic-defining-function}, the renormalized volume is defined through the volume expansion of the exhaustion domains in $X^{n + 1}$,
\begin{align}
	\Omega_{\epsilon} \equiv \{q\in X^{n+1}: x(q) \gesl \epsilon\},
\end{align}
which are obtained by cutting off of the complete manifold $X^{n+1}$ by removing a tubular neighborhood of the conformal infinity.
\begin{definition}
[Renormalized volume] For sufficiently small $\epsilon$, the volume of $\Omega_{\epsilon}$ admits the following expansions:
\begin{enumerate}
	\item When $n$ is odd,
\begin{align}
	\Vol_g(\Omega_{\epsilon}) = c_0 \epsilon^{-n} + \sum\limits_{k\ \text{is odd}} c_k \cdot \epsilon^{-k} + \mathcal{V}(X^{n + 1}, g) + o(1)  \  \text{as}\ \epsilon\to 0.
\end{align}	
	\item When $n$ is even,

\begin{align}
	\Vol_g(\Omega_{\epsilon}) = c_0 \epsilon^{-n} + \sum\limits_{k\ \text{is even}} c_k \cdot \epsilon^{-k} + L\log(\epsilon^{-1}) + \mathcal{V}(X^{n + 1}, g) + o(1)\  \text{as}\ \epsilon\to 0.
\end{align}
\end{enumerate}
Here, the constant $\mathcal{V}(X^{n+1}, g)$ is called the renormalized volume.	
\end{definition}

It was established in \cite{Graham} that when $n$ is odd, the renormalized volume is a conformal invariant in the sense that it is independent of the choice of the representative in the conformal class $[\wg]$ on $M^n$.   When $n$ is even, it was shown in \cite{FG-conformal-invariants} that $L$ is a conformal invariant.
Furthermore, in the case of $n = 3$, a relation between the renormalized volume and the Chern-Gau{\ss}-Bonnet theorem was established by M. Anderson in \cite{Anderson-renormalization}. See also \cite{CQY-renormalized-volume} for a different proof. The following formula will be used in later sections.

\begin{theorem}[\cite{Anderson-renormalization}] \label{t:Chern-Gauss-Bonnet}
If $(X^4, g)$
is a complete Poincar\'e-Einstein manifold, then 
\begin{align}
	8 \pi^2 \chi(X^4) = \int_{X^4}|W_g|^2 \dvol_g + 6 \mathcal{V}(X^4, g).
\end{align}
\end{theorem}

  We collect now several technical results that will be used in the setup and the proof of Theorem \ref{t:uniqueness-large-circle-factor} and Theorem \ref{t:rigidity-for-small-weyl-energy}.

In the context of Poincar\'e-Einstein manifolds, there is an important volume comparison result that relates the Riemannian geometry of a complete Einstein manifold $(X^{n + 1}, g)$ to the  conformal invariants of its conformal infinity $(M^n, [\wg])$.
The following theorem is originally claimed in \cite{DJ}, while a complete and rigorous proof
 was given in \cite[theorem 1.3]{LQS} 
\begin{theorem}
[Volume comparison] \label{t:volume-comparison}
Let $(X^{n + 1}, g)$ be a complete Poincar\'e-Einstein manifold with conformal infinity $(M^n, [\wg])$. If the Yamabe constant $\mathcal{Y}(M^n, [\wg])$ of the conformal infinity is positive, then for any $p\in X^{n + 1}$ and for any $r > 0$,
	\begin{align}
		\left(\frac{\mathcal{Y}(M^n, [\wg])}{\mathcal{Y}_0}\right)^{\frac{n}{2}} \lesl \frac{\Vol_g(B_r(p))}{\Vol_{-1}(B_r)} \lesl 1,
	\end{align}
where $\mathcal{Y}_0 \equiv \mathcal{Y}(S^n, [g_c])$ is the Yamabe constant for the  round metric on $S^n$ and $\Vol_{-1}(\cdot)$ is the volume function on the hyperbolic space $\mathbb{H}^{n + 1}$ of curvature $-1$.
 \end{theorem}
The volume comparison theorem above is a powerful tool in studying the metric geometry of $(X^{n+1}, g)$, which we will apply later to derive the regularity and uniform curvature estimates in this paper. Applying this volume estimate and a contradiction compactness argument, Li-Qing-Shi proved in \cite[theorem 1.6]{LQS} the following uniform curvature estimate.
\begin{theorem}\label{t:curvature-pinching}
Given $n\gesl 3$ and $\epsilon > 0$, there exists 	
a positive number $\delta = \delta(n , \epsilon) > 0$ such that the following property holds. 
Let $(X^{n+1}, g)$ be a Poincar\'e-Einstein manifold with conformal infinity $(M^n, [\wg])$. 
If 
\begin{align}\frac{\mathcal{Y}(M^n, [\wg])}{\mathcal{Y}(S^n, [g_c])} \gesl 1 - \delta,\end{align}
then $\sup\limits_{X^{n+1}}|\sec_g + 1| \lesl \epsilon$.
\end{theorem}

 In the statement of Theorem \ref{t:uniqueness-large-circle-factor}, which we will later verify in this paper, if a conformal structure $[\wg]$ on $S^1 \times S^{n - 1}$ is sufficiently close to the standard conformal class $[\wg_{\lambda}]$ and the circle factor is sufficiently large, then there exists a unique Poincar\'e-Einstein metric on $S^1\times D^n$.   
   Here, we first state the result for the existence part of such metrics, which is a direct consequence of the following perturbation existence result of Lee; see\cite[theorem A]{Lee-fredholm} for a more general statement of such result. 

\begin{theorem}\label{t:Lee-IFT}
Let $(X^{n + 1}, h)$
 be a complete Poincar\'e-Einstein manifold with conformal infinity $(M^n, [\wh])$ that satisfies $\sec_h \lesl 0$ on $X^{n + 1}$. For any metric $\wh\in[\wh]$, there exists some small number $\epsilon > 0$  such that if a Riemannian metric $\wg$ on $M^n$ satisfies
 \begin{align}
 	\|\wg - \wh\|_{C^{2,\alpha}(M^n)} < \epsilon, 
 \end{align}
 then there exists a complete Poincar\'e-Einstein metric $g$ on $X^{n + 1}$ whose conformal infinity is  
 $(M^n, [\wg])$.
 \end{theorem}

Theorem \ref{t:Lee-IFT} was established using the implicit function theorem with respect to some Banach norm (a weighted H\"older norm), the solution $g$ is unique in a sufficiently small  neighborhood of the ``base point" $h$ in this Banach space. As mentioned in Remark \ref{re:uniqueness-large-circle}, we will later establish the  global uniqueness of the metric in Theorem \ref{t:uniqueness-large-circle-factor}.

\subsection{The model hyperbolic metrics}
\label{ss:model-space}
In this subsection, we will discuss the model hyperbolic metrics on $S^1 \times D^n$ whose conformal infinity is given by the standard product metrics 
on $S^1 \times S^{n - 1}$.

For any $\lambda>0$, let $\wg_{\lambda}\equiv \lambda^2\tau^2  + g_c$ be the normalized product metric on $S^1\times S^{n-1}$, where $(S^1,\tau^2)$ is the unit circle with the standard metric and $(S^{n-1}, g_c)$ is the round sphere with constant curvature $+1$. 
Then for each $\lambda > 0$, one can compute explicitly that the metric 
\begin{align}
g_{\lambda} = dr^2 + \sinh^2(r) g_c + \cosh^2(r) (\lambda^2 \tau^2), \quad 0\lesl r< +\infty,\end{align}
on $S^1 \times D^n$
has constant curvature $-1$, where $r$ is the distance to the circle $\sigma_{\dZ}\subset S^1 \times D^n$ and the metric of $\sigma_{\dZ}$ is $\lambda^2\tau^2$. 
It follows that $\lambda$ equals the length of the circle  $\sigma_{\dZ}$, 
which also satisfies
$\lambda = 2 
\cdot \Injrad_{g_{\lambda}}(S^1 \times D^n)$.

One can also re-write $g_{\lambda}$ as in the geodesic defining function associated to   $(S^1 \times S^{n - 1}, \wg_{\lambda})$:
\begin{align}
	g_{\lambda} = x^{-2}\left(dx^2  + \left(1 - \frac{x^2}{4}\right)^2 g_c +  \left(1 + \frac{x^2}{4}\right)^2 (\lambda^2 \tau^2)\right),
\end{align}
where $0\lesl x \lesl 2$.
There is another coordinate representation 
of $g_{\lambda}$
\begin{align}
	g_{\lambda} = \frac{ds^2}{1 + s^2} + s^2 g_c +  (1 + s^2) (\lambda^2 \tau^2).  
\end{align}
 This representation is related to the example in Section \ref{ss:AdS-S}.

The above is a basic model in hyperbolic geometry. Let $\Gamma \subset \Isom(\dH^{n + 1})$
be a discrete subgroup, namely a {\it Kleinian group}. Notice that the isometry group
$\Isom(\dH^{n + 1})$
shares the same identity component as the M\"obius group $\MM(S^n)$, which is the group of conformal transformations on the round sphere $S^n$. This isomorphism also provides a correspondence between the metric Riemannian structure of $\dH^{n + 1}$ and the conformal structure of $S^n$.

 If the discrete group $\Gamma \subset \Isom(\dH^{n + 1})$ acts freely on $\dH^{n + 1}$, then the quotient $\dH^{n + 1}/\Gamma$ becomes a hyperbolic manifold. 
Moreover, it follows from a basic result in Riemannian geometry that $\Gamma$ is torsion free. Thus, for any $\tp \in \dH^{n+1}$, the orbit $\Gamma \cdot \tp$ has limit points
on $S^n$. The set of all such limit points 
are called the {\it limit set} of $\Gamma$, denoted by $\Lambda(\Gamma)$. It is well known that if a locally conformally flat manifold $(M^n ,g)$
is compact and has nonnegative scalar curvature, then it must be a Kleinian manifold such that its fundamental group is isomorphic to a Kleinian group in $\Mob(S^n)$; see \cite{SY}.

A Kleinian group is called {\it elementary}
if $\#\Lambda(\Gamma) < \infty$. 
It is well known that $\#\Lambda(\Gamma) < \infty$
if and only if $\#\Lambda(\Gamma) \in \{0 , 1 , 2\}$. Correspondingly, $\Gamma$ is isomorphic to $\{e\}$, $\dZ^n$, and $\dZ$, respectively. Moreover, the quotients $\Omega(\Gamma)/\Gamma$ from the spherical domain $\Omega(\Gamma)\equiv S^n\setminus \Lambda(\Gamma)$ are $S^n$, $\dT^n$, and $S^1 \times S^{n - 1}$, respectively. 
Notice that the hyperbolic metric $g_{\lambda}$ on $S^1 \times D^n$ corresponds to the elementary Kleinian group $\Gamma \cong \dZ$, which acts on $S^n$ as dilations between the two fixed points --- the north and south poles.  The next case is when $\Gamma \cong \dZ^n$.  
In this case, the group $\Gamma$ is a rank-$n$ Bieberbach group in $\MM(S^n)$, which acts on $S^n\setminus \{N\}\approx \dR^n$ as translations.  Given a flat metric $\wg_0$ on 
$\dT^n$, one can obtain a hyperbolic metric 
\begin{align}
	g = dt^2 + e^{2t} \wg_0, \quad t\in(-\infty, + \infty),
\end{align}
on $\dR \times \dT^n$. Apparently, 
$(\dR \times \dT^n, g)$
has two ends:
the one is of infinite volume, 
and the other is a finite volume cusp end. 
See Section \ref{ss:parabolic} for more discussions about this case.

\subsection{Metric geometry of non-positively curved spaces}
\label{ss:metric-geometry-nonpositive-curvature}

Recall that, a key assumption in Theorem \ref{t:uniqueness-nonpositively-curved} is that the Poincar\'e-Einstein manifold $(X^{n + 1}, g)$
is non-positively curved. In this subsection, 
we will establish some basic tools in the metric geometry of non-positively-curved spaces, which will be frequently used later in Section \ref{s:non-positively-curved-PE}. 

\medskip

We start with some basic notions.
Let $(Z, d)$ be a metric space, and  let $\Isom(Z)$ be the isometry group of $Z$.

\begin{itemize}
\item  A geodesic line in $Z$ is an isometric embedding $\ell: (-\infty,+\infty)\to Z$. In other words,  
\begin{align}
	d(\ell(t) , \ell(s)) = |t - s|\quad \text{for all} \ t,s \in \dR.
\end{align}

\item Given an isometry $\gamma \in \Isom(Z)$, the displacement function of $\gamma$ is the function $d_{\gamma}: Z \to [0,\infty)$ defined by 
 \begin{align}
d_{\gamma}(p) \equiv d(p, \gamma \cdot p), \quad p\in Z.
 \end{align}

\item  Given an isometry $\gamma \in \Isom(Z)$, the translation length of $\gamma$ is defined by 
\begin{align}
	|\gamma| \equiv \inf\left\{d_{\gamma}(x)| x \in Z\right\}.
\end{align}

\item For any element $\gamma \in \Isom(Z)$, the minimum set $\Min(\gamma)\subset Z$ of $\gamma$ is defined by 
 \begin{align}\Min(\gamma) \equiv \{x \in Z| d_{\gamma}(x) = |\gamma|\}.\end{align}

\item An isometry $\gamma\in \Isom(Z)$ is called semi-simple if $\Min(\gamma)\subset Z$ is non-empty. A subgroup $\Gamma \subset \Isom(Z)$ is called semi-simple if all its elements are semi-simple. 
\item Any isometry $\gamma\in \Isom(Z)$ can be classified as follows: 
\begin{enumerate}
	\item {\it elliptic} if $\gamma$ has a fixed point in $Z$;
	\item {\it hyperbolic (or axial)} if the displacement function $d_{\gamma}$ achieves a {\it positive} minimum in $Z$;
	\item {\it parabolic} if the displacement function $d_{\gamma}$ does not achieve a minimum in $Z$, in other words, $\Min(\gamma)$ is empty. 
\end{enumerate}

\item Let $\Gamma \subset \Isom(Z)$ be a subgroup of isometries. Then we denote by 
\begin{align}
	\Min(\Gamma) \equiv \bigcap\limits_{\gamma \in \Gamma} \Min(\gamma)
\end{align}
the minimum set of $\Gamma$.
\end{itemize}
The lemma below is a standard but useful result concerning the invariance property of the minimum set. 

\begin{lemma} \label{l:invariance-of-minimum-set} Let $(X,d)$ be a length metric space and let $\Gamma$ and $N$ be two subgroups of $\Isom(X)$. 
If any two elements $\gamma\in \Gamma$ and $h\in N$ commute,  
then $N$ keeps the minimum set $\Min(\Gamma)$ of $\Gamma$ invariant.
\end{lemma}

Now we present simple examples that illustrate the above notations.  
\begin{example}
Given $\lambda > 0$, the hyperbolic metric $g_{\lambda}$ defined on $X^{n+1} = S^1 \times D^n$ has constant curvature $-1$. Let $\Gamma \equiv \pi_1(X^{n+1}) \cong \dZ$ be the fundamental group of $X^{n + 1}$, which acts freely and isometrically on the universal covering space $(\ucX, \tg_{\lambda})$. One can easily check that the transformation group $\Gamma$ is semi-simple. In fact,  for any $\gamma\in\Gamma$, 
\begin{align}
	\Min(\gamma) = \{\tx\in\ucX: r(\tx) = 0\},
\end{align}
which is a geodesic line in $\ucX$. Therefore, $\Min(\Gamma) \equiv \bigcap\limits_{\gamma\in\Gamma} \Min(\gamma)  = \Min(\gamma)$ for any $\gamma\in\Gamma$.

In this example, every element $\gamma \in \Gamma \subset \Mob(S^n)$ corresponds to a dilation on the cylinder $\dR^n \setminus \{\bo\}\overset{\text{diffeo}}{\approx}\dR \times S^{n - 1}$: 
\begin{align}
	\gamma: \dR^n \setminus \{\bo\} \to \dR^n \setminus \{\bo\},\quad \bx \mapsto \lambda \cdot \bx, \quad \bx\in \dR^n \setminus \{\bo\},
\end{align}
where $\lambda > 0$ is a fixed number.
\end{example}

\begin{example} As mentioned in the last subsection, when $\Gamma\subset \Mob(S^n)$ is a group of translations on $\dR^n$ that is isomorphic to $\dZ^n$, one can find a hyperbolic metric 
on $\dR \times \dT^n$,
\begin{align}
	g = dt^2 + e^{2t} \wg_0, \quad t\in(-\infty, + \infty),
\end{align}
where $\wg_0$ is the quotient metric on 
$\dR^n/\Gamma \approx \dT^n$. Obviously, $|\gamma| = 0$
for any element 
$\gamma\in \Gamma$.
In fact, 
$d_{\dH^{n + 1}}(\gamma\cdot\tx, \tx) \to 0$ as $t(\tx) \to -\infty$. Meantime, 
$\Min(\gamma) = \emptyset$, and thus every element is parabolic. 
\end{example}

The following lemma is a standard result, which characterizes the minimal set of a hyperbolic isometry on a simply-connected space with non-positive curvature. We refer readers to \cite[Part II.6.1]{BH}.
\begin{lemma}\label{l:minimum-set-of-hyperbolic-element}
 Let $(X ,g)$ be a complete simply-connected Riemannian manifold with non-positive sectional curvature. If $\gamma\in\Isom(X^{n + 1})$ is a hyperbolic isometry, then $\Min(\gamma)$ is given by a set of parallel geodesic lines. That is, $\Min(\gamma)$ is isometric to a metric product $Y \times \dR$. Moreover,   $\gamma$ acts on $\Min(\gamma)$ by $\dR$-translation,
 \begin{align}
 	\gamma(y,t) = \gamma(y, t + t_{\gamma}), \quad y\in Y, \ t\in \dR,
 \end{align}
 where $t_{\gamma} \in \dR$ is a fixed number that satisfies $|t_{\gamma}| = |\gamma|$.
\end{lemma}

We conclude this section by introducing two fundamental results in the spaces with non-positive curvature: {\it Flat Strip Theorem} (theorem \ref{t:flat-strip}) and {\it Flat Torus Theorem} (theorem \ref{t:flat-torus}); see \cite[Part II]{BH} for more details. 

\begin{theorem}
[Flat Strip Theorem] \label{t:flat-strip} Let $(X ,g)$ be a complete simply-connected Riemannian manifold with $\sec_g \lesl  0$. Let $\sigma_1,\sigma_2:\dR \to X$ be geodesic lines. If there exists a constant $C > 0$ such that   
\begin{align}
	d_g(\sigma_1(t), \sigma_2(t)) \lesl C , \quad \forall \ t\in\dR,
\end{align}
then $\sigma_1$ and $\sigma_2$ bound a flat strip, i.e., a convex region that is isometric to $[0,D]\times \dR \subset \dR^2$ for some $D > 0$.
\end{theorem}

\begin{theorem}[Flat Torus Theorem]   \label{t:flat-torus}  Let $(X ,g)$ be a complete simply-connected Riemannian manifold with $\sec_g \lesl  0$. If $\Gamma \lesl \Isom(X)$ is a free abelian group of rank $k$ acting properly by semi-simple isometries on $X$, then the following holds. 
\begin{enumerate}
\item $\Min(\Gamma) = \bigcap\limits_{\gamma \in \Gamma} \Min(\gamma) \neq \emptyset$
and splits as a Riemannian product $Y \times \dR^k$.
\item Each element $\gamma\in\Gamma$ preserves the isometric splitting $\Min(\Gamma) \equiv Y \times \dR^k$. Moreover, every $\gamma\in\Gamma$ acts on $Y$ as the identity and acts on $\dR^k$ by translation. 
\item For any $y\in Y$, the $\Gamma$-quotient of the Euclidean slice $\{y\}\times \dR^k$ is isometric to a flat torus $\dT^k$. 
\end{enumerate}
\end{theorem}

\section{Poincar\'e-Einstein Manifolds with Non-Positive Curvature}
\label{s:non-positively-curved-PE}

This section is devoted to proving our first main result, Theorem \ref{t:uniqueness-nonpositively-curved}. 
Here we state this result in a more general form.

\begin{theorem}\label{t:product-top-geo-rigidity}
Let $(X^{n+1}, g)$ be a complete Poincar\'e-Einstein manifold with $\Ric_g \equiv - n\cdot g$  and $\sec_g \lesl 0$. Let $(M^n, [\wg])$
denote its conformal infinity with   $\mathcal{Y}(M^n, [\wg]) \gesl 0$. 
 \begin{enumerate}
	\item If, in addition, $\pi_1(M^n) \cong \dZ$, then $X^n$ is diffeomorphic to  $S^1 \times D^n$ and $M^n$ is diffeomorphic to $S^1 \times S^{n - 1}$.
	
	\item If, furthermore, $(M^n, [\wg])$ is conformal to the standard product $(S^1 \times S^{n - 1}, \wg_{\lambda})$, then $(X^{n+1}, g)$ is isometric to the hyperbolic manifold $(S^1 \times D^n, h_{\lambda})$ with $\sec_{h_{\lambda}} \equiv - 1$.

\end{enumerate}   
\end{theorem}
Item (1) of the theorem provides a topological rigidity result that holds in a general setting, while item (2) establishes hyperbolic rigidity in the case where the conformal infinity is the standard round cylinder.
Throughout this section, we  assume that $(X^{n+1}, g)$ is a complete Poincar\'e-Einstein manifold with $\Ric_g \equiv - n\cdot g$  and $\sec_g \lesl 0$, and that its conformal infinity $(M^n, [\wg])$ satisfies $\mathcal{Y}(M^n, [\wg]) \gesl 0$.
\vskip .1in

In Sections \ref{ss:deck-transformation} - \ref{ss:topological-splitting}, we will  establish item (1) of Theorem \ref{t:product-top-geo-rigidity}, that is, we will prove the topological rigidity of $X^{n+1}$ under the following additional  assumption: 
\begin{assumption}\label{a:pi-1}
The fundamental group of $M^n$ is infinite cyclic. 
\end{assumption}

\noindent The main idea is to adapt tools from the metric geometry of non-positively curved spaces, introduced in Section \ref{ss:metric-geometry-nonpositive-curvature} to the Poincar\'e-Einstein setting.

\vskip .1in

In Sections \ref{ss:extension-of-Killing-fields} - \ref{ss:dimension-reduction}, we will establish item (2) of Theorem \ref{t:product-top-geo-rigidity}, that is, we will prove the isometric rigidity result stated in Theorem \ref{t:product-top-geo-rigidity}(2), under the additional assumption.
\begin{assumption}\label{a:product-conformal-infinity}
The conformal class $[\wg]$ contains a standard product metric $\wg_{\lambda}$ on $S^1 \times S^{n - 1}$.  	
\end{assumption}

 Our overall strategy to establish the hyperbolic rigidity result is to construct a canonical gauge for $(X^{n + 1}, g)$, with respect which we can generate sufficiently large symmetries on $X^{n+1}$. By accurately characterizing the symmetries, we ultimately are led to a dimension reduction argument. This approach also relies fundamentally on the metric-geometric properties of $(X^{n+1}, g)$ as a non-positively curved space, as provided in Sections \ref{ss:deck-transformation} - \ref{ss:topological-splitting}. 

\vskip .1in
\subsection{Deck transformations on the universal cover}
\label{ss:deck-transformation}
In this subsection, we investigate the topological structure of   Poincar\'e-Einstein manifolds in the setting of item (1) of Theorem \ref{t:product-top-geo-rigidity}. Specifically,  
we consider a {\it non-positively curved} Poincar\'e-Einstein manifold $(X^{n + 1}, g)$ that satisfies Assumption \ref{a:pi-1}. 

The main technical result in this subsection, Proposition \ref{p:minimum-set},  reveals 
how the elements in the deck transformation group $\Gamma\equiv \pi_1(\ucX)$ act as isometries, 
  on the universal cover $\ucX$. 
The main assertion is that the minimum set $\Min(\Gamma)$  splits isometrically as flats that are invariant by $\Gamma$. This splitting structure serves as the foundation for establishing the global splitting of $\ucX$, as proved in Theorem \ref{t:normal-exp}.

Let  $i: M^n \to \overline{X^{n+1}}$ be the natural inclusion, which induces 
	a homomorphism 
	\begin{align}
		i_*: \pi_1(M^n) \longrightarrow  \pi_1(\overline{X^{n + 1}}). \label{e:induced-homomorphism}
	\end{align}
 The following lemma relates $\pi_1(M^n)$ to $\pi_1(X^{n + 1})$, under the assumption that $(X^{n+1}, g)$ is non-positively curved. 
\begin{lemma}\label{l:cyclic-fundamental-group}
Let $(X^{n + 1}, g)$ be a complete Poincar\'e-Einstein manifold with $\sec_g \lesl 0$. If Assumption \ref{a:pi-1} holds, then the induced map $i_*$ is an isomorphism.  \end{lemma}

\begin{proof}
Under the curvature assumption $\sec_g \lesl  0$, it follows from the Cartan-Hadamard theorem that the universal cover  $\widetilde{X^{n + 1}}$ 	
	is diffeomorphic to $\dR^{n+1}$.
	Therefore, $X^{n +1}$ is diffeomorphic to $\dR^{n + 1}/\Gamma$, where $\Gamma \equiv  \pi_1(X^{n + 1})$.
Since $(X^{n + 1}, g)$ is a Poincar\'e-Einstein manifold, it is a well-known result that the natural inclusion $i: M^n \to \overline{X^{n+1}}$ induces 
	a surjective homomorphism 
	\begin{align}
		i_*: \pi_1(M^n) \longrightarrow  \pi_1(\overline{X^{n + 1}});
	\end{align}
	for example see \cite{Witten-Yau} or \cite[theorem 1]{Cai-Galloway}.	So it follows that $\pi_1(\overline{X^{n + 1}})$ is isomorphic to $\pi_1(M^n)/\Ker(i_*)$. Since $(X^{n + 1}, g)$ is non-positively curved, by Cartan-Hadamard theorem, $\pi_1(\overline{X^{n + 1}})$ ($\cong\Gamma$) must be torsion-free. Therefore, $\Gamma$
 is either trivial or isomorphic to $\dZ$. 
 Correspondingly, $\Ker(i_*)=\pi_1(M^n)$ or $\{e\}$.
 
 However, in the former case, 	
	$\overline{X^{n+1}}$ is homotopy equivalent to $\dR^{n + 1}$. By elementary topological computations, $M^n$ must be a homology sphere.
	In fact, first we have 
	\begin{align*}
H^{n + 1}(\overline{X^{n + 1}}, M^n) \cong \dZ \quad \text{and} \quad 
H^k(\overline{X^{n + 1}}, M^n) = 0 \ \text{for any}\  k \neq n + 1.
 	\end{align*}
	From the long exact sequence
\begin{align*}
	\ldots \to H^{k - 1}(\overline{X^{n + 1}})  \to H^{k - 1}(M^n) \to H^k(\overline{X^{n + 1}}, M^n) \to H^k(\overline{X^{n + 1}}) \to H^k(M^n)  \ldots 
\end{align*}
we have that 
$H^{k - 1}(M^n) \cong H^k(\overline{X^{n + 1}}, M^n)$ for any $k \gesl 2$.
Next, by Lefschetz duality, it holds that $H^k(\overline{X^{n + 1}}, M^n) \cong H_{n - k}(\overline{X^{n + 1}})$. In particular, $H^1(\overline{X^{n + 1}}, M^n) = 0$ 
Since the restriction
$H^0(\overline{X^{n + 1}}) \to H^0(M^n)$ is an isomorphism, we have $H^0(\overline{X^{n + 1}}, M^n) = 0$.
Therefore,
\begin{align}
	H^k(M^n) \cong 
	\begin{cases}
		\dZ, & k = 0, n
		\\
		0, & 1\lesl  k \lesl  n - 1.
	\end{cases}
\end{align}
 This contradicts to the assumption that $\pi_1(M^n)\cong \dZ$, which completes the proof.
\end{proof}

For the remainder of this section, we denote $\Gamma\equiv \pi_1(X^{n + 1})$.
 Our next goal is to construct a $\Gamma$-invariant geodesic line on the universal cover $\widetilde{X^{n + 1}}$ using the key property that $\Gamma$ is isomorphic to $\dZ$ and generated by a hyperbolic element.  

\begin{proposition}
 \label{p:minimum-set}
Let $(X^{n + 1}, g)$ be a complete Poincar\'e-Einstein manifold with $\sec_g \lesl 0$. If Assumption \ref{a:pi-1} holds,   
 then
the minimum set
$\Min(\Gamma)\subset \ucX$ is isometric to a metric product $Y \times \dR$,  and satisfies the following properties:
\begin{enumerate}
\item $Y$ is a connected, compact, and convex set;
\item Every element $\gamma\in \Gamma$ preserves the isometric splitting of $\Min(\Gamma)$. Moreover,	each $\gamma\in\Gamma$ acts as the identity on $Y$ and as a translation on $\dR$.
\end{enumerate}
 \end{proposition}

\begin{proof}
Let $\gamma\in \Isom(\ucX)$ be the generator of the deck transformation group $\Gamma\cong \dZ$. Our first assertion is that every element $\gamma^k \in \Gamma\subset \Isom(\ucX)$ is hyperbolic. Since $\Gamma$ acts freely on $\ucX$, there is no elliptic element in $\Gamma$. Thus to show the element is hyperbolic,  it suffices to show that no element $\gamma^k\in \Gamma$ is parabolic. 

To see so, we will prove by contradiction.   
Suppose that $\Gamma$ has a parabolic element $\gamma^k$ for some $k\in\dZ$. By definition, $\gamma^k$ satisfies the following properties:
\begin{enumerate}[(a)]
\item for any $\tx \in \ucX$, it holds that
$d_{\gamma^k}(\tx) > |\gamma^k|$; 
\item there exists  a sequence of a points $\{\tx_j\}_{j=1}^{\infty}$ in some fundamental domain $\Omega \subset  \ucX$ of the deck transformation group $\Gamma$ such that  
for some reference point $\tx_0\in \Omega$,  
\begin{align}
	d_{\tg}(\tx_j, \tx_0) \to \infty \quad \text{and} \quad d_{\gamma^k}(\tx_j) \to |\gamma^k| \quad \text{as}\ j \to \infty.
\end{align}
\end{enumerate}

Property (a) is obvious since no element in $\Gamma$ is elliptic. To show property (b), we take a sequence of points 
$\{\ty_j\}_{j=1}^{\infty}\subset \ucX$ such that for some $\tx_0\in \ucX$,
\begin{align}
	d_{\tg}(\ty_j, \tx_0) \to \infty \quad \text{and} \quad d_{\gamma^k}(\ty_j) \to |\gamma^k| \quad \text{as}\ j \to \infty.
\end{align}
Next,  notice that since
since $\Gamma=\LB\gamma\RB$ is cyclic, we have

\begin{align}d_{\gamma^k}(\alpha\cdot \tx) = d_{\gamma^k}(\tx)\end{align}
 for any $\alpha \in \Gamma$ and $\tx \in \ucX$.
 Let $\Omega \subset \ucX$ be the fundamental domain that contains $\tx_0$. Then for any $j$, there exists a compact set $\Omega \subset \ucX$ such that for every $j$, there exists some $\alpha_j\in \Gamma$ that satisfies 
 \begin{align}
\tx_j = \alpha_j\cdot \ty_j\in \Omega \quad \text{and} \quad 	d_{\gamma^k}(\tx_j) = d_{\gamma^k}(\ty_j).
 \end{align} 
Since $\gamma^k$ is not elliptic, the translation length $|\gamma^k|$ cannot be achieved in $\ucX$. It follows that $d_{\tg}(\tx_j, \tx_0) \to \infty$.

 Denote by $\pi: \ucX \to X^{n+1}$ the universal covering map. 
 For each $j$, let $\tilde{\sigma}_j$ be the unique minimizing geodesic segment connecting $\tx_j$ and $\gamma^k\cdot \tx_j$. Immediately, we have 
 that \begin{align}
 	\length(\tilde{\sigma}_j) = d_{\gamma}(\tx_j) \quad \text{and} \quad \lim\limits_{j\to \infty} \length(\tilde{\sigma}_j) = |\gamma^k|.
 \end{align}
On the other hand, we will now use the conformal compactification to show that $\length(\tilde{\sigma}_j) \to \infty$ as $j\to \infty$, which will give the desired contradiction.
 
To do so, we observe that 
   $x_j \equiv \pi(\tx_j)$ tends to the conformal infinity of $X^{n + 1}$ as $j \to \infty$. In fact, letting $x_0 \equiv \pi(\tx_0)$, since $\{\tx_j\}_{j = 1}^{\infty}\subset \Omega$, we have that
   \begin{align}
   	d_g (x_j, x_0) = d_{\tg}(\tx_j, \tx_0) \to \infty.
   \end{align}
 Next, for any $\delta > 0$, we consider a tubular neighborhood of the conformal infinity,
\begin{align}
	X_{\delta} \equiv \{p\in X^{n + 1}: s(p) \lesl  \delta \},
\end{align}
 where $s$ is the geodesic defining function of $(M^n,\wg)$ given by Lemma \ref{l:geodesic-defining-function}.
We denote by $\sigma_j$ the geodesic loop at $x_j=\pi(\tx_j)$ that descends from the geodesic segment $\tilde{\sigma}_j$.
   The previous argument implies that one can find a sequence $\delta_j \to 0$ such that $\sigma_j\subset X_{\delta_j}$ for any sufficiently large $j$. Notice that in terms of the geodesic defining function, the Poincar\'e-Einstein metric $g$ can be written as 
 \begin{align}
 	g = s^{-2} (ds^2 + \bg_s),
 \end{align} 
where $\bg_s = g_0 + O(s^2)$ as $s\to 0$. Thus, for any sufficiently large $j$, the injectivity radius of $X_{\delta_j}$ is much larger than $|\gamma^k|$ so that there exists no geodesic loops with bounded length in $X_{\delta_j}$. This gives the desired contradiction. 
 
 The above argument proved that every element in $\Gamma$
is hyperbolic, which implies that $\Gamma$ acts on $\ucX$ by semi-simple isometries. By assumption,  $X^{n + 1}$ is simply connected and $\sec_g \lesl 0$, we can now apply the Flat Torus Theorem (Theorem \ref{t:flat-torus}) 
to show that the minimum set $\Min(\Gamma)$ of $\Gamma$ is non-empty and isometrically splits as a product $Y \times \dR$ since $\rank(\Gamma) = 1$. 
Moreover, any element $\gamma\in\Gamma$ preserves the isometric splitting $Y \times \dR$ so that $\gamma$ acts as the identity on $Y$ and as a translation  with translation length equal to $|\gamma|$.

Since any deck transformation $\gamma\in\Gamma$
 acts on $Y \subset \ucX$ as the identity, we have $\pi(Y)\subset X^{n+1}$ is isometric to $Y $. This immediately implies that $\pi(Y)$ is compact in $X^{n+1}$. Otherwise, applying the Flat Strip Theorem (Theorem \ref{t:flat-strip}), $Y \times \dR$ is a convex flat region in $\ucX$. Then $\pi(Y \times \dR) = \pi(Y) \times S^1$ is an unbounded flat region in $X^{n+1}$, which contradicts the asymptotic hyperbolicity of $X^{n+1}$. Therefore, $\pi(Y)$ is compact, so is $Y$,
which completes the proof.
\end{proof}

\begin{remark} \label{r:line}
In the special case when the sectional curvature of $(X^{n + 1}, g)$ is {\it strictly negative}, one can quickly show that $\Min(\Gamma)$ is a single geodesic line. In fact, if $Y$ in the splitting $\Min(\Gamma) = Y\times \dR$ is not a single point, then the Flat Strip Theorem (theorem \ref{t:flat-strip}) implies that $Y \times \dR$ contains a flat strip, which contradicts the assumption that $\sec_g < 0$.
 In the general case of $\sec_g \lesl 0$, we will still establish that the minimum set $\Min(\Gamma)$ is a single geodesic line; see the argument in Proposition \ref{p:core-geodesic}. \end{remark}

 \medskip

\subsection{Core geodesic and topological splitting}
\label{ss:topological-splitting}
In Section \ref{ss:deck-transformation}, we discussed the behavior of the deck transformation group $\Gamma$ of the universal cover $\ucX$, and in particular established a structural result for the minimum set $\Min(\Gamma)$.

The goal of this subsection is to prove the topological rigidity stated in Theorem \ref{t:product-top-geo-rigidity}(1). Specifically, under Assumption \ref{a:pi-1}, the Poincar\'e-Einstein manifold $X^{n+1}$ is diffeomorphic to the topological product $S^1 \times D^n$, and its conformal infinity $M^n$ is diffeomorphic to $S^1 \times S^{n - 1}$. In the course of the proof, we will also show that the minimum set $\Min(\Gamma)$ in the universal cover $\ucX$ consists of a single geodesic line, referred to as the core geodesic. This geodesic line gives rise to a canonical global normal coordinate system.

 Recall that we have already established in Proposition \ref{p:minimum-set} that the minimum set $\Min(\Gamma)$ is isometric to a metric product $Y \times \dR$, where $Y$ is a compact, convex subset of $\ucX$. 
Now let us consider an arbitrary geodesic line  $\ell_0:\dR \to \ucX$ in the minimum set $\Min(\Gamma)$.
By Proposition \ref{p:minimum-set} we know that $\Gamma$ acts on $\ell_0$ by translations, that is, for any $\gamma\in \Gamma$, \begin{align}
	\gamma \cdot \ell_0(t) = \ell_0(t + |\gamma|),  \quad \forall t\in \dR, \label{e:gamma-translation}
\end{align}
where $|\gamma| \equiv \inf\{d_{\gamma}(\tilde{x}): \tilde{x} \in \ucX\}$ is the translation length of $\gamma$.
In what follows, we will  use the normal exponential map of $\ell_0$ on $\ucX$ to show that 
$Y = \{\pt\}$ and the minimum set $\Min(\Gamma)$
is reduced to a single minimal geodesic line. This geodesic line will be denoted as $\ell_{\Gamma}$ and called the {\it core geodesic}.
Then we 
 obtain a natural slicing structure that is $\Gamma$-equivariant and centered at the core geodesic $\ell_{\Gamma}$.

We denote by $\Pi^{\perp}: \mathcal{N}(\ell_0) \to \ell_0$ the normal bundle of the line $\ell_0$ in $\ucX$ such that for any $p =   \ell_0(t)$ and for any vector $v \in  (\Pi^{\perp})^{-1}(p) \subset  T_p\ucX$, we have   \begin{align}\langle v, \ell_0'(t) \rangle_{\tg} = 0.	
   \end{align}
 The normal exponential map of the geodesic line $\ell_0$, \begin{align}
 \nexp: \mathcal{N}(\ell_0) \to \ucX	
 \end{align}
is defined as follows: for any $\bm{w} = (v,p) \in  \mathcal{N}(\ell_0)$, we set
\begin{align}\nexp(\bm{w}) \equiv \sigma_{p,v}(1) \in \ucX,\end{align}
where $\sigma_{p,v}$ is the unique unit-speed normal geodesic with $\sigma_{p,v}(0) = p$ and $\sigma_{p,v}'(0) = v$.
For any $p \in \ell_0$, we denote by 
\begin{align}\mathcal{N}(p) \equiv \nexp\left((\Pi^{\perp})^{-1}(p)\right)  \subset \ucX \end{align} the image of the normal fiber $(\Pi^{\perp})^{-1}(p) \subset \mathcal{N}(\ell_0)$ under the normal exponential map.

Let $\ell_0$ be a geodesic line in the minimum set $\Min(\Gamma)$ of $\Gamma$. By Theorem \ref{t:flat-strip}, 
the image of $\ell_0$ under $\pi: \ucX \to X^{n+1}$ is a closed geodesic in $X^{n+1}$, which we denote by $\sigma_0$.

Using the normal exponential map of the geodesic line $\ell_0$, one can construct geodesic normal coordinates.
Let $\tir(\cdot) \equiv d_{\ell_0}(\cdot): \ucX \to [0,\infty)$ be the distance to $\ell_0$. Since $\ucX$ is simply connected and non-positively curved, the function $\tir$ is smooth away from the line $\ell_0$ and $|\nabla_{\tg} \tir| \equiv 1$ on $\ucX \setminus \ell_0$.
Since we have already established that the normal exponential map is $\Gamma$-equivariant, and the normal distance function $\tir$ is $\Gamma$-invariant. Therefore, $r \equiv d_{\sigma_0}(\cdot)$ the distance to the close geodesic $\sigma_0$ is smooth everywhere on $X^{n+1}\setminus \sigma_0$.

The proposition below reduces the minimum set 
$\Min(\Gamma)$ to a single geodesic line.
\begin{proposition}
[Core geodesic]	\label{p:core-geodesic}
In Proposition \ref{p:minimum-set}, the set $Y$ is a single point. Consequently,  $\Min(\Gamma)$ is a minimal geodesic line in $\ucX$, which is called the core geodesic of $\ucX$, and denoted by $\ell_{\Gamma}$. 
\end{proposition}

\begin{proof}

We fix a point $y \in Y\setminus \ell_0$ and consider the geodesic line $\ell_y$ the corresponds to $\{y\}\times \dR$ in the minimum set $Y \times \dR$. Then the flat strip $\overline{y_0 y} \times \dR$ extends to a complete surface $\Sigma_y$ in $\ucX$, where $\overline{y_0y}$ denotes the minimal geodesic connecting $y_0$ and $y$. 
In fact, for any point $t\in \dR$, we take \begin{align}\bm{w}(t) = (v(t), \ell_0(t)) \equiv (\nexp)^{-1} (\ell_y(t)) \in \cN(\ell_0).\end{align} The curve $\bm{w}(t)$ gives a natural complete surface $\mathscr{S}(r, t) \equiv (r \cdot v(t), \ell_0(t))$ in $\cN(\ell_0)$.
Then we define 
\begin{align}
\Sigma_y(r,t) \equiv \nexp(\mathscr{S}(r, t))
\end{align}
 is a complete surface in $\ucX$ such that the flat strip $\overline{y_0 y}\times \dR$ is given by the region
 \begin{align} \{0\lesl r \lesl r(y),\ -\infty  < t < + \infty\}.\end{align}

 By the analyticity of Einstein metric in the polar coordinates (see \cite[section 5]{DK}), one can apply the unique continuation to the flat strip $\overline{y_0 y}\times \dR \subset \Sigma_y$. In fact,
 the curvature of $\overline{y_0 y}\times \dR$ is constantly vanishing, the unique continuation implies that the complete surface $\Sigma_y$ is flat for all $r \gesl 0$. On the other hand, the asymptotic hyperbolicity of $(X^{n + 1}, g)$ implies that $\sec_g < 0$ when $r$ is sufficiently large, 
and thus the desired contradiction arises.
\end{proof}

We will now show that the universal cover $\ucX$ topologically splits along the core geodesic. Moreover, the topological splitting is realized via the normal exponential map and is invariant  under the $\Gamma$-action. 
  The following topological rigidity is the main result of this subsection.

  \begin{theorem}[$\Gamma$-equivariant splitting]\label{t:normal-exp} Under Assumption \ref{a:pi-1}, 
 the normal exponential map $\nExp: \mathcal{N}(\ellG) \to \ucX$ is a $\Gamma$-equivariant diffeomorphism (see Figure \ref{f:covering-space-splits}). In particular,  $X^{n+1}$ is diffeomrophic to $S^1 \times D^n$.	
\end{theorem}

\begin{figure}[h]
\label{f:covering-space-splits}
 \begin{tikzpicture}
 [scale= 0.32]

 \filldraw[fill=gray!25, draw=blue, densely dotted]
         (-8,0) ellipse (2 and 5);

\draw[blue, densely dotted] (-8,0) ellipse (2 and 5);

\draw[thick, blue] (-8,0) [partial ellipse=90:270:2 and 5];

\draw[thick] (-10,5) to (-10,5);

 \filldraw[fill=gray!25, draw=blue, densely dotted]
         (0,0) ellipse (2 and 5);

\draw[blue, densely dotted] (0,0) ellipse (2 and 5);

\draw[thick, blue] (0,0) [partial ellipse=90:270:2 and 5];

\draw[red, thick] (-15,0) to (-10,0);

\draw[red, thick, densely dotted] (-10,0) to (-8,0);

\draw[red, thick] (-8,0) to (-2,0);

\draw[red, thick, densely dotted] (-2,0) to (0,0);

\draw[red, thick] (0,0) to (15,0);

\draw[thick] (-15,5) to (15,5);

\draw[thick] (-15,-5) to (15,-5);

\node at (-8, -1) {$\tilde{p}$};

\node at (-8, 0) {\textcolor{red}{$\bullet$}};

\node at (0, -1) {$\gamma\cdot \tilde{p}$};

\node at (0, 0) {\textcolor{red}{$\bullet$}};

\node at (-8, -7) {$\mathcal{N}(\tilde{p}) $};

\node at (0, -7) {$\mathcal{N}(\gamma\cdot\tilde{p}) $};

\node at (9.5, -1) {\textcolor{red}{$\ell_{\Gamma}$}};

\end{tikzpicture}	

\caption{$\ucX\approx \dR \times D^n$; $\ell_0 = \{r = 0 \}$}
\end{figure}

\begin{proof} By assumption, the universal cover $\ucX$ is non-negatively curved and simply connected. To prove the proposition, we will first establish that the normal exponential map $\nexp: \mathcal{N}(\ell_0) \to \ucX$ is a diffeomorphism. Since $\nexp$ is a local diffeomorphism, it suffices to show that $\nexp$ is a bijection from $\mathcal{N}(\ell_0)$ to $\ucX$. That is, for any $\tilde{q} \in \ucX$, there exists 
a unique element $\bm{w}$ in the normal bundle $\mathcal{N}(\ell_0)$ such that 
\begin{align}\nexp(\bm{w}) = \tilde{q}.\end{align}
 In fact, due to the obvious convexity of the geodesic line $\ell_0$,  for every $\tilde{q} \in \ucX$, there exists a unique point $\tilde{p} \in \ell_0$ that satisfies $d_{\tilde{g}}(\tilde{q}, \ell_0) = d_{\tilde{g}}(\tilde{q}, \tilde{p})$. Let $r(\tilde{q}) \equiv d_{\tilde{g}}(\tilde{q}, \ell_0)$ and let $\sigma:[0,r]\to \ucX$ be the unique geodesic connecting $\tilde{q}$ and $\tilde{p}$. Since $\ucX$ is non-negatively curved and simply connected, one find a unique vector $v$ in the normal space  $(\Pi^{\perp})^{-1}(\tilde{p}) \subset T_{\tilde{p}}\ucX$ that satisfies $\exp_{\tilde{p}}(v) = \tilde{q}$, where $\exp_{\tilde{p}}$ is the exponential map of $\ucX$ at $\tilde{p}$.
Then
 $\bm{w} = (v,\tilde{p})$ is the unique element in 
 $\mathcal{N}(\ell_0)$  that satisfies  
\begin{align}  \nexp(\bm{w}) = \tilde{q},\end{align}
which verifies that $\nexp: \cN(\ell_0) \to \ucX$ is a diffeomorphism. In particular, the normal exponential map $\nexp$ realizes a natural topological splitting $\ucX \overset{\text{diffeo}}{\approx}\dR \times D^n$ along the geodesic line $\ell_0$.

Next, we will prove that the topological splitting of $\ucX$ is preserved by the deck transformation group $\Gamma$, which will imply that $X^{n+1}$ is diffeomorphic to $S^1 \times D^n$.
So we are led to show that the normal exponential map $\nexp$ is $\Gamma$-equivariant. That is, any deck transformation $\gamma\in\Gamma$ translates normal fibers of $\nexp$ and lifts to an isometry $\bar{\gamma}:\mathcal{N}(\ell_0) \to \mathcal{N}(\ell_0)$ such that the following diagram commutes.
\begin{equation}
 	\begin{tikzcd}
\mathcal{N}(\ell_0) \arrow{r}{\bar{\gamma}} \arrow[swap]{d}{\nexp} & \mathcal{N}(\ell_0) \arrow{d}{\nexp}  \\
 \ucX \arrow[swap]{r}{\gamma} & \ucX \end{tikzcd} 	
\end{equation}

Note that any element $\gamma\in\Gamma$ maps a normal geodesic (in a normal geodesic ball centered at $\p = \ell_0(t)$) to a normal geodesic (in another normal geodesic ball centered at $\gamma\cdot p = \ell_0 (t + t_{\gamma})$. This is essentially because $\gamma$ acts on $\ell_0$ by translation.

We claim that, for any point $\tilde{p}\in \ell_0$ and for any deck transformation $\gamma \in \Gamma$, the normal fiber $\cN(\tilde{p})$ is diffeomorphically mapped onto $\cN(\gamma\cdot \tilde{p})$ with $\gamma\cdot \tilde{p} \in \ell_0$.
In fact, let $\tilde{q} \in \mathcal{N}(\tilde{p})$ be any point in the normal fiber of $\tilde{p} = \ell_0(t)$ and let $\sigma$ be the unique geodesic connecting $\tilde{p}$ and $\tilde{q}$. Then the image $\gamma \cdot \sigma$  
under the isometry $\gamma \in \Gamma$ becomes the unique geodesic
connecting $\gamma\cdot \tilde{q}$ and $\gamma\cdot \tilde{p}$. We also notice that $\LB \sigma', \ell_0'\RB_{\tilde{g}} = 0$ at $\tilde{p}$, which implies that 
$\LB D\gamma(\sigma'), D\gamma(\ell_0') \RB_{\tg} = 0$ at $\gamma\cdot\tilde{p}$.  By Proposition \ref{p:minimum-set}, for $\gamma\in\Gamma$, there exists a constant $t_{\gamma} \in \dR$ such that
\begin{align}\gamma \cdot \tilde{p} = \gamma \cdot \ell_0(t) = \ell_0(t + \alpha_{\gamma})\ \text{and}\ D\gamma(\ell_0'(t)) = \ell_0'(t + t_\gamma).\end{align} 
So it follows that 
\begin{align}
\LB (\gamma \cdot \sigma)', \ell_0'  \RB_{\tg} = \LB D\gamma(\sigma'), D\gamma(\ell_0')\RB = 0\ \text{at}\ \gamma\cdot\tilde{p},
\end{align}
which implies that $\gamma\cdot\sigma$ is a normal geodesic from $\gamma\cdot \tilde{p}$. Thus, $\gamma\cdot \tilde{q}$ is contained in the normal fiber $\mathcal{N}(\gamma\cdot \tilde{p})$ with $\gamma \cdot \tilde{p} = \ell_0(t + \alpha_{\gamma})$, which completes the proof of the claim.

Finally, since we have proved that $\nexp$ is $\Gamma$-equivariant and any isometry $\gamma\in\Gamma$ acts on $\ucX$ by translating normal fibers over $\ell_0$, we conclude that the topological splitting of $\ucX$ descends to $X^{n+1}$ so that $X^{n+1}$ is diffeomorphic to $S^1 \times D^n$.
\end{proof}

\medskip

\subsection{Extensions of boundary Killing fields} 
\label{ss:extension-of-Killing-fields}

In the proof of the rigidity result in Theorem \ref{t:uniqueness-nonpositively-curved}, our basic strategy is to perform {\it dimension reduction} for the Poincar\'e-Einstein filling $(X^{n + 1}, g)$. 
We remark that in recent works \cite{Li, Li-2, Li-3}, G. Li has proved several interesting uniqueness results when the conformal infinity $(S^n, [\wg])$ is equipped with homogeneous metrics $\wg$ that are sufficiently close to the round metric on $S^n$.

To achieve the goal of dimension reduction, we will construct sufficiently many continuous symmetries, namely Killing vector fields on $(X^{n + 1}, g)$. Given a complete Riemannian manifold $(X^d, g)$,  let $\Isom(X^d)$ denote the group of isometries on $X^d$. We also denote by $\isom(X^d)$
the Lie algebra of the Killing vector fields on $X^d$, which coincides with the Lie algebra of $\Isom(X^d)$. 

We remark that the results in this subsection hold in the general setting of non-positively curved Poincar\'e-Einstein manifolds (without requiring Assumption \ref{a:pi-1} or \ref{a:product-conformal-infinity}).
To begin with, let us review some general results on the extension of boundary Killing fields.
First, we recall that it was proved in earlier works 
\cite[Section 4]{Wang}
and \cite[theorem 3.4]{Li}
that any Killing field on the conformal infinity $M^n$ always extends to a global Killing field on the Poincar\'e-Einstein filling $X^{n + 1}$, provided that $(X^{n + 1}, g)$ has non-positive sectional curvature. Namely:

\begin{theorem}\label{t:Killing-field-extension}
Let $(X^{n + 1}, g)$ be a Poincar\'e-Einstein manifold with $\sec_g \lesl 0$. Let $(M^n, [\wg])$ be the conformal infinity of $(X^{n + 1}, g)$. If $\wx$ is a Killing field on $(M^n, \wg)$, then there exists a Killing field $\xi$ on $X^{n + 1}$ such that $\xi|_{M^n} = \wx$.
\end{theorem}

Next, we establish several basic properties concerning the extension of Killing fields, which will be essential in   constructing the global normal coordinates. These results hold for general Poincar\'e-Einstein manifolds.

 Let $(X^{n + 1}, g)$ be a Poincar\'e-Einstein manifold with conformal infinity $(M^n, [\wg])$. 
Applying Lemma \ref{l:geodesic-defining-function}, one can take the geodesic defining function $x$ associated with $(M^n, \wg)$, such that for some number  $\epsilon_0 > 0$, the following holds:
\begin{align}
\begin{cases}
	|dx|_{x^2g} \equiv 1, &    0\lesl x <  \epsilon_0,
	\\
x^2 g|_{x = 0} = \wg.	
\end{cases}
\end{align} 
    Let $(U, \bm{u})$ be a coordinate neighborhood in $M^n$ with coordinate functions $\bm{u} = (u_1, \ldots, u_n)$.  Then the Poincar\'e-Einstein metric $g$ can be written on $U \times [0, \epsilon_0)$ as follows,
 \begin{align}
	g = \frac{dx^2 + \hbar}{x^2} = \frac{dx^2 + h_{ij}(x,\bm{u}) d u_i d u_j}{x^2}, \quad 0\lesl x < \epsilon_0.  \label{e:vertical-horizontal}
\end{align}

The following lemma turns out to be particularly useful, and its proof is purely computational. 
\begin{lemma}\label{l:Killing-field-x-independence} Let $(X^{n + 1}, g)$ be a Poincar\'e-Einstein manifold with conformal infinity $(M^n , [\wg])$. For a given representative $\wg \in [\wg]$, let $x$ be the geodesic defining function of $(M^n, \wg)$.	If a Killing field $\xi \in \isom(X^{n + 1}, g)$ is an extension of some Killing field $\hat{\xi} \in \isom(M^n , \wg)$ such that $\xi|_{M^n} = \hat{\xi}$, then restricted to $U \times [0, \epsilon_0)$, $\xi$ is independent of  $x$ and can be written as \begin{align}
		\xi = \xi(\bm{u}) =  \sum\limits_{j = 1}^n B_j(\bm{u}) \p_{u_j},
	\end{align}
	 where the coefficient functions $B_j(\bm{u})$ depend only on $\bm{u}$.
\end{lemma}

\begin{proof}
	 Let  $\xi \in \fq$ be a Killing field on $X^{n + 1}$ and we write it in local coordinates 
\begin{align}
	\xi(x,  \bm{u}) = A(x,  \bm{u}) \p_x +\sum\limits_{j = 1}^n B_j(x, \bm{u}) \p_{u_j}.
\end{align}
By assumption, the boundary restriction $\hat{\xi}(\bm{u}) \equiv \xi(0,\bm{u})$ is a Killing field on $(M^n, \wg)$. The boundary data are denoted by 
\begin{align}\alpha(\bm{u}) \equiv A(0, \bm{u}) \quad \text{and} \quad \beta_j(\bm{u}) \equiv B_j(0, \bm{u}), \ 1\lesl j \lesl n .\end{align}

Next, we express the equation $\mathcal{L}_{\xi} g \equiv 0$ in coordinates, and we compute below the Lie derivatives. Along the normal direction $\p_x$, the Lie derivative reads 
\begin{align}
		  (\mathcal{L}_{\xi} g)(\p_x, \p_x) = \frac{2}{x^2}\left(\frac{\p A}{\p x} - \frac{A}{x} \right). 	
		\end{align}
		For any $1\lesl i,j \lesl n $, 
		\begin{align}
	 	(\mathcal{L}_{\xi} g)(\p_x, \p_{u_i}) = \frac{1}{x^2}\left(\frac{\p A}{\p u_i} + \sum\limits_{j = 1}^{ n } h_{ij} \cdot \frac{\p B_j}{\p x}\right), 
	\end{align}
	and 
	\begin{align}
	\begin{split}
			  (\mathcal{L}_{\xi} g)(\p_{u_i}, \p_{u_j}) = &\  \frac{A}{x^2}\left( \frac{\p h_{ij}}{\p x} - \frac{2h_{ij}}{x}\right)
	  \\
	& \   +  \sum\limits_{k,\ell = 1}^n \left( \left(B_k  \Gamma_{ki}^{\ell}	+ \frac{\p B_{\ell}}{\p w_i}\right) \frac{h_{\ell j}}{x^2}
 +\left(B_k  \Gamma_{kj}^{\ell} + \frac{\p B_{\ell}}{\p w_j}\right) \frac{h_{\ell i}}{x^2} \right),
 	\end{split}
 \end{align}
where $\Gamma_{ij}^k$'s are the Christoffel symbols of the horizontal slice $(\{x = \epsilon\}, \hbar)$ as in \eqref{e:vertical-horizontal}, where $\epsilon \in [0,\epsilon_0)$.
The identity $\mathcal{L}_{\xi} g \equiv 0$ gives rise to three equations: 
\begin{align}
&(\mathcal{L}_{\xi} g)(\p_x, \p_x) = 0,\label{e:radial-Lie-D}
		\\
&	(\mathcal{L}_{\xi} g)(\p_x, \p_{u_i}) = 0,\label{e:mixed-Lie-D}
	\\
&(\mathcal{L}_{\xi} g)(\p_{u_i}, \p_{u_j}) = 0.	\label{e:tangential-Lie-D}
\end{align}

From \eqref{e:radial-Lie-D}, we have that
\begin{align}A(x, \bm{u}) = x \cdot \alpha (\bm{u}). \label{e:A-x}\end{align} 
Since $\hat{\xi}$ is a Killing field on $(M^n , \wg)$,
we have that
\begin{align}
0 \equiv	(\mathcal{L}_{\hat{\xi}} \wg) (\p_{u_i}, \p_{u_j}) =   \sum\limits_{k,\ell = 1}^n  \left( \left(\beta_k  \widehat{\Gamma}_{ki}^{\ell} + \frac{\p \beta_{\ell}}{\p u_i}\right) \wg_{\ell j}
	+ \left(\beta_k  \widehat{\Gamma}_{kj}^{\ell} + \frac{\p \beta_{\ell}}{\p u_j}\right) \wg_{\ell i}\right),
\end{align}
where $\widehat{\Gamma}_{ij}^k$'s are the Christoffel symbols with respect to $\wg$.
 Recall that, in terms of geodesic defining function, the boundary $M^n \equiv \{x = 0\}$ is totally geodesic in $(\overline{X^{n + 1}}, x^2g)$, which implies that $\p_x h_{ij}(0, \bm{u}) = 0$.
Therefore, plugging \eqref{e:A-x} into \eqref{e:tangential-Lie-D} and letting $x\to 0$, we have that 
$-2\alpha(\bm{u}) h_{ij} \equiv 0$. This further implies that $\alpha(\bm{u}) \equiv 0$ on $M^n$. 
By \eqref{e:A-x}, 
\begin{align}A(x, \bm{u}) \equiv 0, \quad
\forall\ x\in[0,\epsilon). \label{e:A-vanishing}\end{align}
Next, applying \eqref{e:A-vanishing} and
\eqref{e:mixed-Lie-D}, we obtain that for each $1 \lesl j \lesl n$, 
\begin{align}
	B_j(x, \bm{u}) \equiv \beta_j(\bm{u}), \quad \forall\ x\in[0,\epsilon).
	\end{align}
Therefore, the Killing field $\xi\in \isom(X^{n + 1}, g)$ that extends from a boundary Killing field can be written as
\begin{align}
	\xi = \xi(\bm{u}) = \sum\limits_{j = 1}^n \beta_j(\bm{u}) \p_{u_j},\quad \text{in} \ U \times [0,\epsilon) \subset X^{n + 1},
\end{align}
where $(U, \bm{u})$ is a coordinate neighborhood of $M^n$. This completes the proof.
\end{proof}

 Combining Theorem \ref{t:Killing-field-extension} and Lemma \ref{l:Killing-field-x-independence}, we reach the global uniqueness of the extension as follows.
\begin{corollary}
Let $(X^{n + 1}, g)$ be a Poincar\'e-Einstein manifold with $\sec_g \lesl 0$. Let $(M^n, [\wg])$ be the conformal infinity of $(X^{n + 1}, g)$. Any Killing field $\wx$ on $(M^n, \wg)$ extends uniquely to a Killing field $\xi$ on $X^{n + 1}$.	
\end{corollary}

\begin{proof}
Let $\xi_1,\xi_2\in \isom(X^{n + 1}, g)$ be two extensions of the Killing field $\wx$.  For $(M^n, \wg)$, let $x$ be the geodesic defining function. By Lemma \ref{l:Killing-field-x-independence}, $\xi_1 - \xi_2 \equiv 0$ in a neighborhood $[0,\epsilon_0) \times M^n$ of the boundary $(M^n, \wg)$. 
For any point $q\in X^{n + 1}$ and $q'\in [0,\epsilon_0) \times M^n$, let $\gamma$ be a geodesic connecting $q$ and $q'$. Then $\xi_1 - \xi_2$ is a Jacobi field along $\gamma$. By the standard uniqueness result for linear ODEs, $(\xi_1 - \xi_2) \equiv 0$ along $\gamma$, which proves the uniqueness of the extension. 
\end{proof}

Combining Theorem \ref{t:Killing-field-extension} and Lemma \ref{l:Killing-field-x-independence}, one can show that linearly independent Killing fields on the boundary extend to Killing fields that are linearly independent near the boundary.  \begin{proposition}\label{p:extension-relations}
	 Let $U \subset M^n$ be an open set. If  
$\{\wx_j\}_{j = 1}^k  \subset \isom(M^n, \wg)$  
is linear independent set in $T_q M^n$ for any $q \in U \subset M^n$, then 
these Killing fields admit unique extensions $\{\xi_j\}_{j = 1}^k \subset \isom(X^{n + 1}, g)$ with the following properties: 
\begin{enumerate}
	\item $\xi_j|_{M^n} = \wx_j$ for any $1\lesl j\lesl n$; 
	\item  $[\xi_i, \xi_j]|_{M^n}= [\wx_i, \wx_j]$ for any $1\lesl i,j\lesl n$;

	\item $\xi_j\perp\p_x$ for each $1\lesl j \lesl k$ , and $\{\xi_j\}_{j = 1}^k$ is a linear independent set in $T_{q'}(X^{n + 1})$ for all  $q'\in  U \times [0,\epsilon_0)$, where $\epsilon_0$ is sufficiently small. 

\end{enumerate} %In particular, $\{\xi_1, \ldots, \xi_n\}$ gives a basis of the horizontal tangent space $T_{q'} (\mathcal{U}_{\epsilon} \cap \{x = \mu \})$ for any $0\lesl  \mu <\epsilon$.
\end{proposition}

\begin{proof}
The existence of the extensions in item (1) is given by Theorem \ref{t:Killing-field-extension},  and the properties in items (2) and (3) follow from Lemma \ref{l:Killing-field-x-independence}.
\end{proof}

\subsection{Invariant normal distance function}  \label{ss:invariant-normal-dist}
In this subsection, we will consider a Poincar\'e-Einstein manifold $(X^{n + 1}, g)$ with non-positive sectional curvature that satisfies Assumption \ref{a:product-conformal-infinity}, namely, the conformal infinity is given by $(S^1 \times S^{n - 1}, [\wg_{\lambda}])$.
Let $\sG$ be the core geodesic loop in $X^{n + 1}$. The main result of this subsection (Proposition \ref{p:r-x-relation}) states that the normal distance function $d_{\sG}$ and the geodesic defining function $x$ with respect to $(S^1 \times S^{n - 1}, \wg_{\lambda})$, satisfy an explicit relation, and that both functions are invariant under the continuous isometries on $X^{n + 1}$ extended from $S^1 \times S^{n - 1}$. 
Our proof below is based on 
the phenomenon that the homogeneity of the conformal infinity extends explicitly to the interior.

For the Riemannian product $(S^1 \times S^{n - 1}, \wg_{\lambda})$, let us
 denote by $\ft_b \equiv \isom(S^1, \wg_{\lambda})$ and $\fq_b \equiv \isom(S^{n - 1}, \wg_{\lambda})$ the Lie algebras of the Killing fields on the two factors, respectively. Then  $\isom(S^{n - 1} \times S^1) = \fq_b \oplus \ft_b \cong \fs \fo (n) \oplus \dR$.
Obviously, $[\ft_b, \fq_b] = 0$.
 We also define \begin{align}
	\fq \equiv \left\{ \xi \in \isom(X^{n + 1}): \left.\xi\right|_{M^n} \in \fq_b \right\} 
	\quad \text{and}\quad 
	\ft \equiv \left\{ \xi \in \isom(X^{n + 1}): \left.\xi\right|_{M^n} \in \ft_b \right\}.\label{e:q-t-extensions-to-X}
\end{align} 
 The groups of isometries generated by these Lie algebras $\fq$ and $\ft$ are denoted by $\sQ$ and $\sT$, respectively. The liftings of $\sQ$ and $\sT$ on the universal cover $\ucX$ are denoted by $\wsQ$ and $\wsT$, respectively.

Applying Proposition \ref{p:extension-relations}, one   immediately obtains the following properties for $\ft$ and $\fq$.
 \begin{lemma}\label{l:isomorphic-factors}
$\ft$ is isomorphic to $\ft_b$ and $\fq$ is isomorphic to $\fq_b$. In particular, $[\ft, \fq]  = 0$.  \end{lemma}
  
Next, we will give a geometric representation of the generator of $\Gamma \equiv \pi_1(X^{n + 1})$ in terms of a Killing field in $\ft_b$. 
   It is clear that for any Killing field in $\ft_b \equiv \isom(S^1)$, its orbits are closed geodesics in $S^1 \times S^{n - 1}$.  
Let $\widehat{\Gamma}\equiv \pi_1(S^1 \times S^{n-1})$. Given any point $q\in S^1 \times S^{n - 1}$, let $\hat{\gamma}$ be the closed geodesic passing through $q$ that  represents    
 the generator of $\widehat{\Gamma}$. 
Then $\hat{\gamma}$ corresponds to the 
deck transformation of the universal cover $\dR \times S^{n - 1}$ that can be {\it uniquely} written as $\exp(\widehat{T}_{\lambda}) \in \Isom(\dR \times S^{n - 1})$, where $\widehat{T}_{\lambda}\in \ft_b$ and $\|\widehat{T}_{\lambda}\|_{\wg_{\lambda}}=\lambda$.

\begin{lemma}\label{l:geometric-representation-of-Gamma} Let $T_{\lambda}\in \isom(X^{n + 1})$ be the extension of $\widehat{T}_{\lambda}$ and let $\widehat{\mathscr{T}}$ be the isometry group generated by $\ft_b$. Then,  $\Gamma$ is a subgroup of $\widetilde{\mathscr{T}}$, where 
$\widetilde{\mathscr{T}} \equiv \exp(\tilde{\ft}_b)\subset \Isom(\ucX)$ and $\tilde{\ft}_b$ is the lifting of $\ft_b$ on $\ucX$.
\end{lemma}

\begin{proof} By our choice of $\widehat{T}_{\lambda}$, for any fixed $\hat{q}\in S^1\times S^{n - 1}$, the orbit of $\widehat{T}_{\lambda}$ at $\hat{q}$ is a closed geodesic $\hat{\gamma}$ that generates $\widehat{\Gamma}$, which also gives a deck transformation $\exp(\widehat{T}_{\lambda})$ on the universal cover.   
By Lemma \ref{l:Killing-field-x-independence}, there exists some $\epsilon_0 > 0$ such that, in a small collar neighborhood $M^n \times [0,\epsilon_0)$, the extended Killing field $T_{\lambda}$ is independent of the geodesic coordinate $x$. This implies that, for any point $q_{\epsilon} \in M^n \times [0,\epsilon)$ with $x(q_{\epsilon}) = \epsilon < \epsilon_0$, the orbit $\gamma_{q_{\epsilon}}$ of $T_{\lambda}$ at $q_{\epsilon}$ is closed and homotopic to $\hat{\gamma}$. 
Since $M^n \times [0,\epsilon]$ contains an open set of $X^{n + 1}$, for   any $q\in X^{n + 1}$, 
the Killing field $T_{\lambda}$ has a closed orbit $\gamma_q$ that is homotopic to $\hat{\gamma}$. 
By Lemma \ref{l:cyclic-fundamental-group}, there is a natural isomorphism from $\pi_1(S^1 \times S^{n - 1})$ onto $\Gamma$, which implies that 
any orbit $\gamma_q$ is non-contractible in $X^{n + 1}$.

Denote by $\pi:(\ucX, \tilde{q}) \to (X^{n + 1}, q)$ the universal covering map  and denote by the lifting of $\gamma_q$ in $\ucX$. Then the non-contractibility of $\gamma_q$ implies that $\tilde{\gamma}_{\tilde{q}}$ is diffeomorphic to $\dR$ and $\tilde{\gamma}_{\tilde{q}}$ is an orbit of the lifted Killing field $\widetilde{T}_{\lambda}$. 
Moreover, $\exp(\widetilde{T}_{\lambda})$ gives a deck transformation of $\ucX$, where $\exp$ is the exponential map from $\isom(\ucX)$ to $\Isom(\ucX)$.
Finally, $\exp(\widetilde{T}_{\lambda})\in\Isom(\ucX)$ is a generator of $\Gamma$ since $\Gamma$ is isomorphic to $\widehat{\Gamma}$ and $\exp(T_{\lambda})$ generates $\widehat{\Gamma}$.  
\end{proof}

The above geometric representation of $\Gamma$
implies the following very useful result for the core geodesic. 

\begin{lemma}\label{l:invariance-of-core-geodesic} The core geodesic line 
$\ellG$ in $\ucX$ is invariant under the actions of
$\widetilde{\sT}$ and $\widetilde{\sQ}$, respectively. Moreover, any element in
$\widetilde{\sT}$ acts on $\ellG$ as a translation, and any element in $\widetilde{\sQ}$ acts on $\ellG$ as the identity. 
 Consequently, the core geodesic loop $\sG$ in $X^{n + 1}$, as the $\Gamma$-quotient of $\ellG$, satisfies the same invariance property.
\end{lemma} 

\begin{proof}
Since $\widetilde{\mathscr{T}}$ is isomorphic to $\dR$, by Lemma \ref{l:geometric-representation-of-Gamma}, any element $\Gamma$ commutes with any element in $\widetilde{\mathscr{T}}$.
By Lemma \ref{l:invariance-of-minimum-set}, $\widetilde{\mathscr{T}}$ leaves $\ell_{\Gamma}$ invariant, and thus any element in $\widetilde{\mathscr{T}}$ acts on $\ellG$  as a translation. 
  
The local $x$-independence of Killing fields (Lemma \ref{l:Killing-field-x-independence}) implies that $[\ft, \fq] \equiv 0$, and thus all any element in $\widetilde{\mathscr{Q}}$ commutes with any element in $\widetilde{\mathscr{T}}$. 
Then $\ell_{\Gamma}$ is invariant under all elements in $\widetilde{\mathscr{Q}}$. Since any $\widetilde{\mathscr{Q}}$-orbit is closed,  it follows that any element $\widetilde{\mathscr{Q}}$ acts on $\ellG$ the identity.
\end{proof}

By Lemma \ref{l:geodesic-defining-function}, the geodesic defining function $x$ of $M^n$ exists only in a small tubular neighborhood of the conformal infinity such that $x$ is only well defined for $0\lesl x \lesl \epsilon$. In general, the size of $\epsilon$ is not uniformly controlled, which
prevents it to be an effective coordinate system. 

The following proposition proves that $x$ is well-defined globally on $\bX$ under the additional Assumption \ref{a:product-conformal-infinity}, which makes it an effective coordinate system for us to carry out the dimension reduction argument towards the proof of the rigidity.

\begin{proposition}
	\label{p:r-x-relation} Let $(X^{n + 1}, g)$ be a Poincar\'e-Einstein manifold with non-positive sectional curvature that satisfies Assumption \ref{a:product-conformal-infinity}. Let $r(\cdot) \equiv d_{\sigma_{\Gamma}}(\cdot)$ be the distance to the core geodesic $\sigma_{\Gamma}$. Then there exists a constant $\varpi_0 > 0$ such that $x = \varpi_0\cdot e^{-r}$ provides a smooth compactification $(\bX, x^2 g)$ with the following property:
\begin{align}\label{e:global-geo-x}
	\begin{cases}
		|dx|_{x^2 g} = 1 & \text{in} \ \overline{X^{n + 1}} \equiv \{0\lesl x\lesl \varpi_0\},
		\\
x^2 g|_{S^1 \times S^{n - 1}} = \wg_{\lambda} & \text{on} \ S^1 \times S^{n - 1} \equiv \{x = 0\},
	\end{cases}
\end{align}
  In particular, the level set $\{ x= \varpi_0\}$ corresponds to the core geodesic $\sigma_{\Gamma}$.
\end{proposition}

\begin{proof}
The relation between $r$ and $x$ is provided by  a standard uniqueness result for non-degenerate non-characteristic first-order PDE.

By Assumption \ref{a:product-conformal-infinity}, the conformal infinity $(M^n , [\wg])$ is given by the standard product $(S^1 \times S^{n - 1} , [\wg_{\lambda}])$.
We now consider the conformal metric $g^{\dag}\equiv \ur^2 g$, where 
\begin{align}
	\hat{\rho}\equiv e^{-r}.
\end{align}
We also write
\begin{align}
g^{\dag} = \frac{\hat{\rho}^2}{x^2} \cdot x^2 g = e^{2(\sX - r)} \cdot (x^2g),
\end{align}
where $\sX \equiv - \log x$. By \cite[lemma 4.1]{LQS}, there exists some constant $C_0 > 0$ such that 
 	$|\nabla_{\bg} (\sX - r)|_{\bg} \lesl C_0$ near the boundary. 
This uniform gradient estimate implies that $g^{\dag}$ gives a compactification of $(X^{n + 1}, g)$ that extends to a {\it  Lipschitz metric} on $\bX$. In particular, $g^{\dag}|_{M^n}$ is a Lipschitz metric on the boundary $M^n$. To apply the PDE uniqueness, we will further verify the following:

\medskip

\begin{enumerate}[(a)]
	\item $g^{\dag}$ is a smooth metric on $\bX$;
	\item there exists some $\varpi_0 > 0$ such that $g^{\dag}|_{M^n} = \varpi_0^{-2} \cdot \wg_{\lambda}$. 
\end{enumerate}  

\medskip 

\noindent 
If items (a) and (b) are verified, then we obtain the desired global normal coordinates.
To see so, we consider the boundary value problem
\begin{align}\label{e:non-characteristic}
\begin{cases}
	|d \Phi|_{\Phi^2 g} \equiv 1 \quad & \text{in}\ \bX,
\\
\Phi > 0  & \text{in}\ X^{n+1}, 
\\
\Phi = 0\ \text{and} \ \Phi^2g|_{M^n} = \varpi_0^{-2} \cdot \widehat{g}_{\lambda} 
  & \text{on} \ M^n,
 \end{cases}
\end{align} 
 By (a) and (b), the smooth compactification factor
$\ur$ becomes a solution of \eqref{e:non-characteristic}. On the other hand, $\varpi_0^{-1} \cdot x$ also solves  \eqref{e:non-characteristic}. Finally, by the uniqueness of non-degenerate non-characteristic first order PDE, we conclude that $x = \varpi_0\cdot e^{-r}$ for all $r \gesl 0$, or equivalently for all $0 \lesl  x \lesl  \varpi_0$.	

\medskip

We now proceed to verify items (a) and (b) modulo the proof of the following Claim. 

\medskip

\noindent {\bf Claim.} There exists a sufficiently small $\epsilon_0 > 0$ such that any level set 
$\Sigma(x  , \epsilon) \equiv \{x   = \epsilon\}$ of the geodesic defining function $x$ with $\epsilon < \epsilon_0$ is orthogonal to $\p_r$.
In particular, $\ur$ and $x$ share the same collection of level sets near the boundary $S^1 \times S^{n - 1}$.

\bigskip

%First, we will prove item (a) by verifying that the function $\hat{\rho} = e^{-r}$ is $C^2$ up to $M^n$. This follows from the regularity of $r$ and the non-degeneracy of $\nabla_g r$ in $\cV$. The arguments are the same as the computations in \cite[theorem 3.4]{Li}. Here we only sketch the proof. Let us take  $\nu \equiv -x \cdot \frac{\p}{\p x}$. It was shown in \cite[lemma 4.1]{LQS} that  \begin{align} \LB \p_r, \nu \RB_g = 1 + O(e^{-2r}) \quad \text{as}\ r \to \infty.	 \label{e:r-nu-asymptotics} \end{align} 

   %let us take the inward unit vector field $\zeta$ that is parallel to $-\p_r$ and satisfies $|\zeta|_{x^2 g} \equiv 1$ in $\mathcal{V}$.
% the level sets $\{x = \epsilon\}$. Since $\{x = \epsilon\}$ is a family of smooth hypersurfaces in $\bX$, $\zeta$ is a smooth vector field in $\mathcal{V}$.  For any fixed point $p\in M^n$, we take the flow line $\gamma_p:[0,\mu_0]\to \bX$ of the unit vector field $\zeta$ in $\cV$ with $\gamma_p(0) = p\in M^n$. We define a smooth function $\psi$ such that each level set $\{\psi = t\}$ coincides with the level set $\{x = \epsilon(t)\}$ of $x$ that contains $\gamma_p(t)$.

To show item (a), notice that $|\p_r|_g = 1$ in $X^{n + 1}$ and $|\p_x|_{x^2 g} = 1$ in an neighborhood $\{0\lesl x < \epsilon_0\}$ of $M^n$.   By the statement in the Claim, we have that $\p_x$ is parallel to $\p_r$. So it follows that \begin{align}
	x \cdot \frac{d r}{d x} = -1, \quad \forall 
	\ 0\lesl x < \epsilon_0.
\end{align}
Then $\hat{\rho} = e^{-r} = C x$ for some constant $C > 0$, which immediately implies that $\hat{g}$ is smooth up to the boundary $M^n$.

Next, we will verify item (b). It follows from item (a) that 
the Lipschitz compactification $g^{\dag}$ turns out to be smooth  up to the boundary $M^n$.
 By the Claim, near the boundary, each level set $\Sigma(\hat{\rho}, \epsilon)$ of $\ur$ coincides with some level set 
 	$\Sigma(x, \epsilon')$
 	of the geodesic defining function $x$. In particular, $\Sigma(\hat{\rho}, \epsilon)$ is a Riemanninan homogeneous space with respect to both the Einstein $g$ and the {\it  rescaled metric} $g^{\dag}$.
 	Letting $\epsilon\to 0$, by the regularity in item (a), $g^{\dag}|_{\Sigma(\hat{\rho}, \epsilon)}$ converges to a smooth metric $g^{\dag}|_{M^n}$. Notice that 
 $g^{\dag}|_{M^n}$ is a smooth metric in the conformal class $[\wg]$ on $M^n$, and $g^{\dag}|_{M^n}$ is homogeneous with the same Killing fields as $\wg$. We can easily conclude that $g^{\dag}|_{M^n}  = \varpi_0^2 \cdot \wg$ for some constant $\varpi_0 > 0$.
Thus, we establish item (b).

\medskip

\noindent {\it Proof of the Claim.}
For any level set $\Sigma(x, \epsilon)$ of $x$ and for any point $q \in  \Sigma(x, \epsilon)$,
let us take Killing fields  $\{T, X_1, \ldots, X_{n - 1}\} \subset \ft\oplus \fq$ extended from $\ft_b\oplus \fq_b$, which gives a basis of the tangent space $T_q \Sigma(x, \epsilon)$. It suffices to show that $T\perp \p_r$ and $X_i \perp \p_r$ for each $1\lesl i\lesl n -1$.

First, for $T\in\ft$, by Lemma \ref{l:Killing-field-x-independence}, $T$ can be represented as the coordinate frame $\p_t$ in a small neighborhood $\{x < \epsilon_0\}$ of $S^1 \times S^{n - 1}$ for some $\epsilon_0 > 0$. Let $\gamma:[0,\ell]\to X^{n + 1}$ be the unique normal geodesic from $q = \gamma(\ell)$ to $\sG$ with $\gamma(0) = \sG(t_0)$ for some $t_0 \in [0,2\pi]$. Taking the flow $\{\varphi_t\}$ of $T$, by Lemma \ref{l:invariance-of-core-geodesic}, any $\varphi_t$ acts on $\sG$ as a translation. Then we have that
\begin{align}
g((\varphi_t\cdot \gamma)'(0), \sG'(t + t_0)) = g((\varphi_t)_*(\gamma'(0)), (\varphi_t)_*(\sG'(t_0))) = g(\gamma'(0), \sG'(t_0)) = 0,
\end{align}
which implies that the geodesic $\varphi_t \cdot \gamma$ is orthogonal to $\sG$ and thus it is a normal geodesic from $\varphi\cdot q$ to $\sG$. Since $\{\varphi_t\}$ are isometries, $L(\gamma) = L(\varphi_t \cdot \gamma)$, i.e., $r(q) = r(\varphi_t\cdot q)$.   
  Therefore, $\{\varphi_t\}$ is a one-parameter isometries on a level set of $r$, which in particular implies that $T\perp \p_r$.

In the next case, we consider $X_i\in \fq$ for $1\lesl i\lesl n - 1$. The arguments are similar to the previous case.
Let $q \in \Sigma(x, \epsilon)$ be any point and let $\gamma$ be the unique normal geodesic from $q$ to $p\in \sG$. 
Denote by $\{\psi_s\}$ the flow generated by $X_i$. By Lemma \ref{l:invariance-of-core-geodesic}, every $\psi_s$ acts on $\sG$ as the identity, which implies that $\psi_s \cdot \gamma$ is a normal geodesic with $L(\psi_s \cdot \gamma) = L(\gamma)$. Therefore, $r(\psi_s \cdot q) = r(q)$. Following the same argument as before, one can conclude that $X_i \perp \p_r$.  

By combining the two cases above, we establish the claim. 
\end{proof}
The following corollary follows immediately from 
Proposition \ref{p:r-x-relation}.

\begin{corollary}\label{c:invariance-of-r-x}
 Both $r$ and $x$ are invariant under any element in the isometry groups $\sQ$ and $\sT$.
 \end{corollary}

\begin{proof}
By Proposition \ref{p:r-x-relation}, the Einstein metric $g$
is a warped product   
	\begin{align}
		g = \frac{dx^2 + h}{x^2}, 
	\end{align}
	which holds globally on $X^{n + 1}$. Then applying Lemma \ref{l:Killing-field-x-independence}, any Killing field in $\fq$ and $\ft$ is independent of $x$ for any $0\lesl x\lesl \varpi_0$, and thus any isometry in $\sQ$ and $\sT$ preserves each level set of $x$. Since $x = \varpi_0 \cdot e^{-r}$ for any point in $X^{n + 1}$, it follows that $r$ is invariant under any isometry in $\sQ$ and $\sT$ as well.
\end{proof}

\subsection{Equivariant splitting structure}

With the invariance property of the normal distance function $d_{\sG}$, we are now ready to show that the topological splitting $X^{n + 1} \approx S^1 \times D^n$, as proved in Theorem \ref{t:normal-exp}, inherits the continuous symmetries from the boundary in a rather explicit manner.
 
Recall that, under Assumption \ref{a:pi-1}, we have already shown in Theorem \ref{t:normal-exp} that $X^{n+1}$ is diffeomorphic to $S^1 \times D^n$. Moreover, on the universal cover $\ucX$, the normal exponential map  $ \nExp: \mathcal{N}(\ellG) \to \ucX$ of the core geodesic line $\ellG$ provides a $\Gamma$-equivariant diffeomorphism between the trivial normal bundle $\mathcal{N}(\ellG)$ and the universal cover $\ucX$, and it descends to a diffeomorphism $ \nExpX: \mathcal{N}(\sG) \to X^{n + 1}$.

For each $r> 0$, let us denote by 
\begin{align}
\ND(r) \equiv \left\{x\in X^{n + 1}: d_{\sG} (x) < r\right\} = \nExpX \left(\left\{\bm{w} \in \mathcal{N}(\sG): |\bm{w}| < r\right\}\right),
\\ 
\NS(r) \equiv \left\{x\in X^{n + 1}: d_{\sG}(x) = r\right\} = \nExpX \left(\left\{\bm{w} \in \mathcal{N}(\ellG): |\bm{w}| = r\right\}\right),
 \end{align}
the sub-level set and the level set of the normal distance function $\br(\cdot)\equiv d_{\sG}(\cdot)$, respectively.  Since the normal exponential map is a global diffeomorphism, for each $r > 0$, $\ND(r)$ is diffeomorphic to $D^n \times \dR$
and $\NS(r)$ is diffeomorphic to $S^{n - 1} \times \dR$.  For any point $p \in \sG$, its {\it normal fiber} in $X^{n + 1}$ is defined as  
\begin{align}\mathcal{N}(p) \equiv \nExpX\left((\Pi^{\perp})^{-1}(p)\right)  \subset X^{n + 1},\end{align} which is the image of the normal fiber $(\Pi^{\perp})^{-1}(p) \subset \mathcal{N}(\sG)$ under the normal exponential map. At any point $p\in \sG$, {\it the normal geodesic disc} $\ND(r,p)$ and {\it the normal geodesic sphere} $\NS(r,p)$ are defined as 
\begin{align}
	\ND(r,p) \equiv \ND(p) \cap \cN(p)\quad \text{and} \quad \NS(r,p) \equiv \NS(r) \cap \cN(p), \end{align}
respectively. See Figure \ref{f:splitting-of-X} for descriptions of these terms.

\begin{figure}
[h] \label{f:splitting-of-X}
 \begin{tikzpicture}
 [scale= 0.34]
\draw[thick] (3,-5) ellipse (15 and 8);

\draw[thick](3,-2.3) [partial ellipse=-160:-20: 10 and 5];

\draw[thick](3,-6) [partial ellipse=18:162: 9.1 and 4];

%\draw[blue, densely dotted] (0,-9.9) ellipse (1.5 and 2.8);

 \filldraw[fill=gray!40, draw= , densely dotted]
         (0,-9.9) ellipse (1.5 and 2.85)
                      [postaction={on each segment={draw,-{stealth[red,bend]}}}];

 \filldraw[fill=blue!20, draw= blue, densely dotted]
         (0,-9.9) ellipse (1.2 and 2.3)
                      [postaction={on each segment={draw,-{stealth[red,bend]}}}];

                       \filldraw[fill=blue!20, draw= blue, densely dotted]
         (0,-9.9) ellipse (0.8 and 1.8)
                      [postaction={on each segment={draw,-{stealth[red,bend]}}}];

\draw[thick, blue, dashed] (0,-9.9) [partial ellipse=90:270:1.2 and 2.3]
           [postaction={on each segment={draw,-{stealth[red,bend]}}}];
           
\draw[thick, blue, dashed] (0,-9.9) [partial ellipse=90:270: 0.8 and 1.8]
           [postaction={on each segment={draw,-{stealth[red,bend]}}}];

%\draw[thick, blue, densely dotted] (0,-9.9) [partial ellipse= -90: 90:1.2 and 2.3];

 \draw[thick, dashed] (0,-9.9) [partial ellipse=90:270:1.5 and 2.85]
           [postaction={on each segment={draw,-{stealth[red,bend]}}}];

  \draw[blue, densely dotted] (0.7, -7.5) to (0.4, -8.5) to  (0,-9.8) 
             [postaction={on each segment={draw,-{stealth[black,bend]}}}];

 \draw[red, thick](3,-5) [partial ellipse=-103:248: 12.1 and 5]
            [postaction={on each segment={draw,-{stealth[blue,bend]}}}];

\draw[red, thick, densely dotted](3,-5) [partial ellipse=248: 270: 12.1 and 5]
           [postaction={on each segment={draw,-{stealth[blue,bend]}}}];

\node at (0.1, -9.9) {\tiny \textcolor{red}{$\bullet$}};

\node at (0, -10.5) {\footnotesize
 $p$};

\node at (0, -14) {$\mathcal{N}(p) \approx D^n$};

%\node at (0, -9) {\scriptsize $\partial_x$};

\node at (9.5, -10) {\textcolor{red}{$\sigma_{\Gamma}$}};

\end{tikzpicture}

\caption{The total space $X^{n + 1}$ is diffeomorphic to $S^1 \times D^n$. In the illustration, the red circle represents the core geodesic loop $\sigma_{\Gamma}$; the larger gray disc represents the normal fiber $\cN(p)$; the smaller blue disc illustrates the normal disc $\cN(p,r)$, whose  boundary $\cS(p,r)$ is shown as the blue circle; the red arrows stand for the $\SO(n)$-action; the blue arrows stand for the $\dR$-action.}

\end{figure}
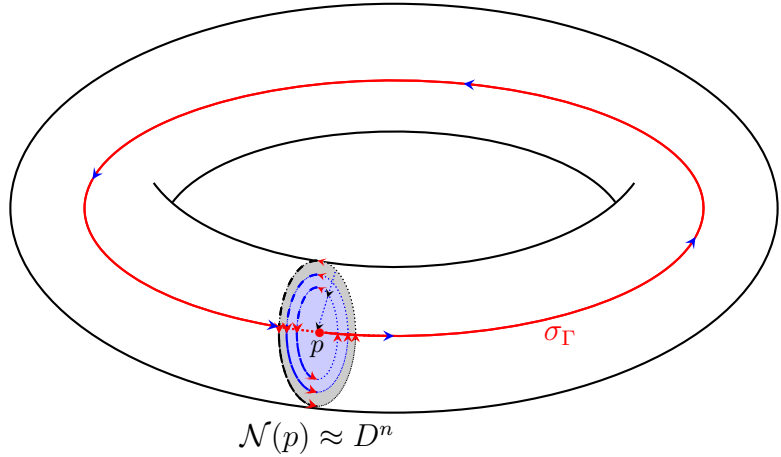

To state the main result of this subsection, let us define  
\begin{align}
	\ft_b \equiv \isom(S^1, \wg_{\lambda}),\quad \fq_b \equiv \isom(S^{n - 1}, \wg_{\lambda}),
\end{align}
 and their extensions $\fq$ and $\ft$ in $X^{n + 1}$, as defined in \eqref{e:q-t-extensions-to-X}. 	
It has been proven in Lemma \ref{l:isomorphic-factors} that $\fq \cong \fq_b$ and $\ft \cong \ft_b$, where $\fq_b \cong \mathfrak{s}\mathfrak{o}(n)$ and $\ft_b \cong \dR$. In particular, $[\fq, \ft] = 0$.
Then the global equivariance properties of the topological splitting of $(X^{n + 1}, g)$
can now be stated in the following theorem.

 \begin{theorem}[Equivariant Splitting Theorem] \label{t:compatible-splitting}
Let $(X^{n + 1}, g)$ be a Poincar\'e-Einstein manifold with non-positive sectional curvature that satisfies Assumption \ref{a:product-conformal-infinity}. Let $\sG$ be the core geodesic loop of $X^{n + 1}$. 
Then the following properties hold:
  	\begin{enumerate}
  		\item For any element $\xi \in \fq$ and $T\in \ft$,  we have that
  		\begin{align}
  			g(\xi ,T) = g(\xi , \p_r) = g(T, \p_r) \equiv 0\quad \text{on}\ X^{n + 1},
  		\end{align}
  		where $r(\cdot) \equiv d_{\sG}(\cdot)$ is the normal distance function.
  		\item For every $p\in \sG$, the isometry group $\sQ$ preserves the normal fiber $\ND(p,r)$. Furthermore,
  		any element $\varphi \in \sQ$ acts on $\sG$ as the identity, and the group $\sQ$ acts transitively on the normal geodesic sphere $\NS(p, r)$ for any $r > 0$. 
  		\item  Any element $\tau \in \sT$ acts on $\sG$ as a translation. Moreover, for any $p = \sG(t_0) \in \sG$ and for any $r > 0$, the restriction 
  		$\tau|_{\NS(p,r)}$ provides an isometry 
  		\begin{align}
  		\tau|_{\NS(p, r)}: \NS(\sG(t_0), r)	 \to  \NS(\sG(t_0^*), r),	
  		\end{align}
  		where $\sG(t_0^*) = \tau\cdot \sG(t_0)$ and $|t_0^* - t_0| = |\tau|$.
  	\end{enumerate}
\end{theorem}

We now establish the equivariant splitting theorem via a series of preliminary results. 

First, we introduce the following {\it equivariant coordinates}.  
For the conformal infinity $S^1 \times S^{n-1}$, let $(I\times U, (t,\bm{w}))$ be a local coordinate system,  where $I$ is an open interval and $\bm{w} = (w_1, \ldots, w_{n-1})$ are local coordinates on $U \subset S^{n-1}$. Let $x$ be the global geodesic defining function provided by Proposition \ref{p:r-x-relation}.
Then $(t,\bm{w})$ extends to a coordinate system $(x,t,\bm{w})$ on an open set $\mathring{\mathcal{U}} \subset X^{n+1}$ that is diffeomorphic to the topological product
 $[0,\varpi_0) \times I\times U$ for $0\lesl x < \varpi_0$.
 Clearly, $\mathring{\mathcal{U}}$ is contained in 
 $\overline{X^{n+1}}\setminus \sG$, and the level set $\{x = 0\}\cap \mathring{\mathcal{U}}$ corresponds to $I \times U \subset M^n$. 
 We also note that the level set $\{x = \varpi_0\}$ corresponds to the geodesic line $\sG$. 
 Let $\mathcal{U}$ be the completion of 
 $\mathring{\mathcal{U}}$ along $\{x = \varpi_0\}$. Then, one can represent 
 the Poincar\'e-Einstein metric $g$ on $\mathcal{U}$ as follows,
 \begin{align}
	g = dr^2 + g_r = \frac{dx^2 + h_x(t, \bm{w})}{x^2}, \ 0\lesl x\lesl \varpi_0,
\end{align}
where for any fixed value $x = x_0 \in[0, \varpi_1)$, $h_x(t, \bm{w})$ is a homogeneous metric on a level set    $\{x = x_0\}$ ($\approx S^1 \times S^{n - 1}$).

 	The lemma below shows that any extended Killing field on $X^{n+1}$ is orthogonal to the normal vector field $\p_r$.
\begin{lemma}\label{l:symmetries-perp-r}
Let  $X \in \fq \oplus \ft$ be a Killing field. Then $g( X, \p_r ) \equiv 0$ on $X^{n + 1}$.	
\end{lemma}

\begin{proof} The arguments are the same as those used in the proof of the claim within the proof of Proposition \ref{p:r-x-relation}. In the statement of the claim, the orthogonality relation holds only near the boundary, where the geodesic defining function $x$ exists.  Since 
Proposition \ref{p:r-x-relation} gives the global existence of the geodesic defining function,  
the desired orthogonality condition $g( X, \p_r ) \equiv 0$ extends to the entire $X^{n + 1}$.	
\end{proof}

To further reduce the complexity of the Einstein equation $\Ric_g = - n g$ on $X^{n + 1}$, the next lemma provides the isometric splitting structure of the normal level set $\NS(r)$.

\begin{lemma}\label{l:translation-perp-to-spherical-symmetries} For each $r > 0$, the normal level set $\NS(r)\equiv \{x\in X^{n + 1}: r(x) = r\}$ is isometric to a Riemannian product $(S^{n - 1} \times S^1, F^2(r) g_{S^{n - 1}} + G^2(r)\tau^2)$, where $F(r)$ and $G(r)$ are smooth functions in $r$,
$g_{S^{n - 1}}$ is the round metric on $S^{n - 1}$ with constant sectional curvature  $+1$, and $\tau^2$ is the standard metric on the unit circle $S^1$. 
\end{lemma}

\begin{proof}
First, recall that each normal level set $\NS(r)$ 
 is diffeomorphic to the cylinder $S^{n - 1} \times  S^1$. 
 It has been proven in Section \ref{ss:invariant-normal-dist} that,  for each $r > 0$, the level set
$(\NS(r), g|_{\NS(r)})$ 
 is a Riemannian homogeneous space 
 whose Lie algebra contains $\fq \oplus \ft$. We also proved in Lemma \ref{l:isomorphic-factors} that $\fq \oplus \ft \cong \fq_b \oplus \ft_b$. In particular, 
 \begin{align}
 \dim ( \Isom(\NS(r)) ) \gesl \frac{n(n - 1)}{2} + 1. 	\label{e:large-isometry-group}
 \end{align}

From this dimension lower bound, it follows that 
 for each $r > 0$, the normal level set $\NS(r)$, viewed as a Riemannian homogeneous space, must be isometric to a round cylinder. 
 More precisely, the metric $g|_{\NS(r)}$ inherited from $(X^{n + 1}, g)$ isometrically splits into a round circle and a round metric on $S^{n - 1}$. Specifically, one can write
 \begin{align}
  g|_{\NS(r)} = F^2(r) g_{S^{n - 1}} + G^2(r)\tau^2 	 \label{e:splitting-of-NS}
 \end{align}
 for smooth functions $F=F(r)$ and $G = G(r)$ in $r$,
 where $g_{S^{n - 1}}$
is the round metric on $S^{n - 1}$ with constant sectional curvature $+1$
 and $(S^1,\tau^2)$ is the unit circle.
The above splitting structure can be obtained from some general facts of the geometric structure  of homogeneous spaces that admit large isometry groups, which are by now well-understood. We list some relevant references of this topic below.

Given the dimension estimate \eqref{e:large-isometry-group},   
    it was shown in \cite[theorem 10]{Obata-homogeneous} that $\dim(\Isom(\NS(r)))$ must be equal to either $1 + \frac{n (n - 1)}{2}$ or $\frac{n (n + 1)}{2}$. Notice that the latter case is impossible, since if this were true, $\NS(r)$ would have constant curvature. Therefore, 
    \begin{align}
   \dim(\Isom(\NS(r))) = 1 + \frac{n (n - 1)}{2}.  	
    \end{align}
 The same result \cite[theorem 10]{Obata-homogeneous} states that $(\NS(r), g|_{\NS(r)})$ 
must be a Riemannian direct product as in \eqref{e:splitting-of-NS}.
We thus complete the proof of the lemma. 
\end{proof}

\begin{remark}
For a proof of the isometric splitting \eqref{e:splitting-of-NS}, see also 
\cite[Chapter II.3]{Kobayashi} for $n\gesl 5$ and \cite{Ishihara} for $n = 4$. The classification of homogeneous $3$-manifolds is well known.
\end{remark}

Based on the previous lemmata, 
one can conclude that each normal slice $\NS(r)$ shares the same Lie algebra information with the conformal infinity $S^1 \times S^{n-1}$. 
We summarize this picture as follows.

\begin{proposition}\label{p:foliation-structure}
For every point $p  \in \sG$, the following invariance properties hold:
\begin{enumerate}
\item For any $r > 0$ and for any isometry $\varphi \in \sQ$,   the restriction $\varphi|_{\NS(p, r)}$ gives an isometry 
\begin{align}
\varphi|_{\NS(p, r)}: 	\NS(p, r) \to \NS(p, r)
\end{align}
of the normal geodesic sphere $\NS(p, r)$.
Furthermore, $\sQ$ acts transitively on each normal geodesic sphere $\NS(p, r)$.

\item For any $p\in \sG$, each element $\tau\in \sT$ acts on $\sG$ by translation and sends $\mathcal{N}(p)$ to 
$\mathcal{N}(\tau\cdot p)$. Moreover, 
for any $r > 0$, the restriction $\tau|_{\NS(p, r)}$ gives an isometry 
\begin{align}
\varphi|_{\NS(p, r)}: 	\NS(p, r) \to \NS(\tau\cdot p, r).
\end{align}

\end{enumerate}
\end{proposition}

\begin{proof}To prove (1),
it suffices to show that any isometry $\varphi\in\mathscr{Q}$ preserves the normal geodesic sphere $\NS(p, r)$
	for any $p \in \sG$ and $r > 0$. Let $p = \sG(t_0)$ and $q\in \NS(p, r)$. We take    a length minimizing geodesic $\gamma:[0, r]\to X^{n + 1}$ with $\gamma(0) = p$ and $\gamma(r) = q$ such that $\gamma'(0) \perp \sG'(t_0)$. Immediately, $\varphi\cdot\gamma$ is also a length minimizing geodesic since $\varphi$ is an isometry. By Lemma
	\ref{l:invariance-of-core-geodesic}, we have that $\varphi\cdot p = p$, which implies that 
	\begin{align}
		d_g (\varphi \cdot q, p) = 		d_g (\varphi \cdot q, \varphi\cdot p) = 
d_g (q, p) = r. 
	\end{align}  
Notice that the geodesic $\varphi\cdot\gamma$ is orthogonal to $\sG$ at $p$. Therefore,  $\varphi\cdot q \in\NS(p,r)$ and the restriction $\varphi|_{\NS(p,r)}$	is an isometry of $\NS(p,r)$.

The proof of (2)   follows from the fact that any element $\tau \in  \sT$ acts on $\sG$ by translation; see Lemma \ref{l:invariance-of-core-geodesic}. In fact, taking any point $p = \sG(t_0)$ and any point $q \in \NS(p, r)$, let $\gamma:[0,r] \to X^{n + 1}$ be a length minimizing geodesic connecting $p = \gamma(0)$ and $q = \gamma(r)$ such that $\gamma'(0) \perp \sG'(t_0)$. Then  $\tau\cdot \gamma:[0,r] \to X^{n + 1}$ becomes a geodesic that connects $\tau\cdot p$ and $\tau\cdot q$. Since $\tau\cdot p \in \sG$, the end point $\tau\cdot q$ is contained in $\NS(\tau\cdot p, r)$. We thus complete the proof.
\end{proof}

By combining the technical results established above, we can now complete the proof of Theorem \ref{t:compatible-splitting}.

\begin{proof}
	[Proof of Theorem \ref{t:compatible-splitting}]
 	The orthogonality relations in item (1) are due to Lemma \ref{l:symmetries-perp-r} and Lemma \ref{l:translation-perp-to-spherical-symmetries}. 
	Items (2) and (3) are proved in Proposition \ref{p:foliation-structure}.
\end{proof}

\subsection{Dimension reduction and hyperbolic rigidity}
\label{ss:dimension-reduction}

In this section, we will carry out dimension reduction argument and complete the proof of the isometric rigidity part of Theorem \ref{t:product-top-geo-rigidity}.

The following is the statement of the main result of dimension reduction.

\begin{theorem}[Dimension reduction]\label{t:dimension-reduction}
For any $\tp \in \ell_{\Gamma}$ and $r > 0$, let $\cN(\tp)$ be the normal fiber given by the normal exponential map $\nExp$, and 
let $\NS(\tp, r)$ be a normal geodesic sphere. Then the following properties hold. 
\begin{enumerate}
 
 	\item  For any $\tilde{q}\in \cN(\tp)$, let $\wsT\cdot \tilde{q}$ be the $\wsT$-orbit at $\tilde{q}$. 
Then $\wsT\cdot \tilde{q}$ is orthogonal to $\cN(\tp)$ at $\tilde{q}$.

 \item For any $\tp\in \ellG$, the normal fiber $\cN(\tp)$ is totally geodesic in $(X^{n+1}, g)$.

	\item $\NS(\tp, r)$ is a round sphere whose sectional curvature is a constant depending only on the radius $r$.

\end{enumerate}
\end{theorem}
\begin{proof}
  Item (1) follows from 
$g(\p_t, \p_r) = g(\p_t, \p_{w_i})  \equiv 0$ on $X^{n+1}$, and these
 orthogonality properties, in turn, follow from Lemma \ref{l:symmetries-perp-r} and Lemma \ref{l:translation-perp-to-spherical-symmetries}.

 We now prove item (2). Let $(r, t,\bm{w})$ be a local normal coordinate system of the normal fiber $\cN(\tp)$, where $\bm{w}$ is some local coordinate system induced from the spherical factor $S^{n-1}$ of the conformal infinity.
Notice that, in local coordinates, $\p_t$ is a Killing field. Then for any $1\lesl i,j\lesl n-1$, 
\begin{align}
	\p_t \left( g(\p_{w_i},\p_{w_j}) \right) = (\mathfrak{L}_{\p_t} g)(\p_{w_i},\p_{w_j}) + g\left([\p_t,\p_{w_i}],\p_{w_j}\right) +  g\left(\p_{w_i},[\p_t, \p_{w_j}]\right) \equiv 0.
\end{align}
Then the conclusion immediately follows.

To show item (3), one needs to analyze the Lie algebra $\isom(\NS(\tp,r))$ of the Killing fields on $\NS(\tp,r)$. For each $\tp\in \ell_{\Gamma}$ and $r > 0$, by Theorem \ref{t:compatible-splitting}(3), the isometry group $\wsQ$ acts transitively on each normal geodesic sphere $\NS(\tp,r)$ such that $\NS(\tp,r)$ is a Riemannian homogeneous space. By Lemma \ref{l:Killing-field-x-independence},  $\isom(\NS(\tp,r)) \cong \fq_b \cong \mathfrak{s}\mathfrak{o}(n)$. 
	It is well known that the round sphere $S^{n - 1}$, as a Riemannian homogeneous space, can be represented as $\SO(n)/\SO(n - 1)$. Correspondingly, the Lie algebra $\isom(S^{n  - 1})$ of the Killing fields on $S^{n - 1}$ is $\mathfrak{s}\mathfrak{o}(n)$.  
This immediately tells us that each slice $\NS(\tp, r)$ has positive constant sectional curvature. 	
\end{proof}

The above structure result provides an effective dimension reduction for the Einstein metric $g$. By the normal exponential map of the line $\ell_{\Gamma}$, one can write the Einstein metric $g$ in terms of the following ansatz.

\begin{corollary}Let $r$ be the distance to the canonical circle $\sigma_{\Gamma}\subset X^{n+1}$. Then there exist analytic functions $F=F(r)$ and $G = G(r)$ such that the Einstein metric $g$ on $X^{n+1}$ can be expressed as
\begin{align}
g = \kappa + G^2(r) \tau^2  = dr^2 + F^2(r) g_{S^{n - 1}}+ G^2(r)\Theta^2,
\end{align}
where $g_{S^{n - 1}}$ is the round metric on $S^{n-1}$ with constant curvature $+1$ and $\Theta^2$ is a Riemannian metric on the circle $S^1$.
\end{corollary}

\begin{proof}
This is an immediate corollary of Theorem \ref{t:dimension-reduction}.
\end{proof}

Next, we will reduce the Einstein equation $\Ric_g = - n g$ to a system of ODEs in $r$ based on the dimension reduction as obtained above.

For any $p\in X^{n+1}$, let us choose a local orthonormal frame $\{\p_r, X_1,\ldots, X_{n - 1}, T\}$, 
where $\p_r$ is given by the normal distance function, $X_i$ is tangential to the normal slice $\NS(\tp, r)$ and $T$ is the unit tangent vector of the $\mathscr{T}$-orbit at $p$. 
By straightforward computations, we have  
\begin{align}
\Ric_g(X_i,X_i)	& = \Ric_{\kappa}(X_i,X_i) - \frac{1}{G}\nabla^2 G(X_i,X_i),
\\
\Ric_g(\p_r,\p_r)	& = \Ric_{\kappa}(\p_r,\p_r) - \frac{1}{G}\nabla^2 G(\p_r,\p_r),
\\
\Ric_g (T, T) & =  -\frac{\Delta_{\kappa} G}{G}.
\end{align}
Let us compute the Ricci curvature of $\kappa$ at any point $p\in \NS(p,r)$: 
\begin{align}
\begin{split}
\Ric_{\kappa}(X_i) & = \frac{1}{F^2(r)}  - 
\frac{(F'(r))^2}{F(r)^2} - (n-2)\cdot \frac{F''(r)}{F(r)} ,
\\
 \Ric_{\kappa}(\p_r) & = - (n-2) \cdot \frac{F''(r)}{F}.
 \end{split}\end{align}
Next, we compute the Hessian $\nabla^2G(v,v)$ with $v\in T_p \NS(p,r)$: 
\begin{align}
\nabla^2 G(X_i,X_i) = G'(r) \cdot \frac{F'(r)}{F(r)} \quad \text{and} \quad \nabla^2 G(\p_r, \p_r) = G''(r).	
\end{align}
Finally, applying the equation $\Ric_g = - n g$, we have the following system:
\begin{align}
\begin{split}
(n-2)\frac{F''(r)}{F(r)}+ \frac{(F'(r))^2}{F(r)^2} - \frac{1}{F(r)^2} + \frac{G'(r)}{G(r)} \cdot \frac{F'(r)}{F(r)}
= n,
\\
(n-1) \cdot \frac{F''(r)}{F(r)} + \frac{G''(r)}{G(r)} = n,
\\
\frac{G''(r)}{G(r)} + (n-1)\cdot \frac{G'(r)}{G(r)} \cdot \frac{F'(r)}{F(r)}
 = n.
 \end{split}\label{e:Einstein-ODE}
\end{align}
The above can be further reduced to the following:  
\begin{align}
& (n-1) \cdot F(r) F''(r) + (F'(r))^2 - 1 - n \cdot F(r)^2 = 0,
\\
& \frac{G''(r)}{G(r)} = n - (n - 1)\cdot\frac{F''(r)}{F(r)}.
\end{align}

Since $g$ is smooth one the circle $\sigma_{\Gamma}$ which corresponds to the level set $\{r=0\}$, this in particular implies that 
\begin{align}
F(0) = 0 ,\quad F'(0) = 1.	
\end{align}

Then we are led to prove the following uniqueness problem. 
\begin{proposition}\label{p:ODE-uniqueness}
The initial value problem 
\begin{align}\label{e:ODE-initial-value-problem}
\begin{cases}
2F(r) F''(r) + (F'(r))^2 - 1 - 3F(r)^2  = 0,
\\
F(0) = 0 , \quad F'(0) = 1,
\end{cases}
\end{align}
has a unique solution $F(t)\equiv \sinh(t)$.
\end{proposition}

\begin{proof}
It is well known that any Einstein metric is analytic in geodesic normal coordinates;   see \cite[theorem 5.1]{DK}. So the function $F = F(r)$ is real analytic in $r$.

We will show that  for any $k\in\dZ_+$, 
 \begin{align}F^{(2k)}(0)=0, \quad F^{(2k+1)}(0)=1.\end{align}
 Once we obtain the above, the uniqueness will follow from the standard unique continuation of analytic functions.

Since the metric $g$ is smooth at $r=0$ and $F(0)=0$, we have that $F^{(2k)}(0) = 0$
for any $k\in\dZ_+$. To complete the proof, it suffices to prove that $F^{(2k+1)}(0)=1$
for any $k\in\dZ_+$. Since $F(r)$ is analytic, let us take the Taylor series of $F$ at $r=0$,
\begin{align}
F(r) = \sum\limits_{m=1}^{\infty} \frac{a_m}{m!}\cdot r^m,	
\end{align}
 where $a_m=0$ for any even integer $m$.
Then 
\begin{align}
F'(r) = \sum\limits_{m=1}^{\infty}\frac{a_m}{(m-1)!} r^{m-1}, \quad 
F''(r) = \sum\limits_{m=1}^{\infty}\frac{a_{m+1}}{(m-1)!} r^{m-1}. 		
\end{align}
Plugging the above into \eqref{e:ODE-initial-value-problem}, 
\begin{align}
\begin{split}
(n - 1) \cdot \left( \sum\limits_{m=1}^{\infty} \frac{a_m}{m!}\cdot r^m \right)	 \left( \sum\limits_{m=1}^{\infty}\frac{a_{m+1}}{(m-1)!} r^{m-1} \right)	& \ +  \left( \sum\limits_{m=1}^{\infty}\frac{a_m}{(m-1)!} r^{m-1} \right)^2 \\
&\  - 1 - n \cdot \left( \sum\limits_{m=1}^{\infty} \frac{a_m}{m!}\cdot r^m \right)^2 = 0.
\end{split}
\end{align}
For any $\nu \geq 4$, comparing the coefficients of the term of order $r^{\nu}$, 
\begin{align*}	
0= & \ (n - 1) \sum\limits_{k=1}^{\nu} \frac{a_k}{k!} \cdot \frac{a_{\nu + 2 - k}}{(\nu - k)!}	
+ \sum\limits_{k=1}^{\nu + 1} \frac{a_k}{(k-1)!} \cdot \frac{a_{\nu + 2 - k}}{(\nu + 1 - k)!} - n \sum\limits_{k = 1}^{\nu - 1} \frac{a_k}{k!} \cdot \frac{a_{\nu - k}}{(\nu - k)!} 
\\
=  &\  (n - 1)\cdot \left( \frac{a_{\nu + 1}}{(\nu - 1)!} + \sum\limits_{k=2}^{\nu} \frac{a_k}{k!} \cdot \frac{a_{\nu + 2 - k}}{(\nu - k)!}\right) + \left( \frac{2 a_{\nu + 1}}{\nu!} + \sum\limits_{k=2}^{\nu} \frac{a_k}{(k - 1)!} \cdot \frac{a_{\nu + 2 - k}}{(\nu + 1- k)!}\right)
\\ 
   &\ - n \cdot  \sum\limits_{k=1}^{\nu-1} \frac{a_k}{k!} \cdot \frac{a_{\nu - k}}{(\nu - k)!} 
      \\
  = &\ \frac{(n-1)\nu + 2}{\nu!} a_{\nu + 1} + (n - 1) \sum\limits_{k=2}^{\nu} \frac{a_k}{k!} \cdot \frac{a_{\nu + 2 - k}}{(\nu - k)!} + \sum\limits_{k=2}^{\nu} \frac{a_k}{(k - 1)!} \cdot \frac{a_{\nu + 2 - k}}{(\nu + 1 - k)!} \nonumber\\ 
  &\ - n  \cdot  \sum\limits_{k=1}^{\nu - 1} \frac{a_k}{k!} \cdot \frac{a_{\nu - k}}{(\nu - k)!}. 
\end{align*}
 Since $a_{2m}=0$ for any integer $m\in\dN_0$, we take $\nu = 2Q$ for some positive integer $Q\geq 2$.
Then the above equation becomes
\begin{equation}
\begin{split}	
	0 = & \ \frac{2Q(n - 1) + 2}{(2Q)!} \cdot a_{2Q + 1} \\
	& \ + (n - 1)\cdot \sum\limits_{m=1}^{Q-1} \frac{a_{2m + 1}}{(2m + 1)!}	\cdot \frac{a_{2Q + 1 -2m}}{(2Q - 1 - 2m)!} + \sum\limits_{m = 1}^{Q - 1}
\frac{a_{2m + 1}}{(2m)!} \cdot \frac{a_{2Q - 2m + 1}}{(2Q - 2m)!}   \\
&\  - n\cdot\left( a_1 \cdot \frac{a_{2Q - 1}}{(2Q - 1)!} + \sum\limits_{m=1}^{Q-1} \frac{a_{2m+1}}{(2m + 1)!} \cdot \frac{a_{2Q - 2m - 1}}{(2Q - 2m - 1)!} \right).
\label{e: Q-sequence}
\end{split}
\end{equation}

In the following, we will prove by induction that $a_{2Q + 1} = 1$ for any positive integer $Q \geq 2$.
First, by straightforward computations $a_3 = 1$. Given any $Q \geq 2$, assuming that $a_{2m+1} = 1$ for any positive integer $2\lesl  m \lesl  Q - 1$, we will show that $a_{2Q + 1} = 1$.  
	
Applying the induction hypothesis, \eqref{e: Q-sequence} becomes  	
	\begin{align}
 0 = &\  \frac{2Q(n - 1) + 2}{(2Q)!} a_{2Q + 1} + \left( \sum\limits_{m = 1}^{Q - 1}
\frac{1}{(2m)!} \cdot \frac{1}{(2Q - 2m)!} - \sum\limits_{m=1}^{Q-1} \frac{1}{(2m + 1)!}	\cdot \frac{1}{(2Q - 1 - 2m)!}   \right)	
\nonumber \\
  &\ -  \frac{n}{(2Q - 1)!}	\nonumber \\
 = & \ \frac{2Q(n - 1) + 2}{(2Q)!} a_{2Q + 1} - \frac{n}{(2Q - 1)!} + \sum\limits_{k = 2}^{2Q - 1} (-1)^k \frac{1}{k!} \cdot \frac{1}{(2Q - k)!}. \label{e:Q-sequence-representation}
	\end{align}
We claim that 
\begin{align}
\sum\limits_{k = 2}^{2Q - 1} (-1)^k \frac{1}{k!} \cdot \frac{1}{(2Q - k)!}  = \frac{2Q -2}{(2Q)!}
\end{align}
Indeed, let us consider the power series expansions of $e^r$ and $e^{-r}$,
\begin{align}
e^r = \sum\limits_{k=0}^{\infty} \frac{1}{k!} r^k,\quad e^{-r} = \sum\limits_{k=0}^{\infty} (-1)^k \frac{1}{k!} r^k.\end{align}
So it follows that 
\begin{align}
1 = e^r \cdot e^{-r} = \sum\limits_{n=0}^{\infty} \left( \sum\limits_{k=0}^n (-1)^k \frac{1}{k!} \cdot \frac{1}{(n - k)!} \right) r^n.	
\end{align}
Taking $n = 2Q$ for $Q\geq 2$, we have 
\begin{align}
 \sum\limits_{k=0}^{2Q} (-1)^k \frac{1}{k!} \cdot \frac{1}{(2Q - k)!} = 0.
\end{align}
  Then  
  \begin{align}
\frac{1}{(2Q)!} - \frac{1}{(2Q - 1)!} + \sum\limits_{k=2}^{2Q -1} (-1)^k \frac{1}{k!} \cdot \frac{1}{(2Q - k)!} + \frac{1}{(2Q)!} = 0,
  \end{align}
which immediately implies 
\begin{align}
	\sum\limits_{k=2}^{2Q -1} (-1)^k \frac{1}{k!} \cdot \frac{1}{(2Q - k)!}  = 
	\frac{2Q - 2}{(2Q)!}. \label{e:closed-form-finite-sum}
\end{align}
This completes the proof of the claim.

Now we are ready to finish the proof. Plugging \eqref{e:closed-form-finite-sum} into 
\eqref{e:Q-sequence-representation}, 
\begin{align}
\frac{2Q(n - 1) + 2}{(2Q)!} a_{2Q + 1}  - \frac{n}{(2Q - 1)!} + \frac{2Q - 2}{(2Q)!} = 0.	
\end{align}
Therefore, $a_{2Q + 1} = 1$, which completes the proof.
\end{proof}

\begin{theorem} Let $g$ be the Poincar\'e-Einstein metric 
\begin{align}
	g = dr^2 + F^2(r) g_{S^{n - 1}} + G^2(r) \Theta^2
	\end{align} on $X^{n+1}\approx S^1 \times D^n$
such that $\Ric_g \equiv  - n g$.
 Then $g$ is a hyperbolic metric with sectional curvature constantly equal to $-1$.	
\end{theorem}

\begin{proof}
Let $X_1, X_2$ be orthonormal frames tangential to the fiber $S^2$ and let $T$ be the unit vector field such that $\Theta(T) = 1$. 
Recall that the Einstein equation $\Ric_g \equiv - n g$ is reduced to the ODE system \eqref{e:Einstein-ODE}. Since Proposition \ref{p:ODE-uniqueness} gives $F(r) \equiv \sinh(r)$,  the system \eqref{e:Einstein-ODE} is further reduced to  
\begin{align}
\frac{G''(r)}{G(r)} = 1,
\\
\frac{G''(r)}{G(r)} + (n - 1) \cdot \frac{F'(r)}{F(r)} \cdot \frac{G'(r)}{G(r)} = n.
\end{align}

Next, by elementary computations, all the sectional curvatures are given by
\begin{align}
&\sec(\p_r \wedge X_i) = - \frac{F''(r)}{F(r)} = - 1, \quad \sec(\p_r \wedge T) = - \frac{G''}{G} = -1,\\
& \sec(X_i \wedge X_j) = \frac{1 - (F'(r))^2}{F(r)^2} = -1,
\\
&\sec(X_i, T) = -\frac{F'(r)G'(r)}{F(r)G(r)} = -1.
\end{align}
Finally we can conclude that $\sec_g \equiv - 1$.
\end{proof}

\section{Rigidity and Quantitative Rigidity}

\label{s:more-results}

In this section, we will prove Theorems \ref{t:large-circle}, \ref{t:uniqueness-large-circle-factor}, and \ref{t:rigidity-for-small-weyl-energy}, as stated in Section \ref{s:introduction}. Each of these results is a consequence of the main hyperbolic rigidity result,  Theorem \ref{t:uniqueness-nonpositively-curved}. 
In Section \ref{ss:pinching}, we will establish two hyperbolic rigidity results under pinching assumptions: one where the circle factor in the product $S^1\times S^{n-1}$ is sufficiently large (Theorem \ref{t:large-circle}), and one where the Weyl energy is sufficiently small (Theorem \ref{t:rigidity-for-small-weyl-energy}). In Section \ref{ss:quantitative-rigidity}, we
 prove a quantitative rigidity result (Theorem \ref{t:uniqueness-large-circle-factor}) under the assumption that the conformal infinity is close to the standard product on $S^1 \times S^{n-1}$.
 We will also formulate and prove a compactness result (Proposition \ref{p:compactness-small-weyl}) that will be used in the proof of the quantitative rigidity result.

To formulate the quantitative rigidity and compactness arguments, we will require the notions of {\it harmonic coordinates} and {\it regularity scale}, introduced as follows:
\begin{definition}
[$C^1$-harmonic coordinates]	
Let $(X^d, g)$ be a Riemannian manifold. Given $p\in X^d$, a $C^1$-harmonic coordinate system $\Phi$ at $p$ with  $\|\Phi\|_{C^1}<\epsilon$ is a diffeomorphism 
\begin{align}
	\Phi = (u_1,\ldots, u_n): B_r(p) \to \dR^d
\end{align}
that satisfies: 
\begin{enumerate}
	\item $\Delta_g u_j = 0$ on $B_r(p)$ for any $1\lesl j \lesl d$.
	\item 
the metric coefficients $g_{ij} \equiv g(\nabla u_i, \nabla u_j)$ in harmonic coordinates satisfy 
\begin{align}
	|g_{ij} - \delta_{ij}|_{C^0(B_r(p))}
	+ |\nabla g_{ij}|_{C^0(B_r(p))} < \epsilon.\label{e:C^1-regularity}
\end{align}
 \end{enumerate}
\end{definition}
In harmonic coordinates, the Ricci tensor $\Ric_g$ can be expressed as 
\begin{align}
	\Ric_{ij} = - \frac{1}{2}\Delta_g g_{ij} + Q(\p g),
\end{align}
where $Q(\p g)$ denotes the terms that are quadratic in the first derivatives of the metric. Standard elliptic estimates imply that the $C^1$-regularity in \eqref{e:C^1-regularity} in harmonic coordinates can be automatically improved to a uniform $C^{1,\alpha}$-bound for some $\alpha\in(0,1)$, provided a uniform Ricci curvature bound $|\Ric_g|\lesl \lambda$.

\begin{definition}
[Harmonic radius and regularity scale] Let $(X^d, g)$ be a Riemannian manifold, and let $p\in X^d$.
\begin{enumerate}
	\item The harmonic radius $r_h(p)$ at $p\in X^d$
 is the supremum of $r\in(0,1)$ for which there exists a $C^1$-harmonic coordinate system $\Phi$ on $B_r(p)$ with  $\|\Phi\|_{C^1} < 10^{-6}$.

	\item Given $k\in\dZ_+$, the $C^k$-regularity scale $\fs_k(p)$ at $p$ is the supremum of $r\in(0,1)$ for which there exists a $C^k$-harmonic coordinate system $\Phi$ on $B_r(p)$   that satisfies $\|\Phi\|_{C^k} < 10^{-6}$ in the sense that
	 \begin{align}
		|g_{ij} - \delta_{ij}|_{C^0(B_r(p))}
	+ \sum\limits_{m = 1}^k |\nabla^m g_{ij}|_{C^0(B_r(p))} < 10^{-6}.	\end{align}

\end{enumerate}
\end{definition}

By standard elliptic regularity theory, if $g$ is an Einstein metric with a uniform bound on the Ricci curvature $|\Ric_g| \lesl \Lambda$, 
then a uniform $C^1$-bound on the metric in harmonic coordinates implies a uniform $C^k$-bound for any fixed $k\in\dZ_+$.

\subsection{Rigidity results under pinching assumptions}
\label{ss:pinching}

We begin by proving Theorem \ref{t:large-circle}, as stated in Section \ref{s:introduction}. This theorem provides a rigidity result for a complete Poincar\'e-Einstein manifold $(X^{n+1}, g)$ whose conformal infinity is $(S^1 \times S^n, [\wg_{\lambda}])$, under the assumption that the normalized product metric $\wg_{\lambda}$    has sufficiently large diameter. For convenience, we restate the result in Theorem \ref{t:large-circle} here as follows.
\begin{theorem} \label{t:rigidity-lambda_0}
 Given $n\gesl 3$, there exists a dimensional constant $\lambda_0 = \lambda_0(n) \gesl 1$ such that the following property holds.
Let $(X^{n + 1}, g)$ be a complete Poincar\'e-Einstein manifold with conformal infinity $(S^1 \times S^{n-1}, [\wg_{\lambda}])$. If the normalized product metric $\wg_{\lambda}$ satisfies $\lambda \gesl \lambda_0$, then 
  $(X^{n + 1}, g)$ is isometric to the hyperbolic manifold $(S^1\times D^n, g_{\lambda})$ with constant sectional curvature $-1$. 
\end{theorem}

\begin{remark}  The family of normalized product metrics $\{\wg_{\lambda}\}_{\lambda \gesl \lambda_0}$ on $S^1 \times S^{n - 1}$ is non-compact even in the Gromov-Hausdorff topology   since $\diam_{\wg_{\lambda}}(S^1 \times S^{n - 1}) \to \infty$ as $\lambda \to \infty$.	
\end{remark}

\begin{proof}
First, let us recall a known estimate on the Yamabe constant of $(S^1\times S^{n-1}, [\wg_{\lambda}])$. As observed in \cite[Section 2]{Schoen}, for any $\eta > 0$, there exists $\lambda_0  = \lambda_0(n,\eta) > 0$ such that 
for any $\lambda \gesl \lambda_0$,
\begin{align}
	\frac{\mathcal{Y}(S^1\times S^{n-1}, [\wg_{\lambda}])}{\mathcal{Y}(S^n, [g_c])} \gesl 1 - \eta. \label{e:eta-pinched-Yamabe-on-product}
\end{align}
By Theorem \ref{t:curvature-pinching}, there exists a constant $\eta_0 = \eta_0(n) > 0$ such that if \eqref{e:eta-pinched-Yamabe-on-product} holds for $\eta \lesl \eta_0$, 	
	then $-2\lesl \sec_g \lesl 0$. Now one can choose $\lambda_0 \equiv \lambda_0(n, \eta_0) > 0$. Applying Theorem \ref{t:uniqueness-nonpositively-curved}, $(X^{n + 1}, g)$ must be isometric to a  hyperbolic manifold $(S^1\times D^n, g_{\lambda})$ when the conformal infinity $(S^1\times S^{n-1}, [\wg_{\lambda}])$ satisfies $\lambda \gesl \lambda_0$.
\end{proof}

The next theorem (cf. Theorem \ref{t:rigidity-for-small-weyl-energy} in Section \ref{s:introduction}) provides a rigidity result for Poincar\'e-Einstein manifolds whose (scale-invariant) Weyl energy is sufficiently small and whose conformal infinity is given by the conformally flat product $S^1 \times S^{n - 1}$.

 \begin{theorem} \label{t:small-Weyl-energy}  For any $n \gesl 3$, there exists a dimensional constant $\epsilon = \epsilon (n) > 0$ such that the following property holds. Let $\lambda > 0$, and let $(X^{n+1}, g)$ be a complete Poincar\'e-Einstein manifold 
  whose conformal infinity is $(S^1 \times S^{n-1}, [\wg_{\lambda}])$. If \begin{align}
	\int_{X^{n+1}}|W_g|^{\frac{n + 1}{2}} \dvol_g   \lesl \epsilon \cdot \left(\mathcal{Y}(S^1 \times S^{n-1}, [\wg_{\lambda}])\right)^{\frac{n}{2}}, \label{e:small-Weyl-energy-2}
\end{align} 	
 then $(X^{n + 1}, g)$ must be isometric to the hyperbolic manifold $(S^1 \times D^n, g_{\lambda})$.
 \end{theorem}
 The basic strategy in proving Theorem \ref{t:small-Weyl-energy} is to obtain a pointwise pinching estimate for the Weyl curvature under the integral pinching condition, as in Proposition \ref{p:eps-regularity}. 
 The key point in the theorem is that the pinching constant $\epsilon = \epsilon (n)$ is a {\bf dimensional constant}. In particular, $\epsilon$ is independent of the volume collapsing rate of the cylinder $(S^1 \times S^{n - 1}, \wg_{\lambda})$. This is significant because, as $\lambda \to 0$, the Yamabe constant $\mathcal{Y}(S^1\times S^{n - 1}, [\wg_{\lambda}])$ tends to zero, implying that there exists no uniform Sobolev constant on $S^1 \times S^{n - 1}$, and consequently,    one would not expect the existence of a uniform Sobolev constant 
 on the Poincar\'e-Einstein filling $(X^{n + 1}, g)$. Thus, the usual Moser iteration scheme cannot be directly applied to obtain a pointwise estimate of the Weyl curvature of the Einstein metric $g$, as in \eqref{e:Weyl-arbitrarily-small}, under the smallness assumption of its $L^p$ bound. Nevertheless, in Proposition \ref{p:eps-regularity}, we can establish the estimate \eqref{e:Weyl-arbitrarily-small} using arguments from collapsing geometry, which leads to the conclusion of Theorem \ref{t:small-Weyl-energy}.

 In the special case when $\lambda$ tends to zero, the Yamabe constant $\mathcal{Y}(S^1 \times S^{n-1}, {\hat g}_{\lambda})$ tends to zero also; in this case one would expect the Sobolev constant 
 of the Poincar\'e Einstein metric filling metrics $g$ tends to zero also.

% \begin{remark}In the theorem, the pinching constant 

%the Weyl energy bound $\delta = \delta(\lambda|n) > 0$ in \eqref{e:small-Weyl-energy-2} satisfies $\lim\limits_{\lambda \to 0} \delta(\lambda | n) = 0$. Note that when $\lambda \to 0$, the local volumes of the Poincar\'e-Einstein fillings $(X^{n + 1}, g)$ of $(S^1\times S^{n - 1}, [\wg_{\lambda}])$ fail to have a uniform lower bound. For example, $(S^1 \times D^n, g_{\lambda})$ is volume collapsing when $\lambda \to  0$.\end{remark}
 
  Theorem \ref{t:small-Weyl-energy} 
follows from a combination of Theorem \ref{t:uniqueness-nonpositively-curved}
and the following $\epsilon$-regularity result for the Weyl curvature, which is new and of independent interest. 
 
 \begin{proposition}
 	\label{p:eps-regularity}
Given any $n \gesl 2$ and $\eta > 0$, there exists some $\epsilon = \epsilon(n, \eta) > 0$ such that the following holds. 	If $(M^n , [\wg])$ is a conformal manifold with positive Yamabe constant and it bounds a complete Poincar\'e-Einstein manifold $(X^{n + 1}, g)$ that satisfies 
\begin{align}
	\int_{X^{n+1}}|W_g|^{\frac{n + 1}{2}} \dvol_g \lesl \epsilon \cdot \left(\mathcal{Y}(M^n, [\wg])\right)^{\frac{n}{2}}, \label{e:epsilon-weyl-energy}
\end{align}
 	then 
 	\begin{align}
 		\sup\limits_{X^{n+1}}|W_g|\lesl \eta. \label{e:Weyl-arbitrarily-small}
 	\end{align}
 \end{proposition}
 
 \begin{remark}Proposition \ref{p:eps-regularity} holds for general Poincar\'e-Einstein manifolds whose conformal infinity is of positive Yamabe type, making it widely applicable and of independent interest. A notable feature of the proposition is that the pinching constant $\epsilon > 0$ depends only on the dimension and the chosen parameter. In particular, the pinching constant $\epsilon$ is independent of the   collapse of the space, since the Yamabe constant $\mathcal{Y}(M^n, [\wg])$ of the conformal infinity is allowed to become arbitrarily small.   
  \end{remark}
 
 The proof of the proposition relies on a preliminary $\epsilon$-regularity result below. Let $(X^d, g)$
 be any complete Einstein manifold with $|\Ric_g|\lesl d - 1$.
 For any $p\in X^d$ and $r > 0$, we define 
 \begin{align}
 		\mathcal{E}_{|\Rm_g|}(p,r) \equiv \frac{\Vol_{-1}(B_r)}{\Vol_g(B_r(p))}	\int_{B_r(p)} |\Rm_g|^{\frac{d}{2}} \dvol_g,
 \end{align}
 where $\Vol_{-1}(\cdot)$ denotes the volume measure on the hyperbolic space $(\dH^d, h)$ with constant curvature $-1$.
 By volume comparison, for any $p \in X^{n + 1}$, the energy $\mathcal{E}_{|\Rm_g|}(p,r)$ is non-decreasing in $r$.

  \begin{lemma}{\rm(\cite[theorem 4.4]{Anderson-L2})}
 	\label{l:standard-eps-reg}
 	Given $d \gesl 3$, there exists a dimensional constant $\tau_0 = \tau_0(d) > 0$ such that the following property holds. If $(X^d, g, p)$ is a complete Einstein manifold with $|\Ric_g|\lesl d - 1$ and satisfies \begin{align}
 		\mathcal{E}_{|\Rm_g|}(p,2r) < \tau_0
 		\end{align}
 		for some $r > 0$, then 
 		\begin{align}
 			r^2 \sup\limits_{B_r(p)}|\Rm_g| \lesl 1.
 		\end{align}
 \end{lemma}
  In our actual applications, we take $d \equiv n + 1$.
With the above curvature estimate, we will now prove Proposition \ref{p:eps-regularity}. 
\begin{proof}[Proof of Proposition \ref{p:eps-regularity}]
    Let $\mathcal{Y}_0$ be the dimensional constant $\mathcal{Y}_0\equiv \mathcal{Y}(S^n, [g_c])$. By Theorem \ref{t:volume-comparison}, for any $p\in X^{n+1}$ and $r > 0$, 
\begin{align}
  	\left(\frac{\mathcal{Y}(M^n, \wg)}{\mathcal{Y}_0}\right)^{\frac{n}{2}}\lesl 	\frac{\Vol_g(B_r(p))}{\Vol_{-1}(B_r)}\lesl 1. \label{e:volume-comparison-yamabe}
\end{align}
Let $\epsilon > 0$ be the number in 
 \eqref{e:epsilon-weyl-energy} which will be determined later. 
Then the assumption \eqref{e:epsilon-weyl-energy} and the volume comparison \eqref{e:volume-comparison-yamabe} imply that for any $p\in X^{n + 1}$ and $r \in (0,1)$:
\begin{align}
\frac{\Vol_{-1}(B_r)}{\Vol_g(B_r(p))}	\int_{B_r(p)} |W_g|^{\frac{n + 1}{2}} \dvol_g \lesl \frac{\Vol_{-1}(B_r)}{\Vol_g(B_r(p))}	\int_{X^{n+1}}|W_g|^{\frac{n + 1}{2}} \dvol_g \lesl \epsilon \cdot \mathcal{Y}_0^{\frac{n}{2}}.\label{e:weyl-energy-r-ball}\end{align}
 Therefore,  by the definition of the curvature energy $\mathcal{E}_{|\Rm_g|}$, for any $p\in X^{n + 1}$ and $r \in (0,1)$,
\begin{align}
\mathcal{E}_{|\Rm_g|}(p,r)  \lesl C(n)\cdot (\epsilon \cdot \mathcal{Y}_0^{\frac{n}{2}} + r^{n + 1}).
\end{align}
In the above estimate, if we choose any $\epsilon > 0$ and $s > 0$ that satisfy
\begin{align}
	\epsilon < \frac{\tau_0}{10 C(n) \cdot \mathcal{Y}_0^{\frac{n}{2}}}\quad \text{and}\quad s <  \frac{1}{10}\left(\frac{\tau_0}{C(n)}\right)^{\frac{1}{n + 1}}, \label{e:small-epsilon}
\end{align}
then 
$\mathcal{E}_{|\Rm_g|}(p, 2s) < \tau_0$.
Applying Lemma \ref{l:standard-eps-reg}, one obtains the pointwise curvature estimate in a ball of definite size of radius,
\begin{align}
	\sup\limits_{B_s(p)}|\Rm_g| \lesl s^{-2}.
\end{align}
Since $p$ is an arbitrary point in $X^{n + 1}$, we have that the uniform curvature bound
\begin{align}
		\sup\limits_{X^{n + 1}}|\Rm_g| \lesl s^{-2}. \label{e:bounded-curvature}
\end{align}

We are now in a position to establish the desired pointwise pinching estimate \eqref{e:Weyl-arbitrarily-small} for Weyl curvature.  
Notice that the Poincar\'e-Einstein manifold $(X^{n + 1}, g)$ may still be collapsing even if its full Riemannian curvature $\Rm_g$ is uniformly bounded, as shown in \eqref{e:bounded-curvature}. A typical example is the hyperbolic manifold $(S^1 \times D^n, g_{\lambda})$, which is collapsing as $\lambda \to 0$. As noted in the remark following Theorem \ref{t:small-Weyl-energy}, this type of volume collapse presents an obstruction to upgrading the integral bound on the Weyl curvature in \eqref{e:epsilon-weyl-energy} to a pointwise estimate as in \eqref{e:Weyl-arbitrarily-small}. To overcome this difficulty, we observe that the uniform curvature estimate \eqref{e:bounded-curvature} allows us to lift the metric to a non-collapsing local covering space, on which the integral pinching condition on the Weyl curvature can be upgraded to a pointwise one. We will implement this construction in two steps.

First, we will invoke the fact that local universal covers of $X^{n + 1}$ is {\it non-collapsing} provided that the pointwise curvature estimate \eqref{e:bounded-curvature} holds. Indeed, since $(X^{n + 1}, g)$ has uniformly bounded curvature as obtained in \eqref{e:bounded-curvature}, by a result in \cite{CFG}, there exist two dimensional constants 
\begin{align}
 r_0 = r_0(n) \in (0, s)\quad \text{and}\quad i_0 = i_0(n) > 0,	
 \end{align}
 such that for any $p\in X^{n + 1}$, 
on the Riemannian universal covering space,   
\begin{align}
\pi: \left(\widetilde{B_{10r_0}(p)}, \tilde{g}, \tilde{p}\right) \to (B_{10r_0}(p), g), \quad \pi(\tilde{p}) = p,
\end{align} one has the uniform injectivity radius estimate 
\begin{align}
	\Injrad_{\tilde{g}}(\tilde{q}) \gesl i_0, \quad \forall \ \tilde{q}\in B_{8r_0}(\tilde{p}).
\end{align}
Then by the harmonic radius estimate in 
\cite[lemma 2.2]{Anderson-Ricci}, there exists some dimensional constant $\mu = \mu(n) > 0$ such that for any $\tilde{q}\in B_{8r_0}(\tilde{p})$, 
\begin{align}
	|\tilde{g}_{ij} - \delta_{ij}|_{C^0(B_{\rho_0}(\tilde{q}))}
	+ |\nabla \tilde{g}_{ij}|_{C^0(B_{\rho_0}(\tilde{q}))} < 10^{-6},
\end{align} 
where $\rho_0 \equiv \mu \cdot r_0 > 0$, $\tilde{g}_{ij}\equiv \tilde{g}(\nabla u_i, \nabla u_j)$, and $\Phi\equiv (u_1,\ldots, u_{n+1})$ are harmonic coordinates on $B_{\rho_0}(\tilde{q})$.
Since $\tilde{g}$ is an Einstein metric on the non-collapsing cover $(\widetilde{B_{10r_0}(p)}, \tilde{p})$, the standard bootstrapping argument for non-collapsing Einstein metrics with bounded geometry, for any $k \in \dN_0$, there exists a uniform constant $Q_k = Q_k(n) > 0$ such that 
$\|\Rm_{\tilde{g}}\|_{C^k(B_{6r_0}(\tilde{p}))} \lesl Q_k$. Therefore, $\tilde{g}$ has $C^k$-uniformly bounded geometry on $B_{6r_0}(\tilde{p})$.

Next, we will obtain an $L^{\frac{n + 1}{2}}$-average estimate for the Weyl curvature on the non-collapsing cover $(\widetilde{B_{10r_0}(p)}, \tilde{p})$.
Let us denote by $\Omega(\tilde{p})$ the fundamental domain of $(\widetilde{B_{10r_0}(p)}, \tilde{p})$
that contains $\tilde{p}$. For any $t < 10r_0$, we define 
\begin{align}
	U(t, \tp) \equiv \pi^{-1}(B_t(p))\cap \Omega(\tp),\label{e:def-U_t}
\end{align}
which is the liftings of $B_t(p)$ that intersect with the fundamental domain $\Omega(\tp)$.
We denote $G \equiv \pi_1(B_{10r_0}(p))$ and we choose $t\equiv 2r_0 > 0$ in \eqref{e:def-U_t}. One can observe that the closed ball $\overline{B_{r_0}(\tp)}$ can be covered by a finite set 
\begin{align}
	\left\{U_m \equiv \gamma_m \cdot U(2r_0, \tp)| \gamma_m \in G,\ U_m\cap B_{r_0}(\tp) \neq \emptyset \right\} ,\quad 1\lesl m \lesl N.
\end{align}
Moreover, by triangle inequality, 
\begin{align}
	\bigcup\limits_{m = 1}^N U_m \subset B_{6r_0}(\tp). \label{e:inclusion-Um}
\end{align}
Now we are ready to establish the desired $L^{\frac{n + 1}{2}}$-average estimate for the Weyl curvature on the covering space. By \eqref{e:weyl-energy-r-ball}, for any $1\lesl m \lesl N$,
\begin{align}
	\frac{\Vol_{-1}(B_{2r_0})}{\Vol_{\tg}(U_m)}	\int_{U_m} |W_{\tg}|^{\frac{n + 1}{2}} \dvol_{\tg} \lesl  \epsilon \cdot \mathcal{Y}_0^{\frac{n}{2}}. \label{e:integral-weyl-Um}
\end{align}
Applying Bishop-Gromov volume comparison, we have that
\begin{align}
	\frac{\Vol_{-1}(B_{r_0})}{\Vol_{\tilde{g}}(B_{r_0}(\tilde{p}))}\int_{B_{r_0}(\tilde{p})}|W_{\tilde{g}}|^{\frac{n + 1}{2}}\dvol_{\tilde{g}} \lesl 	\frac{\Vol_{-1}(B_{6r_0})}{\Vol_{\tilde{g}}(B_{6r_0}(\tilde{p}))}\int_{B_{r_0}(\tilde{p})}|W_{\tilde{g}}|^{\frac{n + 1}{2}}\dvol_{\tilde{g}}.
\end{align}
Since $\{U_m\}_{m = 1}^N$ covers $B_{r_0}(\tp)$ and 
the inclusion \eqref{e:inclusion-Um} holds, the above inequality becomes 
\begin{align}
	\frac{\Vol_{-1}(B_{r_0})}{\Vol_{\tilde{g}}(B_{r_0}(\tilde{p}))}\int_{B_{r_0}(\tilde{p})}|W_{\tilde{g}}|^{\frac{n + 1}{2}}\dvol_{\tilde{g}} & \ 	  \lesl \frac{\Vol_{-1}(B_{6r_0})}{\sum\limits_{m = 1}^N \Vol_{\tilde{g}}(U_m)}\sum\limits_{m = 1}^N \int_{U_m}|W_{\tilde{g}}|^{\frac{n + 1}{2}}\dvol_{\tilde{g}}
	\nonumber\\
	&\ \lesl  \epsilon\cdot \mathcal{Y}_0^{\frac{n}{2}} \cdot \frac{\Vol_{-1}(B_{6r_0})}{\Vol_{-1}(B_{2r_0})},	\label{e:average-estimate}
\end{align}
where the second inequality follows from \eqref{e:integral-weyl-Um}. We notice that, in the above estimate, 
both the spherical Yamabe constant $\mathcal{Y}_0$ and the hyperbolic volume quotient $\frac{\Vol_{-1}(B_{6r_0})}{\Vol_{-1}(B_{2r_0})}$
are dimensional constants.

Finally, combining the property that $B_{6r_0}(\tp)$
 has $C^k$-bounded geometry (for any sufficiently large $k$)
 and the $L^{\frac{n + 1}{2}}$-average estimate \eqref{e:average-estimate}, for any given $\eta > 0$, one can choose a sufficiently small constant $\epsilon = \epsilon(n, \eta) > 0$ such that 
 \begin{align}
 	\sup\limits_{B_{r_0}(\tp)} |W_{\tg}| \lesl \eta.
 \end{align}
Here $\epsilon > 0$ is also required to satisfy \eqref{e:small-epsilon}. Since the covering map $\pi:\widetilde{B_{10r_0}(p)} \to B_{10r_0}(p)$ is locally isometric and $p = \pi(\tilde{p})$ is an arbitrary point in $X^{n + 1}$, we conclude that 
\begin{align}
	\sup\limits_{X^{n + 1}} |W_g| \lesl \eta,
\end{align}
which completes the proof.
\end{proof}

We now prove Theorem \ref{t:small-Weyl-energy}.
\begin{proof}[Proof of Theorem \ref{t:small-Weyl-energy}]
Notice that for any $n\gesl 2$, one can find $\eta_0 = \eta_0(n) > 0$ such that if 
\begin{align}
	\sup\limits_{X^{n + 1}}|W_g|\lesl \eta_0
\end{align} on a Poincar\'e-Einstein manifold $(X^{n + 1}, g)$ with $\Ric_g \equiv - n g$, then $\sec_g\lesl 0$ on $X^{n+1}$.

In the context of the theorem, for any $\lambda > 0$,
let $(X^{n+1}, g)$ be a Poincar\'e-Einstein manifold with conformal infinity $(S^1\times S^{n - 1}, [\wg_{\lambda}])$.
 Let $\eta_0 =\eta_0(n) > 0$ be the  dimensional constant as above and we set $\epsilon = \epsilon(n, \eta_0) > 0$ as the constant 
in Proposition \ref{p:eps-regularity}. If $(X^{n+1}, g)$ satisfies  
\begin{align}
	\int_{X^{n + 1}} |W_g|^{\frac{n + 1}{2}} \dvol_g \lesl \epsilon\cdot \left(\mathcal{Y}(S^1\times S^{n - 1}, [\wg_{\lambda}])\right)^{\frac{n}{2}},
\end{align}
then by Proposition \ref{p:eps-regularity} and the above computations, we obtain that 
$\sec_g \lesl 0$ on $X^{n + 1}$. 
Applying Theorem \ref{t:uniqueness-nonpositively-curved}, we obtain the desired rigidity result.
 \end{proof}

 Theorem \ref{t:small-Weyl-energy} provides hyperbolic rigidity 
 when the Weyl energy is sufficiently small relative to the Yamabe constant. We pose a question regarding rigidity under the smallness of the Weyl energy, independent of the Yamabe constant; see Question \ref{q:4D-pinching} in Section \ref{s:discussions}.
 
We remark that it is natural to pose the question on the relationship between the pinching constants $\epsilon$ and $\delta$: the pinching of the Weyl energy 
 \begin{align}
 	\int_{X^{n+1}} |W_g|^{\frac{n + 1}{2}}\dvol_g < \epsilon,\label{e:Weyl-energy-pinching}
 \end{align}
 and the pinching of the Yamabe constants 
 \begin{align}
 	\frac{\mathcal{Y}(M^n, [\wg])}{\mathcal{Y}(S^n, [g_c])} > 1 - \delta.\label{e:Yamabe-pinching}
 \end{align} 

 In our context, we consider the case $(M^n , [\wg]) \equiv (S^1 \times S^{n - 1}, [\wg_{\lambda}])$ for some $\lambda > 0$. If \eqref{e:Yamabe-pinching} holds for some sufficiently small $\delta$ (or equivalently, for some sufficiently large $\lambda$), then by Theorem \ref{t:large-circle} (or Theorem \ref{t:rigidity-lambda_0}), any Poincar\'e-Einstein filling $(X^{n + 1}, g)$ must be hyperbolic, and thus $W_g \equiv 0$ on 
 $X^{n + 1}$. In particular, \eqref{e:Weyl-energy-pinching} holds.
We note again that, given $\lambda_0 > 0$,  the family of conformal structures
\begin{align}
 	\{(S^1\times S^{n - 1}, [\wg_{\lambda}]): \lambda \gesl \lambda_0\}
 \end{align}
is {\it non-compact} since the diameter $\diam_{\wg_{\lambda}}(S^1 \times S^{n - 1})$ tends to infinity as $\lambda \to \infty$. The direction of the implication from \eqref{e:Weyl-energy-pinching} to \eqref{e:Yamabe-pinching} is not known in general.

We now consider a {\it compact family} of conformal structures $\{(M^n, [\wg])\}$, where the representatives $\wg$ have uniformly bounded geometry.
In this situation, if we consider the case when the pinching constant $\delta > 0$ in \eqref{e:Yamabe-pinching} is sufficiently small 
with respect to some {\it geometric boundedness} of this family of conformal structures, then by a compactness result in \cite[theorem 1.1]{CGJQ}, the Fefferman-Graham metrics on the compactified Poincar\'e-Einstein fillings constitute a compact family in the Cheeger-Gromov topology. Therefore, one obtains \eqref{e:Weyl-energy-pinching}.
On the other hand, if \eqref{e:Weyl-energy-pinching} holds and one assumes that the Yamabe constant of the conformal infinity has a uniform lower bound (which amounts to a scale-free volume {\it non-collapse} assumption), 
then, by the same result \cite[theorem 1.1]{CGJQ}, one obtains Cheeger-Gromov compactness of both the Poincar\'e-Einstein fillings and their Fefferman-Graham compactified metrics, which in turn implies \eqref{e:Yamabe-pinching}. Again, in this situation, the pinching constant depends on the uniform geometric boundedness of the conformal infinity.

In summary, a particularly interesting aspect of Theorem \ref{t:small-Weyl-energy}  (cf. Theorem \ref{t:rigidity-for-small-weyl-energy}) is that, similar to Theorem \ref{t:rigidity-lambda_0} (cf. Theorem \ref{t:large-circle}),   
the validity of the rigidity result does not rely on the uniform geometric boundedness of the conformal infinity, since the $\epsilon$-regularity result (Proposition \ref{p:eps-regularity}) holds in a very general setting, including {\it collapsed} (and thus non-compact families of) Poincar\'e-Einstein spaces.

 \subsection{Quantitative rigidity and compactness}
\label{ss:quantitative-rigidity}

We next prove a quantitative rigidity result for the Poincar\'e-Einstein filling $(X^{n + 1}, g)$ of $S^1 \times S^{n - 1}$ which provides a quantitative description of $g$ when the conformal infinity is close to the standard conformal class of $S^1 \times S^{n - 1}$.

To formulate our quantitative rigidity result, we introduce the following weighted H\"older norm, defined for tensor fields on a Poincar\'e-Einstein manifold. Let $(X^{n + 1}, g)$ be a Poincar\'e-Einstein manifold whose conformal infinity $(M^n, [\wg])$ has a positive Yamabe constant. Let $\bg^* \equiv \varrho^2 g$ denote the Fefferman-Graham compactification of $g$. 

We consider the space $\mathscr{S}^2(X^{n + 1})$ of symmetric $(0,2)$-tensor fields on $X^{n + 1}$. For any $k\in\dN$ and $\alpha\in (0,1)$, 
let $C^{k,\alpha}(X^{n + 1}, \mathscr{S}^2 (X^{n + 1}))$
denote the space of symmetric $(0,2)$-tensor fields $h$ such that  $\|h\|_{C^{k,\alpha}(X^{n + 1})} < \infty$. Furthermore, for any $\mu\in \dR$,
we define 
the \emph{weighted H\"older space $C_{\mu}^{k,\alpha}
	(X^{n + 1}, \mathscr{S}^2(X^{n + 1}))$}:
\begin{align}
\begin{split}
		C_{\mu}^{k,\alpha}
	(X^{n + 1}, \mathscr{S}^2(X^{n + 1}))
	& \ \equiv \varrho^{\mu} \cdot 
		C^{k,\alpha}
	(X^{n + 1}, \mathscr{S}^2(X^{n + 1}))
\nonumber \\ & \ = \left\{\varrho^{\mu}   h: h \in 		C^{k,\alpha}
	(X^{n + 1}, \mathscr{S}^2(X^{n + 1})) \right\}.
	\end{split}
\end{align}
 For any $h\in C_{\mu}^{k,\alpha}
	(X^{n + 1}, \mathscr{S}^2(X^{n + 1}))$
we define the \emph{weighted H\"older norm} by 
\begin{align}
	\|h\|_{C_{\mu}^{k,\alpha}
	(X^{n + 1})} 
	\equiv \|\varrho^{-\mu}\cdot h\|_{C^{k,\alpha}(X^{n + 1})}.
\end{align}

The following theorem is the main result of this subsection. 
 \begin{theorem}\label{t:quantitative-rigidity} Given $n \gesl 3$, there exists a positive number $\lambda_0 = \lambda_0 (n) > 1$ such that for any $\lambda \gesl \lambda_0$ there exists some $\delta_0 = \delta_0 (\lambda, n) > 0$ that satisfies the following property.
  If a smooth metric $\wg$ on $S^1 \times S^{n - 1}$   satisfies 
  \begin{align}
  	\|\wg - \wg_{\lambda}\|_{C^6(S^1 \times S^{n - 1})} < \delta \lesl \delta_0,
  \end{align}
  then the conformal manifold
   $(S^1 \times S^{n - 1}, [\wg])$ has a unique Poincar\'e-Einstein filling $(X^{n + 1}, g)$.  Furthermore, there exists a diffeomorphism $\Phi: S^1\times D^n \to X^{n + 1}$ such that
 \begin{align}
\|\Phi^* g - g_{\lambda}\|_{C_{\mu}^{2,\alpha}(X^{n+1})} \lesl \tau(\delta|n,\lambda,\alpha, \mu), 	
 \end{align}
for any $\mu\in(0,1)$, where $\tau$ is a function of $\delta$, $n$, $\lambda$, $\alpha$, $\mu$, and satisfies $\lim\limits_{\delta \to 0} \tau = 0$.

\end{theorem}

For a conformal manifold $(S^n, [\wg])$, if $\wg$ is sufficiently close to a round metric on $S^n$, then a similar result of the quantitative rigidity of the Poincar\'e-Einstein filling has been established earlier in \cite{CGQ, CGJQ}. Their proof essentially relies on compactness results developed for that setting. 
The proof of Theorem \ref{t:quantitative-rigidity} is based on both the rigidity result given by Theorem \ref{t:uniqueness-large-circle-factor} and a compactness result adapted in our context, as stated in Proposition \ref{p:compactness-small-weyl} below.

We start with a perturbation existence result which is due to Lee and was established via the implicit function theorem and a version of the Fredholm theory in the context of Poincar\'e-Einstein manifolds; see \cite[theorem A]{Lee-fredholm}.   
\begin{theorem} \label{t:Lee-existence}
Given $n \gesl 3$, there exists a positive number $\Lambda_0 = \Lambda_0 (n) > 1$ such that for any $\lambda \gesl \Lambda_0$ there exists a uniform constant $\delta_0 = \delta_0 (\lambda, n) > 0$ that satisfies the following property.
  If a smooth metric $\wg$ on $S^1 \times S^{n - 1}$ satisfies 
  \begin{align}
  	\|\wg - \wg_{\lambda}\|_{C^6(S^1 \times S^{n - 1})} < \delta \lesl \delta_0,
  \end{align}
  then the conformal manifold
   $(S^1 \times S^{n - 1}, [\wg])$ admits a Poincar\'e-Einstein filling $(S^1 \times D^n, g)$ that satisfies $\|g - g_{\lambda}\|_{C_{\mu}^{2,\alpha}(X^{n+1})} \lesl \tau(\delta|n,\lambda)$.
 \end{theorem}
 
In Theorem \ref{t:Lee-existence}, 
 the model metric $g_{\lambda}$ is hyperbolic and in particular negatively curved. Therefore, 
\cite[theorem A]{Lee-fredholm} naturally applies here.

We denote by $\mathfrak{B}$ the moduli space of all smooth Riemannian metrics on $S^1\times S^{n - 1}$, equipped with the H\"older norm $\|\cdot\|_{C^{2,\alpha}(S^1\times S^{n - 1})}$, we 
also denote by $\mathfrak{E}$ the moduli space of Poincar\'e-Einstein metrics on $S^1 \times D^n$, equipped with the weighted H\"older norm $\|\cdot \|_{C_{\mu}^{2,\alpha}(S^1 \times D^n)}$ such that for any $g\in \mathfrak{E}$,
the conformal infinity of $(S^1 \times D^n, g)$ is $(S^1 \times S^{n - 1}, [\wg])$
for some smooth conformal metric $\wg\in \mathfrak{B}$.  Then both $\mathfrak{B}$ and $\mathfrak{E}$ are Banach spaces.  
 Lee's Fredholm theory established in \cite[theorem A]{Lee-fredholm} essentially states that the boundary evaluation map $\Psi: \mathfrak{E} \to \mathfrak{B}$
 is a local diffeomorphism at the model hyperbolic metric $g_{\lambda}$.

  %for any $\epsilon > 0$, there exists a number $\tau  > 0$ such that the following holds. If a Poincar\'e-Einstein metric $g\in \mathfrak{P}$ with $\mathscr{C}(S^1 \times D^n, g) = (S^1\times S^{n - 1}, [\wg])$ and $\wg\in \mathfrak{B}$ satisfies $\sec_g \lesl 0$, then any metric $\widehat{h} \in \mathfrak{B}_{\tau}(\wg)$  bounds a Poincar\'e-Einstein metric $h \in \mathfrak{P}_{\epsilon}(g)$, where $\mathfrak{B}_{\tau}(\wg)\subset \mathfrak{B}$ is the metric ball of radius $\tau$ centered at $\wg$, and $\mathfrak{P}_{\epsilon}(g)\subset \mathfrak{P}$ is the metric ball of radius $\epsilon$ centered at $g$, respectively. Furthermore, it is known that, via an deformation argument, the solution in $\mathfrak{P}_{\epsilon}(g)$ is unique; see \cite{CGQ, CGJQ} for instance. 

 We will apply the following compactness result to prove that for any metric that is $C^{5,\alpha}$-close to $\wg_{\lambda}$, the mapping $\Psi$ has a unique pre-image in $\mathfrak{E}$.
The proposition below provides a compactness result for the Fefferman-Graham metrics (as given in Lemma \ref{l:FG-metric}) by assuming the geometric boundedness of the conformal infinity and a small $L^{\infty}$-bound of the Poincar\'e-Einstein metric $g$. 
\begin{proposition}
	\label{p:compactness-small-weyl}  
Given $n \gesl 3$, $k \gesl 6$, $\alpha \in (0,1)$, $\tau > 0$ and $D > 0$, there exist positive numbers $\epsilon_0 = \epsilon_0(\tau, D, k, \alpha, n) < 1$,  $s_0 = s_0(\tau, D, k, \alpha, n) < 1$ and $D_0 = D_0(\tau, D, k, \alpha, n) < \infty$  such that the following property holds. 
Let $(X^{n + 1}, g)$ be a complete Poincar\'e-Einstein manifold with $\mathscr{C}(X^{n + 1}, g) = (M^n, [\wg])$. Let us denote by $\bg^*$ the Fefferman-Graham metric on $\overline{X^{n+1}}$ associated to a representative $\wg \in [\wg]$ with positive scalar curvature. If the conformal infinity $(M^n,\wg)$ has a uniform diameter bound $\diam_{\wg}(M^n)\lesl D$ and has uniformly bounded $C^6$-geometry $\fs_6(x) \gesl \tau > 0$ for any $x\in M^n$, and the Poincar\'e-Einstein manifold $(X^{n+1},g)$ satisfies 
\begin{align}\sup\limits_{X^{n+1}}|W_g|\lesl \epsilon_0, \label{e:Weyl-pointwise-small}\end{align} then  $\diam_{\bg^*}(\overline{X^{n+1}}) \lesl D_0$ and for any $p\in \overline{X^{n+1}}$, the $C^{2,\alpha}$-regularity scale $	\fs_{2,\alpha}(p)$ at $p$ has a uniform lower bound,
\begin{align}
	\fs_{2,\alpha}(p) \gesl s_0 > 0.
\end{align}
In particular,
\begin{align}\sup\limits_{\overline{X^{n + 1}}}|\Rm_{\bg^*}| \lesl C_0 \cdot s_0^{-2},\end{align}for some dimensional constant $C_0 = C_0(n) > 0$.
\end{proposition}

In this proposition, there is no topological restriction on the boundary $M^n$. The same compactness result for $M^n$ diffeomorphic to $S^n$ 
 was obtained in \cite[theorems 1.1 and 1.2]{CGJQ}; however, the topological assumption was not needed in their proof, so we simply state the compactness result as above.

\begin{proof}
	[Proof of Theorem \ref{t:quantitative-rigidity}]	
 	
We start with some simple facts. Let $\lambda > 1$ be any fixed number. For any $\eta > 0$, there exists $\delta = \delta(\eta|\lambda, n) > 0$ with $\lim\limits_{\eta \to 0}\delta = 0$ such that if  
$\|\wg  - \wg_{\lambda}\|_{C^{k,\alpha}(S^1 \times S^{n - 1})} < \delta$, then 
\begin{align}
	 \frac{\mathcal{Y}(S^1 \times S^{n - 1}, [\wg])}{\mathcal{Y}(S^1 \times S^{n - 1}, [\wg_{\lambda}])} \gesl  1 - \eta.
\end{align}	
Let $\epsilon_0  = \epsilon_0(n) > 0$ be the constant in Proposition \ref{p:compactness-small-weyl}. By the curvature estimate in Theorem \ref{t:curvature-pinching} and the above comparison, one can find a dimensional constant $\ul_0 = \ul_0(n) > 0$ and for every $\lambda \gesl \ul_0 >0$, there exists a number
 $\delta = \delta(\lambda, n) > 0$ such that
 if $(S^1\times S^{n-1}, [\wg_{\lambda}])$ satisfies $\lambda \gesl \ul_0$ and conformal manifold $(S^1 \times S^{n - 1}, [\wg])$ satisfies $\|\wg  - \wg_{\lambda}\|_{C^{k,\alpha}(S^1 \times S^{n - 1})} < \delta$, then any Poincar\'e-Einstein filling
 of $(X^{n+1}, g)$ of $(S^1 \times S^{n - 1}, [\wg])$ 
  satisfies 
	\begin{enumerate}
		\item $X^{n+1}$ is diffeomorphic to $S^1 \times D^n$.
		\item $\sup\limits_{S^1\times D^n}|W_g| \lesl \epsilon_0$ and $g$ is negatively curved.
	\end{enumerate}

	We will show that the dimensional constant 
	\begin{align}\Lambda_0 \equiv \max\{\lambda_0 , \ul_0\}\end{align} is the one satisfying the property in the statement of the theorem, 
	where $\lambda_0$ is the dimensional constant in the statement of  Theorem \ref{t:rigidity-lambda_0}.
	 We will prove this by contradiction. Suppose that for some large positive number $\lambda \gesl \Lambda_0$, no such $\delta = \delta(n, \lambda)$ as in the statement of theorem exists. That is, one can find a sequence of positive numbers $\delta_j \to 0$, and a sequence of smooth metrics $\wg_j$ on $S^1 \times S^{n - 1}$, such that for any $j\in\dZ_+$, the following holds:	
	\begin{enumerate}
	\item $\|\wg_j - \wg_{\lambda}\|_{C^{2,\alpha}(S^1 \times S^{n - 1})} < \delta_j \to 0$, and $\wg_j$ is not isometric to $\wg_{\lambda_j}$; 	
	\item $(S^1\times S^{n-1}, [\wg_j])$ admits two non-isometric Poincar\'e-Einstein fillings $(X_j^{n+1}, g_j)$ and $(Y_j^{n + 1}, \check{g}_j)$. 	

	\end{enumerate}
Under the contradiction assumption, by Theorem \ref{t:product-top-geo-rigidity} and the arguments in the first paragraph,  one obtains the following properties: 
\begin{enumerate}
\item 	The metrics $g_j$ and $\chg_j$ are both negatively curved.  
\item  Both $X_j^{n+1}$ and $Y_j^{n + 1}$ are diffeomorphic to $S^1 \times D^n$,  which we denote by $Q$.
\end{enumerate}

We now denote by $\bg_j^*$ and $\bar{\chg}_j^*$ the Fefferman-Graham compactifications of $g_j$ and $\chg_j$, respectively, as defined in Lemma \ref{l:FG-metric}. 
 By Proposition \ref{p:compactness-small-weyl}, both
 $\{(\overline{Q}, \bg_j^*)\}$ and $\{(\overline{Q}, \bar{\chg}_j^*)\}$ are $C^{2,\alpha}$-compact in the Cheeger-Gromov sense. Consequently, there  there exist two diffeomorphisms $\varphi_j,\check{\varphi}_j \in \Diff(\overline{Q})$ such that passing to subsequences, 
 \begin{align}
 \|\varphi_j^* \bg_j^*  - \bg_{\infty}^* \|_{C^{2,\alpha}(\overline{Q})} \to 0    \quad \text{and} \quad    \|\check{\varphi}_j^* \bar{\chg}_j  - \bar{\chg}_{\infty} \|_{C^{2,\alpha}(\overline{Q})} \to 0, 	
 \end{align}
 	where $\bg_{\infty}^*$ and  $\bar{\chg}_{\infty}^*$ 
 are the Fefferman-Graham compactifications 	
 	of two Poincar\'e-Einstein metrics $g_{\infty}$ and $\chg_{\infty}$ on $Q$ with conformal infinity    $(S^1 \times S^{n - 1}, [\wg_{\lambda}])$ for some $\lambda \gesl \Lambda_0$ in the contradiction assumption. 

Next, by the rigidity result of Theorem \ref{t:rigidity-lambda_0}, one concludes that the two Poincar\'e-Einstein metrics  $g_{\infty}$ and $\chg_{\infty}$  are both isometric to the hyperbolic metric $g_{\lambda}$ on $Q$. Furthermore, applying the same arguments as in \cite[theorem 4.4]{CGQ}, we have the improved weighted convergence of the two sequences of Poincar\'e-Einstein metrics:
\begin{align}
	\Dw (g_j, g_{\lambda}) \to 0    \quad \text{and} \quad \Dw (\chg_j, g_{\lambda}) \to 0,
\end{align}
for some $\mu \in (0,1)$,
where, for two smooth metrics $h$ and $\check{h}$ on $Q$, their weighted distance $\Dw(h, \check{h})$ is defined by
\begin{align}
\Dw(h, \check{h}) \equiv	\|\varphi^* h - \check{h}\|_{C_{\mu}^{2,\alpha}(Q)},
\end{align}
for some diffeomorphism $\varphi\in \Diff(Q)$. Therefore, by the triangle inequality, we have that \begin{align}
 	\Dw(g_j, \chg_j) \to 0. 
 \end{align}

 The above weighted estimate enables us to proceed with deformation arguments and apply the implicit function theorem to ultimately conclude that $g_j$ 
and $\chg_j$ are isometric to each other. 
This part follows the same arguments as in the proofs of \cite[theorem 1.9]{CGQ} and \cite[theorem 1.3]{CGJQ}; therefore we only sketch main steps here. 

Let $\mathcal{M}et(Q)$ be the space of smooth Riemannian metrics on $Q$ and let $\mathscr{S}^2(Q)$ be the space of symmetric $(0,2)$-tensors on $Q$.
We consider the nonlinear functional 
\begin{align}
	\sF(\cdot, \cdot): \mathcal{M}et(Q) \times \mathcal{M}et(Q) \to \mathscr{S}^2(Q), \quad 	\sF(g, h) \equiv \Ric_g + n\cdot  g - \delta_g (B_h(g)),
\end{align}
 where $B_h(g)\equiv \delta_h g + \frac{1}{2} d\Tr_h(g)$. In our context, since $\Ric_{g_j} = - n  \cdot g_j$, obviously
\begin{align}\sF(g_j, g_j) = 0 \label{e:origin}.\end{align} Using the same deformation arguments as in Step 1 of the proof of  \cite[theorem 1.9]{CGQ}, for any large $j$, one can find a diffeomorphism $\Phi_j: \overline{Q} \to \overline{Q}$ with $\Phi_j\left.\right|_{S^1\times S^{n - 1}} = \Id_{S^1 \times S^{n - 1}}$ that solves the equation 
\begin{align}
\sF(\Phi_j^* \chg_j, g_j) = 0   \label{e:deform-chg}	
 \end{align}
and satisfies the following estimates:
\begin{align}
\|\Phi_j - \Id_{\overline{Q}}\|_{C^{3,\alpha}(\overline{Q})}\to 0\quad \text{and}\quad 	\|\Phi_j^* g_j - \chg_j \|_{C_{\mu}^{2,\alpha}(\overline{Q})} \to 0  \quad \text{as} \ j \to \infty,
\end{align}
for some $\mu \in (0,1)$.
On the other hand, one can show that $\sF(\cdot, g_j):\mathcal{M}et(Q) \to \mathscr{S}^2(Q)$ is a local diffeomorphism with respect to the weighted norm defined above. Indeed,
for any Poincar\'e-Einstein metric $h\in \mathcal{M}et(Q)$ , let $\mathscr{L}_h^{(1)}$ be the linearization of $\sF(\cdot, h)$ with respect to the first component. By straightforward computations, we have that 
$\sL\left.\right|_{(h , h)} = \frac{1}{2}(\Delta_h + 2n)$. By Under the contradiction assumption, 
the operator 
\begin{align}\sL_{g_j}\left.\right|_{(g_j , g_j)}:C_{\mu}^{2,\alpha}(Q, \mathscr{S}^2(Q)) \longrightarrow C_{\mu}^{0,\alpha}(Q, \mathscr{S}^2(Q))\end{align}
is an isomorphism; see \cite[theorems A and C]{Lee-fredholm}. 
Finally, by combining \eqref{e:origin} and \eqref{e:deform-chg} and applying the implicit function theorem, we conclude that $\Phi_j^*\chg_j = g_j$. This yields the desired contradiction, and thus completes the proof of the theorem.\end{proof}

\section{Examples}
\label{s:examples}

In this section, we will discuss more examples of Poincar\'e-Einstein manifolds whose conformal infinity is a cylinder $S^1 \times S^{n - 1}$ or a torus $\dT^n$.

\subsection{Hyperbolic isometries and hyperbolic fillings}\label{ss:parabolic}
In this subsection, we will review some relevant examples of Kleinian groups and the corresponding geometry of Poincar\'e-Einstein spaces.

Let us recall that,  in all the theorems of this paper, we have always assumed that 
the conformal infinity is $(S^1 \times S^{n - 1}, [\wg_{\lambda}])$ has dimension at least $3$, where $\wg_{\lambda}$ is the standard product metric on $S^1 \times S^{n - 1}$, defined in
\eqref{e:product-metric-on-cylinder}.
It is also clear that, for $n\gesl 3$, the moduli space of the standard product metrics on $S^1 \times S^{n - 1}$, modulo metric rescalings, can be naturally realized as a one-parameter family $\wg_{\lambda}$ with $\lambda\in(0,\infty)$. 
In Theorem \ref{t:uniqueness-nonpositively-curved}, we have proved that any conformal manifold $(S^1 \times S^{n - 1}, [\wg_{\lambda}])$ bounds a unique Poincar\'e-Einstein filling $(S^1\times D^n, g_{\lambda})$, provided that the filling is non-positively curved, where
\begin{align}
	g_{\lambda} = dr^2 + \sinh^2(r) g_c + \cosh^2(r) (\lambda \tau)^2
\end{align} is the hyperbolic metric on $S^1 \times D^n$,  uniquely determined by the parameter $\lambda$, and $(S^1, \tau^2)$ denotes the unit circle. 

In this case, the fundamental group 
$\Gamma\equiv\pi_1(S^1 \times S^{n - 1})\cong \dZ$ can be realized by the Kleinian group in $
\Mob(S^n)$, the group of conformal transformations on $S^n$,  which shares the same identity component as $\Isom(\mathbb{H}^{n+1})$. 
As a discrete subgroup of $\Isom(\mathbb{H}^{n+1})$, $\Gamma$ is generated by a {\it hyperbolic element}.

There is another type of geometry in which the Kleinian group $\Gamma$ is generated by {\it parabolic elements} in $\Isom(\mathbb{H}^{n + 1})$. Let $\Gamma \cong \dZ^n$ be 
 a rank-$n$ free abelian group 
 generated by $n$ parabolic elements. Geometrically, $\Gamma$ is the fundamental group of a flat torus $(\dT^n, g_{\omega})$.
In the special case of $n = 2$,  $(\dT^2, g_{\omega})$ represents a fixed flat torus which is determined by the lattice $\Gamma\equiv \LB  1,\omega \RB\cong \dZ \oplus \dZ$, where $\omega$ is a fixed complex number in the upper half-plane $\dR \times \dR_+$.

  Since the hyperbolic metric $h$ on $\dH^{n + 1}$ can be written explicitly as:
\begin{align}
	h = dt^2 + e^{2t} g_{\dR^n},\quad t\in(-\infty, +\infty),
\end{align}
 the quotient $(\dH^{n + 1}, h)/\Gamma$ is isometric to the hyperbolic manifold $(\dR \times \dT^n , h_{\omega})$, where 
 \begin{align}
 	h_{\omega} = dt^2 + e^{2t} g_{\omega},\quad t\in(-\infty, +\infty),
 \end{align}
and $\dR\times \dT^n$ has a cusp end as $t\to -\infty$. However,  this does not appear in standard Poincar\'e-Einstein geometry. 

We now describe  hyperbolic metrics on $S^1\times D^2$ whose conformal infinities are given by the flat tori $(\dT^2, [g_{\omega}])$.
    To begin with, we will consider a simple example.
The hyperbolic metric $h$ on $\dH^3$ can be written explicitly as:
\begin{align}
	h = dr^2 + \sinh^2(r) d\theta^2 + \cosh^2(r) dt^2,\quad \theta\in[0,2\pi], t\in \dR, 
	\label{e:hyperbolic-metric-3D}
\end{align}
where $(S^1,d\theta^2)$ is the unit circle. 
Given any $\lambda > 0$, one can choose a free isometric $\dZ$-action on $(\dH^3, h)$ which is generated by 
\begin{align}
\alpha 	: (r, \theta, t) \mapsto (r, \theta, t + 2\pi\lambda).
\end{align}
 Then the quotient metric becomes 
\begin{align}
	h_{\lambda} = dr^2 + \sinh^2(r) d\theta^2 + \cosh^2(r) (\lambda \tau)^2,	
\end{align} which lives on the hyperbolic manifold $D^2 \times S^1$. Moreover, in this case, it is easy to check that the conformal infinity of $(D^2 \times S^1, h_{\lambda})$ is given by $(\dT^2, [d\theta^2 + (\lambda \tau)^2])$.

In general, it was shown in \cite{Anderson-cusp} that there exist infinitely many hyperbolic fillings of a fixed conformal infinity given by a flat torus $(\dT^2, [g_{\omega}])$. We now briefly describe the geometric picture of this case.
Let $\sigma_j$ be a sequence of simple closed geodesics such that $\sigma_j$ is not dense in $(\dT^2, g_{\omega})$ and $|\sigma_j|\to \infty$.  
Then there exists a sequence of hyperbolic metrics $h_j$ on $S^1 \times D^2$ whose conformal infinity is given by  $(\dT^2, [g_{\omega}])$. However, $\Injrad_{h_j}(S^1 \times D^2) \to 0$. Therefore, $\{(S^1 \times D^2), h_j\}_{j\in \dZ_+}$ is a non-isometric 
family of hyperbolic manifolds. By the construction, one also observes that the sequence 
$\{(S^1 \times D^2, h_j, \bm{x}_j)\}_{j\in \dZ_+}$
 converges to the hyperbolic manifold $(\dR \times \dT^2, h_{\omega}, \bm{x}_{\infty})$ in the pointed Cheeger-Gromov sense as $j\to \infty$, where $\bx_j$ is chosen such that $\Injrad_{h_j}(\bx_j) = 1$.
Notice that the limiting hyperbolic manifold has a cusp end since $\Injrad_{h_j}(S^1 \times D^2) \to 0$ as $j\to \infty$.

  \subsection{AdS-Schwarzschild metric}
 \label{ss:AdS-S}

The AdS-Schwarzschild space is a well-known example of Poincar\'e-Einstein manifold. In \cite{Hawking-Page}, the authors provided the first known example of non-uniqueness of  Poincar\'e-Einstein filling for a fixed conformal infinity.     In four dimensions, the conformal infinity of the AdS-Schwarzschild space is $S^1 \times S^2$, equipped with the standard product metric. In this case, the underlying topological manifold of the AdS-Schwarzschild space is diffeomorphic to $D^2 \times S^2$. In general, 
for a fixed integer $n\gesl 3$ and a fixed number $m > 0$, the $(n + 1)$-dimensional AdS-Schwarzschild metric is defined by
\begin{align}
	g_m = (V_m)^{-1} ds^2 + V_m dt^2 + s^2 g_{S^{n - 1}}, \quad s \in [s_h , +\infty),\ t \in [0, 2\pi\lambda], 
\end{align}
where 
\begin{align}
V_m = 1 + s^2 - \frac{\varpi_n \cdot m}{s^{n - 2}},	
\end{align}
and $s_h > 0$ is the positive root of $V_m$, referred to as the horizon. The parameter $m$ here is precisely the mass of $g_m$. For fixed $\lambda > 0$, as $s\to +\infty$, the Schwarzschild metric $g_m$ is asymptotic to the hyperbolic metric $g_{\lambda}$, so that $(S^1 \times D^n, g_{\lambda})$ can regarded as the asymptotic model of $(D^2 \times S^{n - 1}, g_m)$. 
In the special case of $n = 3$, $V_m$ is reduced to 
\begin{align}
V_m = 1 + s^2 - \frac{2m}{s^{n - 2}}.
\end{align}
See \cite{Hawking-Page, Witten-AdS} for details.

In the following, we will review some basic geometric quantities and their behavior on the $4$-dimensional AdS-Schwarzschild space $D^2 \times S^2$, and explain how these relate to the main results of the current paper. 

\medskip

\noindent {\bf The length parameter $\lambda$ and  non-unique filling.} First, note that the parameter $\lambda$ is uniquely determined if $g_m$ is smooth 
at $s = s_h$. In fact, the smoothness of the metric $g_m$ at $s = s_h$ implies that the tangent space at $s = s_h$ is isometric to $\dR^4$. Let $dr = (V_m)^{-\frac{1}{2}}ds$ such that $r = 0$ when $s = s_h$. It is necessary that  
\begin{align}
	2\pi\lambda \cdot \left.\frac{d V_m}{dr}\right|_{r = 0} = 2\pi,
\end{align}
which implies that 
\begin{align}
\lambda^{-1} = \frac{1}{2} V_m'(s_h) = \frac{1}{2}\left(3s_h + \frac{1}{s_h} \right). \label{e:lambda-horizon}
\end{align}
It then follows that the length parameter $\lambda$ is contained in the range $(0, \frac{1}{\sqrt{3}}]$. Moreover,  any  $\lambda \in (0, \frac{1}{\sqrt{3}})$
corresponds to two distinct values of 
$s_h$, and hence two distinct values of the mass. Therefore, for any $\lambda\in(0, \frac{1}{\sqrt{3}})$, including the hyperbolic manifold $(S^1 \times D^3, g_{\lambda})$, there exist {\it at least three} distinct (non-isometric) Poincar\'e-Einstein fillings which are not isometric to each other. 

\medskip

\noindent {\bf Renormalized volume and Weyl energy.}
Next, we consider the
renormalized volume and the Weyl energy on the AdS-Scharzschild space. 
It is known that the renormalized volume has an explicit formula in terms of $s_h$:
\begin{align}
	\mathcal{V}(D^2\times S^2, g_m) = \frac{8\pi^2}{3} \cdot \frac{s_h^2 (1 - s_h^2)}{3 s_h^2 + 1}.\label{e:renormalized-volumed-AdS-S}
\end{align}
For example, one can see \cite{CQY}
for detailed computations. In \eqref{e:lambda-horizon}, for any parameter   $\lambda \in (0,\frac{1}{\sqrt{3}})$, there are two positive numbers $s_+$ and $s_-$ that satisfy \eqref{e:lambda-horizon}, where 
\begin{align}
\lim\limits_{\lambda \to 0}s_+(\lambda) = 0
	\quad \text{and} \quad \lim\limits_{\lambda \to 0}s_-(\lambda) = +\infty.
\end{align}
We denote by $m_+$ and $m_-$ the corresponding masses.
Plugging this asymptotics into \eqref{e:renormalized-volumed-AdS-S} and letting $\lambda \to 0$, we find:
\begin{align}
	\mathcal{V}(D^2\times S^2, g_{m_+, \lambda}) \to 0 \quad \text{and}\quad 	\mathcal{V}(D^2\times S^2, g_{m_-, \lambda}) \to -\infty. \end{align}
Furthermore, the renormalized volume $\mathcal{V}(D^2\times S^2, g_m) $ achieves its maximum $\frac{8\pi^2}{27}$ 
	when $s_h =  \frac{1}{\sqrt{3}}$.
By the Chern-Gau{\ss}-Bonnet theorem (Theorem \ref{t:Chern-Gauss-Bonnet}),  the $L^2$-energy of the Weyl curvature satisfies 
\begin{align}
	\int_{D^2\times S^2}|W_{g_{m_{\pm},\lambda}}|^2 \dvol_{g_{m_{\pm},\lambda}} = 16\pi^2 - 6\mathcal{V}(D^2\times S^2, g_{m_{\pm},\lambda}) \gesl \left(16 - \frac{8}{27}\right)\pi^2, \label{e:lower-bound-of-L2-Weyl}
\end{align}
which provides a {\it uniform lower bound}. 

\medskip

The key implications are as follows. For $\lambda \in (0, \frac{1}{\sqrt{3}})$,
we obtain two families of AdS-Schwarzschild metrics $\{g_{m_+,\lambda}\}$ and 
$\{g_{m_-,\lambda}\}$. Here for the family $\{g_{m_+,\lambda}\}$,
the $L^2$-norm of the Weyl curvature  is {\it uniformly bounded} as $\lambda \to 0$, 
and for the family $\{g_{m_-,\lambda}\}$,
the $L^2$-norm of the Weyl curvature is {\it unbounded}. In both of the two cases, the $L^2$-norm of the Weyl curvature is uniformly bounded from below. This uniform lower bound shows that the pinching condition of the Weyl energy in Theorem \ref{t:rigidity-for-small-weyl-energy} will be never satisfied for Ads-Schwarzschild metrics, but is essential for the hyperbolic rigidity.

\section{Discussions and Questions}
\label{s:discussions}

This section will discuss some open questions and several future directions related to the results of the paper. 

\subsection{Rigidity for fixed topology and sharp pinching}
The main results of this paper study the rigidity or quantitative rigidity of the Poincar\'e-Einstein filling of $S^1 \times S^{n-1}$. The example of the AdS-Schwarzschild metric in Section \ref{ss:AdS-S} 
gives complete Poincar\'e-Einstein metrics on $X^4 \approx S^2 \times D^2$ 
whose conformal infinities are the standard product metrics on $S^2 \times S^1$. Obviously, the AdS-Schwarzschild space has a distinct topology compared to the hyperbolic manifold $S^1 \times D^3$. Therefore, it is natural to ask the following question. 
\begin{question}\label{q:topology-fixes-geometry}
Let $(S^1 \times S^{n - 1}, \wg_{\lambda})$ be a standard Riemannian product. If the conformal manifold $(S^1 \times S^{n - 1}, [\wg_{\lambda}])$ bounds a Poincar\'e-Einstein space $(X^n, g)$ with $X^n$ diffeomorphic to $S^1 \times D^n$, is it necessary that $g$ is isometric to the hyperbolic metric $g_{\lambda}$? 
\end{question}

Note that under the assumption $\sec_g \lesl 0$, as observed in Section \ref{ss:deck-transformation}, it is obvious that $X^{n + 1}$ is diffeomorphic to $S^1 \times D^n$. 
So the key point in Question \ref{q:topology-fixes-geometry} is how refined geometric information can be read out from the topological inputs.

The next question is more intriguing and challenging. 
Recall that, for any AdS-Scharzschild metric on $S^2 \times D^2$,  the corresponding length parameter $\lambda$ of the conformal infinity $(S^1 \times S^2, g_{\lambda})$ satisfies 
\begin{align}
	\lambda\in \left(0,\frac{1}{\sqrt{3}}\right].
\end{align}
On the other hand, in Theorem \ref{t:large-circle}, the Poincar\'e-Einstein filling must be hyperbolic if the length parameter $\lambda$ is sufficiently large. 
\begin{question}
Let $(S^1 \times S^{n - 1}, \wg_{\lambda})$ be a standard Riemannian product. What is the best lower bound for $\lambda$ such that the Poincar\'e-Einstein filling is hyperbolic? Is it equal to $\frac{1}{\sqrt{3}}$?
\end{question}

\subsection{Weak energy pinching in four dimensions}

Recall that Theorem \ref{t:rigidity-for-small-weyl-energy} gives the hyperbolic rigidity under the {\it strong} Weyl energy pinching, where the $L^2$-Weyl curvature is sufficiently small relative to the Yamabe constant of the conformal infinity. So the smallness of the $L^2$-Weyl curvature depends on the {\it collapse} of the space. It is natural to ask whether one can find a uniform pinching energy that is independent of collapse.   
\begin{question}
	\label{q:4D-pinching}
	Let $(S^1 \times S^2, \wg_{\lambda})$ be the standard Riemannian product for some given $\lambda  > 0$. Does there exist an absolute constant $\epsilon > 0$ (independent of $\lambda$) such that the pinching of the $L^2$-Weyl curvature 
	\begin{align}
		\int_{X^4}|W_g|^2 \dvol_g < \epsilon 
	\end{align}
	implies that the Poincar\'e-Einstein filling $(X^4,g)$ of $(S^1 \times S^2, [\wg_{\lambda}])$ is isometric to the hyperbolic manifold $(S^1 \times D^3, g_{\lambda})$?

Recall that in the AdS-Schwarzschild metrics on $S^2\times D^2$
has a uniform lower bound of the Weyl energy.
\end{question}

Technically, the proof of Theorem \ref{t:rigidity-for-small-weyl-energy} is based on the $\epsilon$-regularity for the normalized curvature energy (Lemma \ref{l:standard-eps-reg}), where the pinching energy depends on the collapse. The pioneering result that eliminates such a constraint is  
Cheeger-Tian's $\epsilon$-regularity for Einstein $4$-manifolds \cite[theorem 0.8]{CT}, which we now state here.

\begin{theorem}
	Let $(M^4 , g)$ be an Einstein manifold with $|\Ric_g| \lesl 3$ such that $\overline{B_2(p)}$ is compact in $B_4(p)$. Then there exist {\it absolute constants} $\epsilon > 0$ and $C > 0$ such that if for some $r\in(0,1]$,
\begin{align}
	\int_{B_{2r}(p)} |\Rm_g|^2 < \epsilon,
\end{align}
then 
\begin{align}
	\sup\limits_{B_r(p)}|\Rm_g|\lesl Cr^{-2}.
\end{align}
\end{theorem}
Under the pinching assumption of Question \ref{q:4D-pinching}, applying Cheeger-Tian's $\epsilon$-regularity, immediately one obtains the global uniform $C^0$-estimate of the full curvature $\Rm_g$. The key point in Question \ref{q:4D-pinching} whether the Weyl curvature is sufficiently small everywhere on $X^4$.

\subsection{Finiteness of Poincar\'e-Einstein fillings}

The next question is concerned with the topological finiteness of Poincar\'e-Einstein fillings.
\begin{question}
Let  $(S^1 \times S^{n - 1}, \wg_{\lambda})$ be a Riemannian product with $\lambda \gesl \lambda_0$. For any 
$\lambda_0 > 0$, does there exist a number $N = N(n, \lambda_0) > 0$ such that all the Poincar\'e-Einstein fillings of $(S^1 \times S^{n - 1}, [\wg_{\lambda}])$ have at most $N$ diffeomorphism types? 
\end{question}

\begin{question}\label{q:diffeomorphism-finiteness}
Let $Q > 0$, $D > 0$, and $\Upsilon > 0$  be fixed positive constants. If $(M^3, \wg)$ is a closed manifold that has uniformly bounded geometry in the sense that 
\begin{align}\|\Rm_{\wg}\|_{C^0(M^3)}\lesl Q, \quad \diam_{\wg}(M^3)\lesl D,\quad \text{and} \ \mathcal{Y}(M^3, [\wg])\gesl \Upsilon > 0,
\end{align}	
  does there exist a uniform constant $N = N(Q, D,\Upsilon) > 0$
  such that all the Poincar\'e-Einstein fillings of $(M^3, [\wg])$ have at most $N$ diffeomorphism types? 
 \end{question}
 
This question is motivated by a diffeomorphism finiteness result for $4$-dimensional complete Ricci-flat manifolds with Euclidean volume growth. The theorem below is essentially due to Cheeger-Naber (\cite{Cheeger-Naber}). 
 \begin{theorem}\label{t:ALE-diffeomorphism-finiteness} Given $v > 0$, there exists some uniform constant $N = N(v) > 0$ such that the following holds.
 If $(X^4, g)$ be a complete Ricci-flat manifold with $p\in X^4$ that satisfies 
 \begin{align}
 	\frac{\Vol_g(B_R(p))}{R^4}\gesl v > 0,\quad \text{for all}\  v >0, \label{e:Euclidean-volume-growth}
 \end{align} 
 then $X^4$ is ALE (asymptotically locally Euclidean) of order $4$ and diffeomorphic to one of at most $N$ complete manifolds.
  \end{theorem}

  To better understand the setting of Question \ref{q:diffeomorphism-finiteness}, let us sketch the proof of Theorem \ref{t:ALE-diffeomorphism-finiteness}, which follows from analyzing the diffeomorphism types of the unbounded region and the bounded region.

The Euclidean volume growth condition \eqref{e:Euclidean-volume-growth} implies that the Ricci-flat manifold $(X^4,g)$ satisfies the $L^2$-curvature estimate 
\begin{align}
	\int_{X^4}|\Rm_g|^2 \dvol_g \lesl \Lambda(v)	
\end{align}
for some uniform constant $\Lambda = \Lambda(v) > 0$. In fact, this estimate follows 
 from \cite[theorem 1.13]{Cheeger-Naber} and a simple rescaling argument.
 Then applying the main result of \cite{BKN},
one can conclude that such a Ricci-flat manifold must be an ALE space of order $4$. Therefore, there exist $R_{\#} > 0$ (depending on $X^4$), $\omega =\omega(v) > 0$, and a diffeomorphism 
\begin{align}
	\Phi: A_{R_{\#}, \infty} (\bo_*) \longrightarrow X_{R_{\#}}^4,
\end{align}
 where $X_{R_{\#}}^4\equiv X^4 \setminus B_{R_{\#}}(p)$, $A_{R_{\#}, \infty} (\bo_*)$ is annulus in the flat cone $\dR^4/\Gamma$ for some finite group $\Gamma\lesl O(4)$ with $|\Gamma|\lesl \omega$. Immediately, $A_{R_{\#}, \infty} (\bo_*)$ is diffeomorphism to $(S^3/\Gamma) \times [R_{\#}, \infty)$. It is well known that $A_{R_{\#}, \infty} (\bo_*)$ belongs to bounded number $N_{\#} = N_{\#}(v)$ of diffeomorphism types, so does $X_{R_{\#}}^4$.
 
For the bounded region of $X^4$, under the Euclidean growth assumption \eqref{e:Euclidean-volume-growth}, one can apply the diffeomorphism finiteness result  \cite[theorem 1.12]{Cheeger-Naber} so that there exists a uniform constant 
$N_0 = N_0 (v) > 0$ such that for any $R > 0$, $B_R(p)$
 is diffeomorphic to at most one of $N_0$ manifolds.
Finally, Theorem \ref{t:ALE-diffeomorphism-finiteness} follows when we choose $R \gesl 2R_{\#}$.

A plausible strategy in proving the diffeomorphism finiteness in Question \ref{q:diffeomorphism-finiteness}  
is to establish some $L^2$-estimate for the Weyl curvature and construct some model geometry on a large annulus of $X^4$.

\subsection{Poincar\'e-Einstein fillings of Kleinian manifolds}

There are a large variety of geometric models of Poincar\'e-Einstein manifolds  whose conformal infinities are Kleinian manifolds.
In this context, it would be also interesting to consider case when the conformal infinity has a {\it large} limit set. Let $(M^n, \wg)$ be a compact Kleinian manifold with $\Gamma \equiv \pi_1(M^n)$ such that the universal cover $\widetilde{M^n}$ is conformally embedded into $S^n$ and the limit set $\Lambda(\Gamma)$ of $\Gamma$ in $S^n$ has a positive Hausdorff dimension. It is natural to study the Poincar\'e-Einstein fillings in this case, in which more model geometries will be included.

\bibliographystyle{amsalpha}

\bibliography{CYZ}

 \end{document}